\documentclass[reqno,11pt]{amsart}
\usepackage{a4wide}
\usepackage[english]{babel}
\usepackage{amssymb}
\usepackage[foot]{amsaddr}

\usepackage{graphicx}
\usepackage{wrapfig}
\usepackage{enumerate}
\usepackage[usenames,dvipsnames]{color}
\usepackage{mathtools}
\mathtoolsset{showonlyrefs,showmanualtags}
\usepackage{esint}

\usepackage{tikz}
\usetikzlibrary{arrows,hobby}
\usetikzlibrary{patterns, patterns.meta}
\usepackage{xcolor} 

\usepackage[mathscr]{eucal}

\allowdisplaybreaks

\newtheorem{thm}{Theorem}[section]

\newtheorem{cor}[thm]{Corollary}
\newtheorem{defn}[thm]{Definition}
\newtheorem{lem}[thm]{Lemma}
\newtheorem{prop}[thm]{Proposition}

\newtheorem*{conjecture*}{Conjecture}

\newtheorem{innercustomgeneric}{\customgenericname}
\providecommand{\customgenericname}{}
\newcommand{\newcustomtheorem}[2]{%
	\newenvironment{#1}[1]
	{%
		\renewcommand\customgenericname{#2}%
		\renewcommand\theinnercustomgeneric{##1}%
		\innercustomgeneric
	}
	{\endinnercustomgeneric}
}
\newcustomtheorem{customthm}{Theorem}

\theoremstyle{definition}

\theoremstyle{remark}
\newtheorem{rem}[thm]{Remark}

\numberwithin{equation}{section}

\newcommand{\DeclareAutoPairedDelimiter}[3]{%
	\expandafter\DeclarePairedDelimiter\csname Auto\string#1\endcsname{#2}{#3}%
	\begingroup\edef\x{\endgroup
		\noexpand\DeclareRobustCommand{\noexpand#1}{%
			\expandafter\noexpand\csname Auto\string#1\endcsname*}}%
	\x}

\DeclareAutoPairedDelimiter{\abs}{\lvert}{\rvert}
\DeclareAutoPairedDelimiter{\norm}{\lVert}{\rVert}
\DeclareAutoPairedDelimiter{\bra}{(}{ )}
\DeclareAutoPairedDelimiter{\pra}{[}{]}
\DeclareAutoPairedDelimiter{\set}{\{}{\}}
\DeclareAutoPairedDelimiter{\skp}{\langle}{\rangle}

\DeclareMathAlphabet{\mathup}{OT1}{\familydefault}{m}{n}
\newcommand{\dx}[1]{\mathop{}\!\mathup{d} #1}

\DeclareMathOperator{\loc}{loc}

\DeclareMathOperator{\tr}{tr}

\newcommand{\N}{\mathbb{N}}

\newcommand{\R}{\mathbb{R}}

\newcommand{\cA}{\ensuremath{\mathcal A}}

\newcommand{\cC}{\ensuremath{\mathcal C}}

\newcommand{\cH}{\ensuremath{\mathcal H}}

\newcommand{\cR}{\ensuremath{\mathcal R}}

\newcommand{\bV}{\ensuremath{V}}

\usepackage[colorinlistoftodos,prependcaption,textsize=tiny,linecolor=red,backgroundcolor=red!25,bordercolor=red]{todonotes}

\usepackage[pdfusetitle]{hyperref}
\definecolor{darkblue}{rgb}{0,0,0.6}
\hypersetup{
    colorlinks=true,       % false: boxed links; true: colored links
    linkcolor=darkblue,          % color of internal links (red)
    citecolor=darkblue,        % color of links to bibliography (green)
    filecolor=darkblue,      % color of file links (magenta)
    urlcolor=darkblue           % color of external links (cyan)
}

\title[Complete classification of global solutions
 to the obstacle problem]{Complete classification of global solutions\\ to the obstacle problem}

\author{Simon Eberle$^1$}
\address{$^1$Basque Center of Applied Mathematics, Spain}
\email{seberle@bcamath.org}
\author{Alessio Figalli$^2$}
\address{$^2$Department of Mathematics, ETH Zürich, Switzerland}
\email{alessio.figalli@math.ethz.ch}
\author{Georg S. Weiss$^3$}
\address{$^3$Faculty of Mathematics, University of Duisburg-Essen, Germany}
\email{georg.weiss@uni-due.de}

\let\rho\varrho
\let\phi\varphi
\let\epsilon\varepsilon

\begin{document}

\begin{abstract}
The characterization of global solutions to the obstacle problems in $\R^N$, or equivalently of null quadrature domains, has been studied over more than 90 years. In this paper we give a conclusive answer to this problem by proving the following 
long-standing conjecture:
{\it The coincidence set of a global solution to the obstacle problem is either a half-space, an ellipsoid, a paraboloid, or a cylinder with an
ellipsoid or a paraboloid as base.}
\end{abstract}
		\maketitle

\tableofcontents

\section{Introduction}
\subsection{Null quadrature domains and the obstacle problem}
In 1678,  Newton stated his famous \emph{no gravity in the cavity theorem}: spherical shells do not exert gravitational force inside the cavity of the shell. This result was later extended to ellipsoidal shells (homoeoid) first by Laplace, and soon after by Ivory using a more geometric approach.

In modern terms, these results can be stated in terms of null quadrature domains.
We recall that an open set $\Omega\subset \R^N$ is called a {\it null quadrature domain} if
$$
\int_\Omega h\,dx=0
$$
for every harmonic function $h \in L^1(\Omega)\cap C^0(\overline \Omega)$.
With this terminology, the results of Newton, Laplace, and Ivory can be stated saying the complement of a ball/elliposoid is a null quadrature domain.
In greater generality, one can prove that null quadrature domains include:\\
- half-spaces;\\
- exteriors of ellipsoids;\\
- exterior of paraboloids;\\
-  cylinders over domains of the types listed above.\\
A major question, which has been investigated over the last 90 years, is to understand whether this list includes all the possible null quadrature domains.
Before discussing it, it is important to point out that null quadrature domains are rigorously related to solutions to the obstacle problem.
More precisely, as discussed for instance in \cite[Theorem II]{CaffarelliKarpShahgolian_Annals_2000} and \cite[Theorem 4.1]{KarpMargulis_free_boundaries}
{
\begin{multline*}
\Omega \text{ is a null quadrature domain }\quad \Leftrightarrow \quad \\
\text{ $\Omega=\{u>0\}$ for some non-negative solution $u \in C^{1,1}_{\loc}(\R^N)$ of $\Delta u=\chi_{\{u>0\}}$.}
\end{multline*}
}
In other words, characterizing null quadrature domains is equivalent to characterizing the coincidence set $\{u=0\}$ for global solutions to the obstacle problem
\begin{equation}\label{eq:obstacle-global}
\Delta u= \chi_{\{u>0\}}, \quad u \geq 0, \qquad \hbox{in }\R^{N}.
\end{equation}
It is well-known that global solutions to the obstacle problem are convex (see for instance \cite[Theorem 5.1]{PetrosyanShahgholianUraltseva_book}). In particular, the coincidence set $\{u=0\}$ is convex.

\subsection{Classification results}
The first partial classification of global solutions with {\it compact} coincidence sets has been achieved more than 90 years ago: in 1931, Dives \cite{Dive} showed that, for $N=3$, if $\{u=0\}$ has non-empty interior and is bounded then it is an ellipsoid. Many years later, in 1979, Lewy gave a new proof of this result \cite{Lewy79}.

In 1981, M. Sakai gave a full classification of global solutions in \emph{two dimensions} using complex analysis (cf. \cite{Sakai_Null_quadrature_domains}).

The higher dimensional analogue to Dive's result, i.e., if $\set {u=0}$ is bounded and has non-empty interior then it is an ellipsoid, was proved shortly after in two steps. First, DiBenedetto and Friedman proved the result in 1986 under the additional assumption that $\set {u=0}$ is symmetric with respect to $\set {x_j =0}$ for all $j \in \set {1, \dots, N}$ (cf. \cite{DiBenedettoFriedman}). Then, in the same year, Friedman and Sakai \cite{FriedmanSakai} removed the symmetry assumption.
Very recently, in \cite{ellipsoid}, two of the authors gave a very concise proof of the characterization of compact coincidence sets.

%It is noteworthy that \cite{FriedmanSakai} has a beautiful application to Eshelby's inclusion problem (\cite{Kang-Milton-2008,Liu-2008,KANG_Eshelby_review}).
Hence, while global solutions with compact coincidence sets have been completely classified, the structure of solutions with \emph{unbounded} coincidence sets remained largely open and is related to the following conjecture (here, one is implicitly assuming that $\{u=0\}$ has non-empty interior, as otherwise solutions are trivially classified, see Remark~\ref{rem:global liouville} below):
\begin{center}
	{\bf Conjecture:} {\it The coincidence set of a global solution to the obstacle problem is either a half-space, an ellipsoid, a paraboloid, or a cylinder with an
ellipsoid or a paraboloid as base.}
\end{center}
This conjecture, which has been investigated over more than {30} years, has been {officially} raised in several papers: first 
by Shahgholian in \cite[conjecture on p. 10]{Shahgholian92_conjecture}, then by Karp and Margulis in \cite[Conjecture 4.5]{KarpMargulis_bounded_sources}, and recently in the monograph `Research Problems in Function Theory' \cite[§3.1 pp. 63-64, and Problem 3.28]{Reseach_problems_in_function_theory}.

In the recent work \cite{esw_arXiv} the first and third author, together with Shahgholian, have been able to solve the conjecture in dimension $N \geq 6$ (see Remark~\ref{rem:6} below). 

Here we are able to fully characterize global solutions in the remaining dimensions (actually, our proof gives a complete characterization for all dimensions $N \geq 3$), allowing us to prove the conjecture above. 

\begin{thm}[Complete characterization of global solutions to the obstacle problem]\label{thm:MainTheorem_Intro_I}
	 
	Let $N \geq 2$, and let  $u$ be a   solution of   \eqref{eq:obstacle-global} 
	such that the coincidence set $\{ u=0\}$ has non-empty interior. 
	Then the coincidence set is is either a half-space, an ellipsoid, a paraboloid, or a cylinder with an
ellipsoid or a paraboloid as base.
\end{thm}

As we shall explain later, this result is a rather direct consequence of the classification of monotone solutions. More precisely, the core of this paper is the following:

\begin{thm}[Characterization of monotone solutions] \label{thm:main}
	Let $N\geq 3$ and let $u$ be an $x_N$-monotone solution according to Definition~\ref{def:solution} below. Then $\{u=0\}$ is a paraboloid.
\end{thm}

\begin{rem}
Thanks to \cite[Main Theorem]{esw_singularities}, Theorem \ref{thm:MainTheorem_Intro_I} implies a very refined result on the behavior of the regular part of the free boundary close to singularities.
\end{rem}

\begin{rem}\label{rem:6}
As mentioned before,
Theorem~\ref{thm:main} has already been proved for $N\geq 6$ in \cite{esw_arXiv}. An important reason for this dimensional restriction comes from the fact that, in the proof in \cite{esw_arXiv}, a key role is played by the Newtonian potential associated to $\{u=0\}$, defined (up to a multiplicative constant) as $\frac{1}{|x|^{N-2}}\ast \chi_{\{u=0\}}$.  However, if $\{u=0\}$ is a paraboloid then the above convolution converges only for $N\geq 6$.

As we shall see later, this definition can be ``corrected'' to obtain a convergent expression also in lower dimensions (see Definition~\ref{NP} and Lemma~\ref{lem:existence_of_newton_potential}). However,  the positivity  of the Newtonian potential is important for the arguments in \cite{esw_arXiv}, while our generalized potential loses this property.

At a more ``fundamental'' level, the role of the dimension can be seen as follows: if {$p(x)=\lim_{r\to \infty}\frac{u(rx)}{r^2}$} denotes the blow-down polynomial appearing in Definition~\ref{def:solution}\eqref{PDE_asymptotics}, then the behaviour of $u-p$ changes considerably with the dimension. In particular, if $\{u=0\}$ is a paraboloid (this is a particular case of $x_N$-monotone solution) then one can check by explicit computations that, for $N\geq 4$, there exists a linear function $\ell$ such that:\\
- $\fint_{B_R}|u-p-\ell| \dx{x}$ is bounded for $N\geq 6$;\\
- $\fint_{B_R}|u-p-\ell| \dx{x} \simeq \log(R)$ for $N=5$;\\
- $\fint_{B_R}|u-p-\ell| \dx{x} \simeq \sqrt{R}$ for $N=4$.\\
This different behavior is the reason for the dimensional restriction $N \geq 6$ in \cite{esw_arXiv}.
In this paper, instead, we develop a new approach that only requires
 $\fint_{B_R}|u-p-\ell| \dx{x}=o(R)$, giving a unified proof of Theorem~\ref{thm:main} for $N\geq 4$. Unfortunately,  in the ``critical'' dimension $N=3$, $\fint_{B_R}|u-p| \dx{x} \simeq R\log R$. In particular, there is no affine function that dictates the behaviour at infinity of $u-p$. 
As the reader will see, this fact is a source of major difficulties for proving Theorem~\ref{thm:main} in dimension $N=3$.

We note that also for $N=2$ the behavior of $u-p$ is superlinear: $\fint_{B_R}|u-p| \dx{x} \simeq R^{3/2}$. However, when $N=2$ one can rely on the Riemann mapping theorem to obtain a short proof of Theorem \ref{thm:MainTheorem_Intro_I} (see \cite{Sakai_Null_quadrature_domains}).
\end{rem}

\subsection{Structure of the paper}\label{section:structure}
In Section~\ref{sec:notation}, we begin by introducing some notation and collecting a series of useful preliminary estimates on solutions to the obstacle problem and on the Alt-Caffarelli-Friedman (ACF) functional, that will play a crucial role in our proof.

In Section~\ref{section:Newton_potential_expansion_of_u} we prove that, for an $x_N$-monotone solution $u$, one can define a Newtonian-type potential $V_\cC$ associated to its coincidence set $\cC=\{u=0\}$ so that the expansion $u=p+V_\cC$ holds.

In Section~\ref{sec:improved_estimate_on_asymptotics_of_C}, thanks to the Newtonian expansion proved in the previous section, we prove that the coincidence set $\cC$ is asymptotically contained inside a paraboloid.
The proof of this result is rather easy in dimension $N \geq 4$, while the case $N=3$ requires an extremely delicate argument.

In Section~\ref{sec:BMO-type_estimate}, we use the result from Section~\ref{sec:improved_estimate_on_asymptotics_of_C} to analyze the asymptotic behavior of $V_\cC$. In particular:
 for $N\geq 4$ we can show that $V_\cC$ behaves at infinity like a linear function (up to sublinear corrections);
 for $N=3$, on each large ball $B_R$, $V_\cC$ 
{is at most $C R$-away from} an affine function whose slope behaves like $\log R$. 
In other words, while for $N\geq 4$ the gradient of $V_\cC$ is essentially bounded, for $N=3$ it has a BMO-type behavior (see Lemma \ref{lem:improvement_of_growth_by_subtracting_affine_linear_function}).

In Section~\ref{sec:constructing_the_correct_ellipsoids}, exploiting the information on $V_\cC$ obtained in the previous section, we can construct matching paraboloid solutions (i.e. solutions that have paraboloids as coincidence sets).
More precisely, for $N\geq 4$ we can find a paraboloid solution $u_P$ such that $u-u_P$ grows sublinearly at infinity.
Instead, for $N=3$, for each $R$ we construct a paraboloid solution $u_{P_R}$ such that 
{$\frac{1}{|B_R|}\|u-u_{P_R}\|_{L^1(B_R)}\le CR$.}

With all this preparatory work, we can then prove our main result.

More precisely, in Section~\ref{sect:proof 4} we focus on the case $N\geq 4$. In that case, applying the ACF formula to the difference between $u$ and suitable translations of the paraboloid solution constructed in Section~\ref{sec:constructing_the_correct_ellipsoids} {and exploiting the sublinear growth at infinity}, we are able to prove that such solutions are ordered. Once this is achieved, we conclude easily.

Instead, in Section~\ref{sect:proof 3} we focus on the case $N=3$. In this case, {due to the lack of a sublinear approximation of $u$ via paraboloid solutions,} we cannot directly apply the ACF formula to deduce that $u$ and some suitable paraboloid solution are ordered.
Instead, we apply the ACF formula to the functions $\frac{1}{R}(u-u_{P_R})(R \,
\cdot)$ (each of which may not satisfy a linear growth bound) in order to construct a comparison solution $u_\infty$ whose coincidence set is a paraboloid.
Then, by a delicate ACF-type dichotomy, we show that one-homogeneous blow-down limits of $u-u_{\infty}$ exist and:\\
- either they have constant sign (so $u$ and $u_\infty$ are ``ordered at infinity'');\\
- or they are linear functions.\\
While in the first case we can conclude similarly to the case $N\geq 4$, 
the second case requires a very refined analysis. More precisely, exploiting the information that  $u-u_{\infty}$ behaves as a linear function at infinity, 
we can construct fine adjustments of the paraboloid solution $u_{\infty}$ to
show that, for some suitable translations of $u_\infty$, the ACF energy vanishes.
Then, we conclude similarly as in the first case.

Finally, in Section~\ref{sect:proof thm}, we provide a new self-contained argument showing how Theorem~\ref{thm:MainTheorem_Intro_I} follows from Theorem~\ref{thm:main}.

\medskip
{\it Acknowledgements.} The second author has received funding from the European Research Council under the Grant Agreement No.
721675 ``Regularity and Stability in Partial Differential Equations (RSPDE)''.
%\section*{Acknowledgments}

\section{Notation and preliminaries}\label{sec:notation}

Throughout this work, $\R^N$ will be equipped with the Euclidean inner product $x \cdot y$ and the induced norm $\abs{x}$. Due to the nature of the problem, we will often write $x \in \R^N$ as $x= (x',x_N) \in \R^{N-1} \times \R$. Also, we denote by $(e^i)_{1\leq i \leq N}$ the elements of the canonical base of~$\R^N$.

In our estimates, $C$ denotes a generic positive constant that may change from line to line. We shall use $C_N$ whenever the constant depends only on the dimension.

We write $B_r(x)$ to denote the open $N$-dimensional ball of center $x$ and radius $r$, while $B_r'(x')$ is the open $(N-1)$-dimensional ball of center $x' \in \R^{N-1}$ and radius $r$. Whenever the center is omitted, it is assumed to be the origin $0$.

When considering a set $A$, $\chi_A$ shall denote the characteristic function of $A$. With $\cH^{N-1}$, we refer to $(N-1)$-dimensional Hausdorff measure. If $A$ and $B$ are two sets, we denote their symmetric difference by $A \triangle B := (A \setminus B) \cup (B \setminus A)$. Given a function $f: A \subset \R^N \to \R$,  we define $f_+:= \max\{f,0\}$ and $f_- := \max \{ - f,0\}$. Furthermore, we define the differential operator $\nabla'f := (\partial_1 f, \dots, \partial_{N-1}f)$.

\begin{defn}[Coincidence set]\label{def:coincidence_set}
Given a solution $u$ to the obstacle problem \eqref{eq:obstacle-global}, we define its \emph{coincidence set} $\cC$ to be
\begin{align}
\cC := \set {u=0}.
\end{align}
\end{defn}

\begin{rem} \label{rem:convexity_of_coincidence_set_of_global_solution}	As already mentioned before, global solutions to the obstacle problem are convex (see e.g. \cite[Theorem 5.1]{PetrosyanShahgholianUraltseva_book}).
In particular, the coincidence set $\cC$ of a {global} solution is {convex}.
\end{rem}

{
To get compactness of solutions, it is useful to recall that they are uniformly $C^{1,1}$-regular.
Also, as shown by Caffarelli, their blow-down limits with respect to quadratic rescaling are either half-space solutions or quadratic polynomials (see \cite{Caffarelli-revisited}). We summarize these results in the following two lemmas:

\begin{lem}[Characterization of blow-down limits]\label{lem:C11}
Let $u:\R^N\to [0,\infty)$ be a global solution to the obstacle problem.
  Then
  %% $\norm{D^2u}_{L^\infty(\R^N)}\leq C_N$ for some dimensional constant $C_N>0.$
  %% Also,
  the following convergence holds in $C^{1,\alpha}_{\operatorname{loc}}(\R^N)$ for each $\alpha \in (0,1)$:
$$
\lim_{r\to \infty}\frac{u(rx)}{r^2}=\left\{
\begin{array}{ll}
\frac12 \max(x\cdot e,0)^2&\text{for some }e \in \partial B_1,\\
\\
\frac12 x^TQx&\text{for some } Q\in \R^{N\times N} \text{ symmetric},\\
&\qquad \text{ nonnegative definite, satisfying }tr(Q)=1.
\end{array}
\right.
$$
A global solution of the form 
$\frac12 \max(x\cdot e,0)^2$ is called {\em half-space solution}.%%   A global solution of the form 
%% $\frac12 x^TQx$ with $Q$ as above is called {\em quadratic polynomial solution}. 
\end{lem}

\begin{lem}[Uniform regularity and compactness]
  \label{lem:compact C1}
The following regularity and compactness properties hold:
  \begin{enumerate}[(i)]
  \item Let $u$ be a global solution to the obstacle problem in $\R^N$ such that $u(0)=0$. Then $\norm{D^2u}_{L^\infty(\R^N)}\leq C_N$. \label{item:C11}
     \item \label{item:compact}
Let $(u_k)_{k \in \mathbb N}$ be a sequence of global solutions to the obstacle problem in $\R^N$ that vanish at the origin.
Then there exists a subsequence $(u_{k_j})_{j \in \mathbb N}$ converging to a global solution $u_0$ in $C^{1,\alpha}_{\rm loc}(\R^N)$ for each $\alpha \in (0,1)$. In addition, $\chi_{\{u_{k_j}=0\}} \to \chi_{\{u_0=0\}}$ a.e. in $\R^N$.
  \end{enumerate}
\end{lem}

\begin{proof}
The fact that $\norm{D^2u}_{L^\infty(\R^N)}$ is bounded by a dimensional constant is a consequence of the uniform quadratic growth of the solution
  (see \cite[Theorem 2.1]{PetrosyanShahgholianUraltseva_book}) combined with regularity
  estimates for harmonic functions inside the open set $\{ u>0\}$, as done for instance in the proof of \cite[Theorem 2.3]{PetrosyanShahgholianUraltseva_book}.
  This proves (i).
  
Concerning (ii), we note that the compactness in $C^{1,\alpha}_{\rm loc}(\R^N)$ is a direct consequence of (i) and Ascoli-Arzel\`a Theorem. The a.e. convergence of the characteristic functions of the contact sets follows  from \cite[Proposition 3.17(i)-(ii)]{PetrosyanShahgholianUraltseva_book}. 
\end{proof}
}
\begin{rem}
\label{rem:global liouville}
As noted in the previous proof,
global solutions grow at most quadratically at infinity (cf. \cite[Theorem 2.1]{PetrosyanShahgholianUraltseva_book}).
Also, if the convex set $\{u=0\}$ has empty interior, then $\Delta u \equiv 1$. Hence, Liouville's theorem implies that the only global solutions whose coincidence sets have empty interior are quadratic polynomials.
\end{rem}

Within the class of global solutions to the obstacle problem, we now introduce some terminology for denoting some special solutions

\begin{defn}[Cylindrical solutions] \label{def:cylindrical}
We say that a global solution to the obstacle problem is \emph{cylindrical} if there exists $e \in \partial B_1$ such that
\begin{align}
	\nabla u \cdot e \equiv 0 \quad \text{ in } \R^N.
\end{align}
\end{defn}

A useful criterion for being a cylindrical solution is contained in the following:
\begin{lem}
\label{lem:infinite line}
Let $u$ be a global solution, and assume that its coincidence set $\cC$ contains an infinite line. Then $u$ is constant in the direction of that line.
\end{lem}
\begin{proof}
First of all we may assume that $\cC$ has non-empty interior, as otherwise $u$ is a non-negative quadratic polynomial (see Remark~\ref{rem:global liouville}) and the result follows easily.

Since $\cC$ is convex, the assumption of containing a line implies that $\cC$ is a product, namely there exists a system of coordinates such that $\cC=\mathcal K\times \R$ for some convex set $\mathcal K\subset \R^{N-1}.$\footnote{This classical fact can be proved as follows. Assume that the line $\ell$ is parallel to $e^N$, that $\ell=\{\bar x+s e^N:s \in \R\}$ for some $\bar x \in \R^N$, and define $K_\tau:=\cC\cap \{x_N=\tau\}$. 
Let ${\rm conv}(A)$ denote the convex hull of the set $A$. Then, by convexity of $\cC$, ${\rm conv}(K_\tau\cup \ell)\subset \cC$ for any $\tau \in \R$. Since ${\rm conv}(K_\tau\cup \ell)=K_\tau \times \R$, it follows that  $\cC\supset \mathcal K\times \R$ with 
{$ \mathcal K\times \R:=\cup_\tau \left(K_\tau\times \R\right)$.} On the other hand, it is clear by construction that $\cC\subset \mathcal K\times \R$, so the result follows.}
Hence, given $\sigma \in \R$, the global solution $u_\sigma(x):=u(x+\sigma e^N)$ has the same contact set as $u$, and therefore $\Delta (u-u_\sigma)\equiv 0$.
Since $u-u_\sigma$ vanishes on $\cC$ which has non-empty interior, it follows by unique continuation that $u-u_\sigma\equiv 0$.
Since $\sigma$ is arbitrary, this shows that $u$ is invariant in the $e^N$-direction, proving the result.
\end{proof}

\begin{defn}[$x_N$-monotone solutions] \label{def:solution}
We say that a global solution to the obstacle problem \eqref{eq:obstacle-global} is \emph{$x_N$-monotone} if:
	\begin{enumerate}[(i)]
		\item \label{item:cC_has_nonempty_interior} $\cC$ has non-empty interior;
		\item \label{eq:conincidence_set_of_u_is_on_the_right} $\cC \subset \set{x_N \geq 0}$ and $0 \in \partial \cC$;
		\item \label{eq:monotone} 
{$\partial_Nu \leq 0$ in $\R^N$;}
		\item  \label{PDE_asymptotics}
		$\frac{u(rx)}{r^2} \to \frac12 x'^T Q x' =:p(x')$ in $C^{1,\alpha}_{\operatorname{loc}}(\R^N)$ as $r \to \infty$, $\alpha \in (0,1),$
		where $x=(x',x_N)$, $Q  \in \R^{(N-1) \times (N-1)}$ is symmetric, positive definite, and satisfies $\tr(Q) =  1$. 
	\end{enumerate}
\end{defn}

\begin{rem} \label{rem:coincidence_set_contains_ray}
Thanks to Definition \ref{def:solution}\eqref{eq:conincidence_set_of_u_is_on_the_right}-\eqref{eq:monotone}, if $u$ is  \emph{$x_N$-monotone} then $\{ t  e^N : t \geq 0   \} \subset \cC$.
Also, since the matrix $Q\in \R^{(N-1) \times (N-1)}$ in Definition \ref{def:solution}\eqref{PDE_asymptotics} is positive definite, there exists a constant $c_p>0$  such that 
		\begin{align} \label{eq:def_of_c_p}
		p(x') \geq c_p \abs {x'}^2 \qquad \text{ for all }x' \in \R^{N-1}.
		\end{align}
\end{rem}

The following important result on $x_N$-monotone solutions is proved in  \cite{esw_arXiv}:
\begin{lem}[$\cC$ is ``almost contained'' in a paraboloid]
\label{lem:prop5.1}
Let $N\geq 3$, and let $u$ be an $x_N$-monotone solution.
Fix $\delta\in (0,1)$, and define $T_\delta:=\left\{(y',y_N) \in \R^N: |y'|^2 < y_N^{1 + \delta}\right\}$.
Then there exists a radius ${\hat r}>1$ such that 
\begin{equation}
\label{eq:cC contained Tdelta}
\cC\setminus B_{2\hat r}\subset \{y_N>{\hat r}\}\cap T_\delta.
\end{equation}
\end{lem}

\begin{proof}
This result is proved in \cite[Proposition 5.1]{esw_arXiv} assuming that $\cC \cap \set {x_N \leq 0}=\set{0}$ in place of Definition~\ref{def:solution}\eqref{eq:conincidence_set_of_u_is_on_the_right}.
However, it is easy to check that exchanging this assumption with ours does not affect the proof.
\end{proof}

\begin{defn}[Ellipsoids and Paraboloids] \label{def:ellipsoid_and_paraboloid}
	 
We call a set $E \subset \R^N$ \emph{ellipsoid} if, after a translation and a rotation,
\begin{align}
	E = \biggl\{x \in \R^N : \sum \limits_{j = 1}^N \frac{x_j^2}{a_j^2} \leq 1   \biggr\}
\end{align}
for some $a=(a_1,\ldots,a_N) \in (0,\infty)^N$.  We call a set $P \subset \R^N$ a \emph{paraboloid}, 
if, after a translation and a rotation,   
$$
	P = \set { (x',x_N) \in \R^N : x_N\geq 0,\, x' \in \sqrt{x_N} E'   } ,
$$	  
where  $E' $  is an $(N-1)$-dimensional ellipsoid.
\end{defn}

An important role in this paper will be played by the Alt-Caffarelli-Friedman (ACF) functional originally introduced in \cite{ACF84}:
given a function $v:\R^N \to \R$ with $N\geq 2$, we define
\begin{align} \label{eq:definition_of_ACF-functional}
\Phi(v, r) := \frac{1}{r^4} \int \limits_{B_r} \frac{\abs {\nabla v_+}^2}{\abs {x}^{N-2} }  \dx{x} \int \limits_{B_r} \frac{\abs {\nabla v_-}^2}{\abs {x}^{N-2} }  \dx{x}.
\end{align}
We recall in the following lemma some useful facts about the ACF functional.
{
\begin{lem}[Properties of the Alt-Caffarelli-Friedman monotonicity functional]
\label{lem:ACF}
Let $N \geq 2$, and let $v:\R^N\to \R$ be a continuous $W^{1,2}_{\rm loc}$ function such that both $v_+$ and $v_-$ are subharmonic. Then:
\begin{enumerate}[(i)]
\item \label{item:ACF monotone} The functional $\Phi(v,r)$ is finite for each $r>0$, and 
$$
r\mapsto \Phi(v,r) \qquad \text{is  non-decreasing}.
$$
\item \label{eq:bound ACF individual}
The following bound holds for any $r>0$:
$$ \int \limits_{B_r} \frac{\abs {\nabla v_\pm}^2}{\abs {x}^{N-2} }  \dx{x} \leq C_N\biggl(\fint \limits_{B_{4r}}v_\pm\dx{x}\biggr)^2.$$
\item \label{item:ACF L2 bound} The following bound holds for any $r>0$:
$$
\Phi(v,r) \leq \frac{C_N}{r^4}\biggl(\fint \limits_{B_{4r}}v_+\dx{x}\biggr)^2 \biggl(\fint \limits_{B_{4r}}v_-\dx{x}\biggr)^2.$$
\item \label{item:ACF order}
Assume that $\Phi(v,R)\to 0$ as $R\to \infty$. Then either $v \geq 0$ 
{in $\R^N$} or $v \leq 0$ {in $\R^N$}.
\end{enumerate}

\end{lem}
\begin{proof}
Usually \eqref{item:ACF monotone} is stated and proved under the extra assumption $v(0)=0$. However, as noted in \cite[Theorem 2.4]{PetrosyanShahgholianUraltseva_book}, this extra condition is not needed and therefore \eqref{item:ACF monotone} holds in our setting.

By subharmonicity of $v_\pm$ and H\"older's inequality,  we can estimate 
\begin{equation}\label{acfeq}
\norm{v_\pm}_{L^2(B_{2r})}^2\leq \norm{v_\pm}_{L^\infty(B_{2r})}\norm{v_\pm}_{L^1(B_{2r})}\quad \text{ and }\quad
\norm{v_\pm}_{L^\infty(B_{2r})}\leq \frac{C_N}{r^{N}} \norm{v_\pm}_{L^1(B_{4r})}.\end{equation}
Also, as noted in \cite[Section 2.2.3]{PetrosyanShahgholianUraltseva_book}, the bound 
$$ \int \limits_{B_r} \frac{\abs {\nabla w}^2}{\abs {x}^{N-2} }  \dx{x} \leq C_N\biggl(\fint \limits_{B_{2r}}w^2\dx{x}\biggr)$$
holds for any nonnegative subharmonic function $w$.
Applying this inequality to $w=v_\pm$ and using \eqref{acfeq}, we obtain 
\eqref{eq:bound ACF individual}.

Multiplying the two estimates in \eqref{eq:bound ACF individual} for $v_+$ and $v_-$, \eqref{item:ACF L2 bound} follows.
%% Therefore
%% \begin{equation}
%% \label{eq:bound ACF individual}
%% \frac{1}{r^2} \int \limits_{B_r} \frac{\abs {\nabla v_\pm}^2}{\abs {x}^{N-2} }  \dx{x} \leq \frac{C_N}{r^2}\biggl(\fint \limits_{B_2r}v_\pm\dx{x}\biggr)^2,
%% \end{equation}
%% and multiplying the two estimates above (the one for $v_+$ and the one for $v_-$) we conclude the validity of \eqref{item:ACF L2 bound}.

To prove \eqref{item:ACF order} we can assume that there is a point $y \in \R^N$ such that $v(y)=0$ as otherwise, by continuity of $v$, either $v>0$ or $v<0$ and the result is trivially true.
By the monotonicity and non-negativity of $\Phi$, our assumption implies that 
$\Phi(v,r)\equiv 0$ for all $r>0$.
Hence, 
by the definition of the ACF functional (cf. \eqref{eq:definition_of_ACF-functional}), for each $r \in (0,\infty)$,
\begin{align}
\text{ either } \quad  \nabla v_+ \equiv 0 \quad \text{ in } B_r \quad \text{ or } \quad  \nabla v_- \equiv 0 \quad \text{ in } B_r.
\end{align}
Since by assumption $v(y)=0$ we deduce that, for all $r>|y|$,
\begin{align}
\text{ either } \quad  v_+ \equiv 0 \quad \text{ in } B_r \quad \text{ or } \quad  v_- \equiv 0 \quad \text{ in } B_r.
\end{align}
Therefore, by continuity,
\begin{align}
\text{ either } v_+ \equiv 0 \quad \text{ in } \bigcup_{r>|y|}B_r = \R^N \quad \text{ or } \quad v_- \equiv 0 \quad \text{ in } \bigcup_{r>|y|}B_r = \R^N,
\end{align}
which proves \eqref{item:ACF order}.
\end{proof}
}
We conclude this section with a couple of simple but important results on the difference of two global solutions.
These results will play a crucial role in the proof of Theorem~\ref{thm:main} where we will apply the ACF functional to the difference of two global solutions.

\begin{lem}[Subharmonicity properties and Caccioppoli estimate]
\label{lem:subharm}
Let $u_1,u_2:\R^N\to \R$ be global solutions to the obstacle problem. Then the following hold:
\begin{enumerate}[(i)]
\item \label{subharm}
The functions $(u_1-u_2)_+$, $(u_1-u_2)_-$ and $|u_1-u_2|$ are subharmonic.
\item {The following bound holds for any $r>0$: \label{cacc}
$$ \fint \limits_{B_r} |\nabla (u_1-u_2)|^2 \dx{x}
\le \frac{C_N}{r^2} \fint \limits_{B_{2r}(x^0)} (u_1-u_2)^2 \dx{x}.
$$}
\end{enumerate}
\end{lem}
\begin{proof}
Set $w:=u_1-u_2$ and note that, since $\Delta u_i=\chi_{\{u_i>0\}}$, 
\begin{equation}
\label{eq:lapl difference}
\Delta w=\chi_{\{u_1>0\}}-\chi_{\{u_2>0\}}\geq 0 \quad \text{ inside }\{u_1>u_2\}.
\end{equation}
Choosing a sequence of smooth convex non-decreasing functions $\phi_\epsilon:\R\to \R$ such that $\phi_\epsilon|_{(-\infty,0)}\equiv 0$ and $\phi_\epsilon(s)\to s_+$ as $\epsilon\to 0$ locally uniformly,
we see that
{
$$
\Delta [\phi_\epsilon(w)]=\phi_\epsilon'(w)\left[\chi_{\{u_1>0\}}-\chi_{\{u_2>0\}}\right]+\phi_\epsilon''(w)|\nabla w|^2 \geq 0.
$$
}
Letting $\epsilon \to 0,$ we conclude that $(u_1-u_2)_+$ is subharmonic.
Since $(u_1-u_2)_-=(u_2-u_1)_+$, by symmetry between $u_1$ and $u_2$ we deduce that $(u_1-u_2)_-$ is subharmonic.
Finally, since $|u_1-u_2|=(u_1-u_2)_+ +(u_1-u_2)_-$, the subharmonicity of $|u_1-u_2|$ follows. This proves \eqref{subharm}.

{
To prove \eqref{cacc} we define $w_r(x):=\frac{w(rx)}{r^2}$ and we note that, 
as a consequence of \eqref{eq:lapl difference}, it holds $w_r \Delta w_r \geq 0$, or equivalently 
\begin{equation}
\label{eq:delta w2}
\Delta (w_r^2)\geq 2|\nabla w_r|^2.
\end{equation}
Now, fix $0\le \eta\in C^{1}_c(B_{2})$ a cut-off function satisfying
$\eta\equiv 1$ in $B_1$.
Integrating the inequality \eqref{eq:delta w2} against $\eta$ we obtain
\begin{align}
\int \limits_{B_{1}} |\nabla w_r|^2 \dx{x} \leq \int \limits_{B_{2}} |\nabla w_r|^2\eta  \dx{x} \leq \frac12 \int_{B_2} w_r^2 \Delta\eta\dx{x}\leq 
 C \int \limits_{B_{2}} w_r^2 \dx{x},
\end{align}
as desired.}
\end{proof}

\begin{rem}
\label{rem:Deu}
As a  direct consequence of Lemma~\ref{lem:subharm}\eqref{subharm} we recover the well-known fact that, given a global solution $u$, $(\partial_e u)_+$ and $(\partial_e u)_-$
are subharmonic for each $e \in \partial B_1$. Indeed, given $h>0$ it suffices to apply Lemma~\ref{lem:subharm}\eqref{subharm} to $u$ and $u(\cdot +he)$ to deduce that both
$\left(\frac{u(\cdot +he)-u}{h}\right)_+$ and $\left(\frac{u(\cdot +he)-u}{h}\right)_-$ are subharmonic, and then the result follows by letting $h\to 0$.
\end{rem}

\begin{lem}[Strong convergence] 
\label{lem:strong_W_1_2_convergence_of_v_R_R}
Let $\rho \in (0,\infty)$ and  let $(u_k)_{k \in \N}$, $(v_k)_{k \in \N}$ be two sequences of global solutions to the obstacle problem in $\R^N$ such that 
\begin{align} \label{eq:weak_convergence_in_the_weak_implies_strong_lemma}
w_k := u_k - v_k \rightharpoonup w \quad \text{ weakly in } W^{1,2}(B_\rho) \text{ as } k \to \infty,
\end{align}
for some harmonic function $w:\R^N\to \R$. Then, for each $\delta \in (0,1)$,
\begin{align}
w_k = u_k - v_k \to w \quad \text{ strongly in } W^{1,2}(B_{\delta \rho}) \text{ as } k \to \infty.
\end{align}
\end{lem}

\begin{proof}
First of all note that, for all $k \in \N$,
$$
w_k \Delta w_k=(u_k-v_k)\big(\chi_{\{u_k>0\}}-\chi_{\{v_k>0\}}\big) \geq 0.
$$
Hence, given $\eta \in C_c^\infty(B_\rho; [0,\infty))$ satisfying $\eta \equiv 1$ in $B_{\delta \rho}$, integrating by parts twice we get
\begin{align}
\int \limits_{B_{\delta \rho}} \abs{ \nabla  w_k }^2  \leq \int \limits_{B_{\rho}} \abs{ \nabla  w_k }^2 \eta &= - \int \limits_{B_{\rho}} \left(w_k  \nabla w_k \cdot \nabla \eta + \eta w_k  \Delta w_k\right) \leq - \int \limits_{B_{\rho}} w_k \nabla w_k  \cdot \nabla \eta \\ 
&\quad \to - \int \limits_{B_{\rho}} w \nabla w \cdot \nabla \eta = \int \limits_{B_{\rho}} \abs{ \nabla w}^2 \eta \quad \text{ as } k \to \infty,
\end{align}
where the last equality follows by the harmonicity of $w$.

Now, choosing a sequence $(\eta_j)_{j \in \N} \subset C^\infty_c(B_{\rho}; [0,\infty))$ such that $\eta_j \equiv 1$ in $B_{\delta \rho}$ for all $j \in \N$ and $\eta_j \to \chi_{B_{\delta \rho}}$ pointwise in $B_{\rho}$, we conclude that 
\begin{align}
	\limsup \limits_{k \to \infty} \int \limits_{B_{\delta \rho}} \abs{ \nabla  w_k }^2 \leq \int \limits_{B_{\delta \rho}} \abs{ \nabla w }^2.
\end{align}
Therefore, by the lower-semicontinuity of the Dirichlet energy we deduce that $\norm{\nabla w_k}_{L^2(B_{\delta \rho})} \to \norm{\nabla w}_{L^2(B_{\delta \rho})}$. This convergence of the $L^2$-norm of the gradients together with the weak convergence implies the desired strong convergence.
\end{proof}

%\begin{lem}[Estimate of the tail of the potential integral]\label{tail}
%	Let $N = 3$, $M \subset \R^{2} \times [0,\infty)$ such that $\bV_M$ is well-defined. Then there is $C>0$ such that for all  $x \in \R^3$,
%	\begin{align}
%	\abs{	\bV_{M \setminus B_{2 \abs{x}}} (x)  }    \leq C ~\abs{x}^2 \int \limits_{M \setminus B_{2 \abs{x}}}   \abs{y}^{-3} \dx{y}.
%	\end{align}
%\end{lem}
%
%\begin{proof}
%
%\end{proof}

%%%%%%%%%%%%
%%%%%%%%%%%%
%%%%%

%\section{Earlier results and  reduction to the paraboloid case}\label{early-results}
%\subsection{Earlier results}

%%%%%%%%%%%%%%%%

\section{{The Newtonian potential expansion}} \label{section:Newton_potential_expansion_of_u}

%% In this section, thanks to the growth estimate for the coincidence set in Lemma~\ref{lem:prop5.1} we will show that, for $x_N$-monotone solutions, the generalized Newtonian potential associated to the coincidence set $\cC$ is well-defined and has subquadratic growth. 

As mentioned in Remark~\ref{rem:6}, in \cite{esw_arXiv} a very important role is played by the Newtonian potential associated to the coincidence set of a solution, defined (up to a multiplicative constant) as $\frac{1}{|x|^{N-2}}\ast \chi_{\mathcal C}$. Unfortunately, if $\mathcal C$ is a paraboloid then the above convolution converges only for $N\geq 6$. For this reason we will introduce a generalized Newtonian potential in the spirit of \cite{GNP1}, which will be shown in Lemma \ref{lem:scaling_of_generalized_potential} to be well-defined and to have subquadratic growth.

\begin{defn}[Generalized Newtonian potential]\label{NP}
	 
	Let $N \geq 3$, and define the function
	$$
	G(x,y):= \frac{1}{\abs{x-y}^{N-2}} - \frac{1}{\abs{y}^{N-2}}  - (N-2)\frac{x \cdot y}{ \abs{y}^N} \qquad \text{ for all }x,y \in \R^N. 
	$$
	Given $M\subset \R^N$ measurable, assume that $G(x,\cdot)\chi_M \in L^1(\R^N)$ for each $x \in \R^N$. Then 
	we define the \emph{generalized Newtonian potential} associated to $M$ as
        \begin{align} \label{eq:generalized Newtonian potential}
        	\bV_M(x) := \alpha_N \int \limits_M G(x,y)\dx{y}, \text{ where } \alpha_N:=\frac{1}{N(N-2)|B_1|}.
        \end{align}
\end{defn}

\begin{lem}[Scaling of the generalized Newtonian potential] \label{lem:scaling_of_generalized_potential}
	Let $M \in \R^N$ be a measurable set  for which  $\bV_M$ is well-defined. Then $V_{M}$ satisfies the following scaling law:
	\begin{align}
		\bV_M(\gamma x) = \gamma^2 ~ \bV_{\frac{1}{\gamma} M}(x) \text{ for all } \gamma >0.
	\end{align}
\end{lem}
\begin{proof}
The proof follows from a direct calculation: 
since 
$$
G(\gamma x,y)=\gamma^{2-N}G\left(x,\frac{y}{\gamma}\right)\qquad \text{ for all } \gamma>0,\,x,y\in \R^N, 
$$
\begin{multline}
\bV_M(\gamma x) = \alpha_N \int \limits_M G(\gamma x,y) \dx{y} 
=\gamma^{2-N}  \alpha_N \int \limits_M G\left(x,\frac{y}{\gamma}\right) \dx{y} \\
=\gamma^{2-N}  \alpha_N \int \limits_{\frac{1}{\gamma} M} G(x,z) \gamma^N \dx{z}  =  \gamma^2 ~ \bV_{\frac{1}{\gamma} M}(x).
\end{multline}
\end{proof}

\begin{lem}[Generalized Newtonian potential of $\cC$] \label{lem:existence_of_newton_potential}
	 
Let $N\geq 3,$ and let $u$ be an $x_N$-monotone solution in the sense of Definition \ref{def:solution}. Then
	\begin{enumerate}[(i)]
		\item  \label{item:generalized_potential_is_locally_bounded}
		The generalized potential $\bV_\cC$ of $\cC$ is well-defined and locally bounded.
		 \\
		\item \label{item:subquadratic_growth_of_Newton_potential_of_coincidence_set}
		$\bV_\cC(x)$ grows subquadratically as $\abs {x} \to \infty$. More precisely, there exists a constant $C$ such that
		\begin{align}
			\frac{|\bV_\cC(x)|}{(1+|x|)^{7/4}} \leq C\qquad \text{ for all }x \in \R^N.
		\end{align}
		\item \label{item:W2p VC} $\bV_{\cC} \in W^{2,p}_{\rm loc}(\R^N)$ for each $p\in [1,\infty)$, $\Delta \bV_\cC = - \chi_\cC$, and $\bV_\cC(0)=|\nabla  \bV_\cC(0)|=0$.
	\end{enumerate}
\end{lem}

\begin{proof}
Fix $\delta\in (0,1)$,  let $T_\delta$ be as in Lemma~\ref{lem:prop5.1}, and recall that \eqref{eq:cC contained Tdelta} holds.

To prove the estimate, we first note that the trivial bound
$$
|G(x,y)|\leq \frac{1}{\abs{x-y}^{N-2}} + \frac{1}{\abs{y}^{N-2}}  + (N-2)\frac{|x|}{ \abs{y}^{N-1}}
$$
holds. Also, by the Taylor expansion $f(1) = f(0)+ f'(0) + \int_0^1 (1-\tau) f''(\tau) \dx{\tau}$ applied to $f(\tau) := \frac{1}{|\tau x-y|^{N-2}}$, we get
\begin{equation}
\label{eq:Taylor G}
|G(x,y)|\leq C\frac{|x|^2}{ \abs{y}^{N}}\qquad \text{for }|y| > 2|x|.
\end{equation}
Using these two bounds and \eqref{eq:cC contained Tdelta}, we obtain
\begin{align}
\int_{\mathcal C}|G(x,y)|\dx{y} &\leq \int_{\mathcal C\cap B_{2\hat r}} \bra{\frac{1}{\abs{x-y}^{N-2}} + \frac{1}{\abs{y}^{N-2}}  + (N-2)\frac{|x|}{ \abs{y}^{N-1}}}\dx{y}\\
&
+ \int_{\mathcal C\cap (B_{2|x|}\setminus B_{2\hat r})}\bra{\frac{1}{\abs{x-y}^{N-2}} + \frac{1}{\abs{y}^{N-2}}  + (N-2)\frac{|x|}{ \abs{y}^{N-1}}}\dx{y}\\
&+{ C\int_{\mathcal C\setminus B_{2|x|}} \frac{|x|^2}{ \abs{y}^{N}} \dx{y}}\\
&\leq \int_{B_{2\hat r}} \bra{\frac{1}{\abs{x-y}^{N-2}} + \frac{1}{\abs{y}^{N-2}}  + (N-2)\frac{|x|}{ \abs{y}^{N-1}}}\dx{y} \label{eq:int G}\\
&
+ \int_{T_\delta \cap (B_{2|x|}\setminus B_{2\hat r})}\bra{\frac{1}{\abs{x-y}^{N-2}} + \frac{1}{\abs{y}^{N-2}}  + (N-2)\frac{|x|}{ \abs{y}^{N-1}}}\dx{y}\\
&+C \int_{T_\delta \setminus (B_{2|x|}\cup B_{2\hat r})}\frac{|x|^2}{ \abs{y}^{N}}\dx{y}=:I_1+I_2+I_3.
\end{align}
Since the integral of $\frac1{|x-y|^{N-2}}$ over a ball is maximized when $x$ coincides with the center of the ball,
for the first integral $I_1$ we have
$$
\int_{B_{2\hat r}} \bra{\frac{1}{\abs{x-y}^{N-2}} + \frac{1}{\abs{y}^{N-2}} }\dx{y} \leq 2
\int_{B_{2\hat r}} \frac{1}{\abs{y}^{N-2}} \dx{y} \leq C \text{ and }
\int_{B_{2\hat r}} \frac{1}{ \abs{y}^{N-1}}\dx{y}\leq C,
$$
hence 
\begin{equation}
\label{eq:I1}
I_1\leq C(1+|x|).
\end{equation}
About $I_2$, we note that this integral is nonzero only if $|x|>{\hat r}$. In such a case, we 
observe that, provided that ${\hat r}$ is large enough,
$$
T_\delta \cap (B_{2|x|}\setminus B_{2\hat r}) \subset T_\delta \cap \{{\hat r}<y_N<2|x|\}.
$$
Hence, since $\frac1{|y|}\leq \frac1{y_N}$ and $N \geq 3$,
\begin{multline}
\int_{T_\delta \cap (B_{2|x|}\setminus B_{2\hat r})} \frac{1}{\abs{y}^{N-2}} \dx{y} \leq 
\int_{T_\delta \cap \{{\hat r}<y_N<2|x|\}}\frac{1}{y_N^{N-2}} \dx{y} 
= \int_{{\hat r}}^{2|x|}  \cH^{N-1}(T_\delta \cap  \{y_N=t\}) \frac{1}{t^{N-2}} \dx{t} \\
\leq C\int_{{\hat r}}^{2|x|}  t^{\frac{(N-1)(1+\delta)}{2}+2-N} \dx{t}\leq C(1+|x|)^{\frac{(N-1)(1+\delta)}{2}+3-N} \label{eq:I2 1}
\leq C(1+|x|)^{1+\delta},
\end{multline}
and analogously
\begin{align}
\int_{T_\delta \cap (B_{2|x|}\setminus B_{2\hat r})} \frac{1}{\abs{y}^{N-1}} \dx{y} & \leq C\int_{{\hat r}}^{2|x|}  t^{\frac{(N-1)(1+\delta)}{2}+1-N} \dx{t}\\
& \leq C(1+|x|)^{\frac{(N-1)(1+\delta)}{2}+2-N}\leq C(1+|x|)^{\delta}, \label{eq:I2 2}
\end{align}
therefore
$$
\int_{T_\delta \cap (B_{2|x|}\setminus B_{2\hat r})}\bra{ \frac{1}{\abs{y}^{N-2}}  + (N-2)\frac{|x|}{ \abs{y}^{N-1}}}\dx{y} \leq C(1+|x|)^{1+\delta}.
$$
Finally, to estimate the integral of $\frac{1}{\abs{x-y}^{N-2}}$, we use again that the integral of $\frac1{|x'-y'|^{N-2}}$ over a ball is maximized when $x'$ coincides with the center of the ball.
This yields
$$
\int_{T_\delta \cap  \{y_N=t\}} \frac{1}{|x'-y'|^{N-2}}\dx{y'}=
 \int_{B_{t^{\frac{1+\delta}{2}}}'} \frac{1}{|x'-y'|^{N-2}} \dx{y'}  \leq  \int_{B_{t^{\frac{1+\delta}{2}}}'} \frac{1}{|y'|^{N-2}} \dx{y'}  =C t^{\frac{1+\delta}{2}},
$$
and therefore, since 
$\frac1{|x-y|}\leq \frac1{|x'-y'|}$,
\begin{align}
\int_{T_\delta \cap (B_{2|x|}\setminus B_{2\hat r})} \frac{1}{\abs{x-y}^{N-2}} \dx{y} \leq C\int_{{\hat r}}^{2|x|}  t^{\frac{1+\delta}{2}} \dx{t}\leq C|x|^{\frac{3+\delta}{2}}.
\end{align}
Overall, this proves that $I_2\leq C(1+|x|)^{\frac{3+\delta}{2}}$.

Finally, for $I_3$, we simply note that 
$$
T_\delta {\setminus} (B_{2|x|}\cup B_{2\hat r}) \subset T_\delta \cap \big\{y_N>\max\{|x|,{\hat r}\}\big\},
$$
hence
\begin{multline}\label{eq:I3}
 \int_{T_\delta \setminus (B_{2|x|}\cup B_{2\hat r})}\frac{|x|^2}{ \abs{y}^{N}}\dx{y}
 \leq |x|^2 \int_{\max\{|x|,{\hat r}\}}^{\infty}  \cH^{N-1}(T_\delta \cap  \{y_N=t\}) \frac{1}{t^{N}} \dx{t}\\
 = C|x|^2 \int_{\max\{|x|,{\hat r}\}}^{\infty}  t^{\frac{(N-1)(1+\delta)}{2}-N}  \dx{t} \leq  C |x|^2(1+|x|)^{\frac{(N-1)(1+\delta)}{2}+1-N}\leq C(1+|x|)^{1+\delta}.
\end{multline}
Combining all these bounds, we have shown that 
$$
\int_{\mathcal C}|G(x,y)|\dx{y}\leq C(1+|x|)^{\frac{3+\delta}{2}},
$$
where $\delta \in (0,1)$ is arbitrary. This proves that $V_{\mathcal C}$ is well-defined and locally bounded.
Also, choosing $\delta = 1/2$, we obtain that
\begin{equation}
\label{eq:sublinear}
\frac{|\bV_\cC(x)|}{(1+|x|)^{7/4}} \leq C_{\hat r,N}\qquad \text{ for all }x \in \R^N
\end{equation}
where the constant $C_{\hat r,N}$ depends only on the dimension $N$ and the radius $\hat r$ defined in \eqref{eq:cC contained Tdelta} for $\delta=1/2.$

To prove the $W^{2,p}$-regularity of $V_{\mathcal C}$ we note that, for $\rho > {2}\max\{{\hat r},|x|\}$,
$$
|\bV_{\cC}(x)-\bV_{\cC\cap B_{\rho}}(x)|\leq C \int_{T_\delta \setminus B_{\rho}}\frac{|x|^2}{ \abs{y}^{N}}\dx{y} \leq C |x|^2 \int_{\rho}^{\infty}  t^{\frac{(N-1)(1+\delta)}{2}-N}  \dx{t} \leq C|x|^2\rho^{\delta-1}.
$$
This implies that $\bV_{\cC}$ is the locally uniform limit of the sequence of the continuous functions $\bV_{\cC\cap B_\rho}$ as $\rho \to \infty$.
Also, since
$$
\alpha_N\Delta_x\biggl(\frac{1}{\abs{x-y}^{N-2}} - \frac{1}{\abs{y}^{N-2}}  - (N-2)\frac{x \cdot y}{ \abs{y}^N} \biggr)=-\delta_x\quad \text{in the sense of distributions},
$$
one easily deduces that $\Delta \bV_{\cC\cap B_\rho}=-\chi_{\cC\cap B_\rho} \in L^\infty(\R^N)$ for each $\rho>0$. Thus, by elliptic regularity, the functions $\bV_{\cC\cap B_\rho}$ are locally uniformly bounded in $W^{2,p}$ for each $p<\infty$. In particular
because of the compact embedding $W^{2,p}(B_\rho)\hookrightarrow C^{1,\alpha}(B_\rho)$
for $p>N$, we deduce that
$$
\bV_{\cC\cap B_\rho} \to \bV_{\cC} \quad \text{and}\quad \nabla\bV_{\cC\cap B_\rho} \to \nabla \bV_{\cC} \quad \text{locally uniformly in $\R^N$, as $\rho \to \infty$.}
$$
Since $G(0,\cdot)\equiv 0$ and $\nabla_xG(0,\cdot)\equiv 0$, we obtain for each $\rho>0$ that
$$
\bV_{\cC\cap B_\rho}(0)=\int_{\cC\cap B_\rho}G(0,y)\dx{y}=0,\qquad \nabla \bV_{\cC\cap B_\rho}(0)=\int_{\cC\cap B_\rho}\nabla_x G(0,y)\dx{y}=0,
$$
so we conclude that $\bV_\cC(0)=\nabla  \bV_\cC(0)=0$.
\end{proof}
	
	As a consequence of the previous lemma, we can now show the following important result.
\begin{prop}[Newtonian potential expansion]\label{prop:Newton_potential_expansion}
Let $N\geq 3$, let $u$ be an $x_N$-monotone solution in the sense of Definition \ref{def:solution},
and let $p$ be the blow-down limit in Definition \ref{def:solution}\eqref{PDE_asymptotics}.  Then the expansion
	\begin{align} \label{eq:potential_expansion_of_the_solution}
		u = p +\bV_\cC
	\end{align}
	holds.
\end{prop}

\begin{proof}
Recall that, thanks to Lemma \ref{lem:existence_of_newton_potential}\eqref{item:W2p VC},
$\bV_\cC$ is a strong $W^{2,p}_{\operatorname{loc}}(\R^N)$  solution of 
$\Delta \bV_\cC = - \chi_\cC$.
Moreover, if we set $v := u-p$, then $v \in C^{1,1}_{\loc}(\R^N)$ (see Lemma~\ref{lem:compact C1}\eqref{item:C11}) and it solves the same equation as $\bV_\cC$, i.e. $\Delta v = - \chi_\cC$.
Hence $v-\bV_\cC$ is harmonic in $\R^N$, and it follows from Definition \ref{def:solution}\eqref{PDE_asymptotics} and Lemma \ref{lem:existence_of_newton_potential}\eqref{item:subquadratic_growth_of_Newton_potential_of_coincidence_set}
that $v-\bV_\cC$ has subquadratic growth. This allows us to apply Liouville's theorem
to obtain that
\begin{align}
	v-\bV_\cC = \ell +c,
\end{align}
where $\ell$ is a linear function and $c$ is a constant. Thus we have proved 
\begin{align} \label{eq:Newton_potential_expansion_of_u}
	u = p + \ell + c + \bV_\cC\quad \text{ in } \R^N.
\end{align}
Now, since $0 \in \partial \cC$, it follows from Lemma \ref{lem:existence_of_newton_potential}\eqref{item:W2p VC} that
\begin{align}
0 &= u(0) = p(0) + \ell(0) + c +V_\cC(0) = c, \\
0 &= \nabla u(0) = \nabla p(0) + \nabla \ell(0) + \nabla V_\cC(0) =  \nabla \ell(0).
\end{align}
This proves that both $\ell$ and $c$ vanish,
 concluding  the proof.
\end{proof}

As we shall see in the next section, this potential expansion allows us to obtain a  very precise control on the asymptotic behavior of the coincidence set $\cC$.

\section{Improved estimate on the asymptotic behavior of the coincidence set $\cC$} \label{sec:improved_estimate_on_asymptotics_of_C}

The goal of this section is to prove that $\cC$ is contained in some paraboloid. While for $N\geq 4$ there is a very simple argument to prove this result, the proof for $N=3$ is amongst the most delicate of this paper (see in particular the proof of Lemma~\ref{lem:estimate_on_paraboloid_from_outside_in_measure_sense} below).

\begin{prop}[$\cC$ is contained in a paraboloid] \label{prop:C_is_contained_in_paraboloid}
	 
	Let $N\geq 3$, and let $u$ be an $x_N$-monotone solution in the sense of Definition \ref{def:solution}.
 Then there are constants $a_0, \gamma_0 \in (0,+\infty)$ such that:
	\begin{enumerate}[(i)]
		\item \label{item:growth_of_coincidence_set}
				$\cC \cap \set {x_N >a_0} \subset \big\{\abs{x'}^2 < \gamma_0 x_N\big\}$;
		
		\item \label{item:boundedness of coincidence_set_around_zero}
		$\cC \cap \set {x_N \leq a_0}$  is bounded.
	\end{enumerate}
\end{prop} 

\begin{proof}%[Proof of Proposition \ref{prop:C_is_contained_in_paraboloid}] 
Thanks to Lemma~\ref{lem:prop5.1} it follows that $\cC \cap \set {x_N \leq a}$  is bounded for each $a>0$, so \eqref{item:boundedness of coincidence_set_around_zero} holds for each $a_0>0$. In particular,
 it suffices to prove \eqref{item:growth_of_coincidence_set} for $a_0$ sufficiently large.

We first prove the result in the case $N\geq 4$ (since the proof is very simple), and then focus on the delicate case $N=3$.

\noindent
$\bullet $ \textbf{The case $N\geq 4$.}
Arguing as in the proof of Lemma \ref{lem:existence_of_newton_potential}, given $\delta\in (0,1)$, for $|x|> \hat r$ we have
\begin{align}
\frac{1}{\alpha_N}V_{\cC}(x)=\int_{\mathcal C}G(x,y)\dx{y} &\geq \int_{\cC\cap B_{2\hat r}} \bra{\frac{1}{\abs{x-y}^{N-2}} - \frac{1}{\abs{y}^{N-2}}  - (N-2)\frac{|x|}{ \abs{y}^{N-1}}}\dx{y} \\
&\quad+ \int_{\cC \cap (B_{2|x|}\setminus B_{2\hat r})} \bra{\frac{1}{\abs{x-y}^{N-2}} - \frac{1}{\abs{y}^{N-2}}  - (N-2)\frac{|x|}{ \abs{y}^{N-1}} }\dx{y}\\
&\quad-C \int_{T_\delta \setminus B_{2|x|}}\frac{|x|^2}{ \abs{y}^{N}}\dx{y}=:I_1+I_2+I_3
\end{align}
(cp. \eqref{eq:int G}).
Then, again as in the proof of Lemma \ref{lem:existence_of_newton_potential}, we have that $|I_1|\leq C(1+|x|)$
and 
$$
|I_3| \leq  C |x|^2(1+|x|)^{\frac{(N-1)(1+\delta)}{2}+1-N} \leq C(1+|x|)^{\frac{1+3\delta}{2}} \qquad \text{for }N \geq 4
$$
(cp. \eqref{eq:I1} and \eqref{eq:I3}).
For $I_2$, we observe that the first term is non-negative and we estimate the remaining two as in \eqref{eq:I2 1} and \eqref{eq:I2 2}, so to get
\begin{align}
I_2 &\geq -\int_{T_\delta \cap (B_{2|x|}\setminus B_R)} \bra{\frac{1}{\abs{y}^{N-2}}  + (N-2)\frac{|x|}{ \abs{y}^{N-1}} }\dx{y}\\
&\geq -C(1+|x|)^{\frac{(N-1)(1+\delta)}{2}+3-N}\geq -C(1+|x|)^{\frac{1+3\delta}{2}} \qquad \text{for }N \geq 4.
\end{align}
Choosing $\delta {=} 1/3$ proves that $V_\cC(x)\geq -C(1+|x|)$ for all $x \in \R^N$.

Now, applying Proposition \ref{prop:Newton_potential_expansion} and combining this bound with \eqref{eq:def_of_c_p}, we conclude that
$$
0=u(x)=p(x)+V_\cC(x)\geq c_P |x'|^2-C(1+|x|) \qquad \text{ for all }x \in \{u=0\}.
$$
From this estimate we easily deduce that
$$
|x'|^2\leq C(1+ |x_N|)=C(1+x_N)\qquad  \text{ for all }x \in \{u=0\}
$$
(recall that $\{u=0\}\subset \{x_N \geq 0\}$),
so \eqref{item:growth_of_coincidence_set} follows.

\noindent
$\bullet$ \textbf{The case $N=3$.}
This case follows from Lemmas \ref{lem:sections_are_asymptotically_bounded_by_measure} and \ref{lem:estimate_on_paraboloid_from_outside_in_measure_sense} below.
\end{proof}

The rest of the section is devote to the proof of Lemmas \ref{lem:sections_are_asymptotically_bounded_by_measure} and \ref{lem:estimate_on_paraboloid_from_outside_in_measure_sense}.

%
%\begin{lem}[{Monotonicity of sections of $\cC$ \cite[Lemma 5.4]{N4_5}}] \label{lem:monotonicity_of_the_coincidence_set}
%	 
%	The sections $\cC_t := \{y' \in \R^2 : (y',t) \in \cC\}$ are increasing in $t$ in the sense of sets i.e.
%	\begin{align}
%	\text{ for all } t, \tilde{t} \geq 0 \text{ such that } t \geq \tilde{t}: \quad \cC_{\tilde{t}} \subset \cC_t.
%	\end{align}
%\end{lem}

\begin{lem}[Sections of $\cC$ are controlled by their measure] \label{lem:sections_are_asymptotically_bounded_by_measure}
Let $N= 3$, and let $u$ be an $x_N$-monotone solution in the sense of Definition \ref{def:solution}.
We define $\cC_t := \{y' \in \R^{2} : (y',t) \in \cC\}$ and $H(t) := \cH^{2}(\cC_t)$ for all $t \geq 0$. Then:\\
- either $\{\cC_t\}_{t \geq 0}$ is bounded, i.e. $\sup_{t \geq 0} \operatorname{diam}(\cC_t) < \infty$;\\
- or
there exist $a_0>1$ and $C_0 < \infty$ such that, for all $x_3 \geq a_0$,
\begin{align}
%	\set{ |x'| < c_0 \sqrt{H(t)}  } \subset 
	\cC_{x_3} \subset \set { |x'|^2 < C_0 H(x_3)}  .
\end{align}
\end{lem}

\begin{proof} We may assume that
	\begin{align} \label{eq:diam_of_sections_is_unbounded}
	\sup_{t \geq 0} \operatorname{diam}(\cC_t) = \infty.
	\end{align}
Suppose, towards a contradiction, that the statement of the lemma is not true. Then, there exists a sequence $(x^n)_{n \in \N} \subset \R^{2} \times (0, \infty)$ such that $x_3^n \to \infty$ as $n \to \infty$ and, for all $n \in \N$,
\begin{align} \label{eq:contradiction_assumption_nodegeneracy_of_sections}
%\bra { |(x^n)'| < \frac{1}{n} \sqrt{H(x_3^n)} \text{ and } x^n \not \in \cC   } \quad \text{ or } \quad
 |(x^n)'|^2 > n H(x^n_3) \quad  \text{ and } \quad x^n \in \cC .
\end{align}
Define $d_n := \operatorname{diam}(\cC_{x_3^n})$. From Lemma~\ref{lem:prop5.1} we know that, given $\delta \in (0,1)$,
\begin{align} \label{eq:bound_of_diameter_of_sections}
	d_n \leq (x_3^n)^{\frac{1+\delta}{2}}\qquad \text{ for all $n$ sufficiently large}.
\end{align}
On the other hand, \eqref{eq:diam_of_sections_is_unbounded} together with the monotonicity of $t\mapsto \operatorname{diam}(\cC_t)$ (recall that by Definition \ref{def:solution}, $u$ is {de}creasing in the $e^3$-direction) imply that 
\begin{align} \label{eq:diam_of_section_goes_to_infty}
	d_n \to \infty \quad \text{ as } n \to \infty.
\end{align}
Let us define for each $n \in \N$ the rescaling
\begin{align}
	u_n(x) := \frac{u( (0, x_3^n) + d_n x)}{d_n^2}.
\end{align}
Note that, as a consequence of \eqref{eq:contradiction_assumption_nodegeneracy_of_sections}, the convex sets $\{u_n=0\}\cap \{x_3=0\}\subset \R^{2}$ have diameter $1$, contain the origin, and their $\cH^{2}$-measure goes to zero. 
Hence, 
thanks to Lemma~\ref{lem:compact C1}\eqref{item:compact}, the fact that
$t \mapsto \cC(t)$ is increasing (recall that by Definition \ref{def:solution}, $u$ is decreasing in the $e^3$-direction) and the convergence of coincidence sets (see \cite[Proposition 3.17 (iv) and Proposition 3.17 (v)]{PetrosyanShahgholianUraltseva_book}), passing if necessary to a subsequence we obtain that
\begin{align}
	u_n \to u_0 \quad \text{ in } C_{\loc}^{1,\alpha}(\R^3) \text{ as } n \to \infty,
\end{align}
and 
\begin{align} \label{eq:measure zero section}
\operatorname{diam}(\{u_0=0\}\cap \{x_3=0\}){\geq} 1,\qquad	{|\{u_0=0\}\cap \{-1<x_3<0\}|=0},
\end{align}
where $u_0$ is a global solution to the obstacle problem.
Also, since $0 \in \cC$ (cf. Definition \ref{def:solution}\eqref{eq:conincidence_set_of_u_is_on_the_right}) and $\delta \in (0,1)$, it follows from \eqref{eq:bound_of_diameter_of_sections} that 
\begin{align}
\frac{\operatorname{dist}(0, (0,x^n_3))}{d_n} \geq \tfrac{x_3^n}{(x_3^n)^\frac{1+\delta}{2}} \to \infty \quad  \text{ as } n \to \infty.
\end{align}
Thus, by the convexity of $\cC$ (cf. Remark \ref{rem:convexity_of_coincidence_set_of_global_solution}), we deduce that $\{t e^3 : t \leq 0\} \subset \{u_0=0\}$. 
On the other hand, the fact that $\{t e^3 : t \geq 0\} \subset \cC$ (see Remark \ref{rem:coincidence_set_contains_ray}) implies that $\{ t e^3 : t \geq 0\} \subset \{u_0=0\}$.
Hence
\begin{align}
	\{ t e^3 : t \in \R \} \subset \{ u_0=0\},
\end{align}
and therefore it follows from Lemma~\ref{lem:infinite line} that $u_0$ is invariant  the $e^3$-direction, i.e.
\begin{align}
	\quad u_0(x) = u_0(x',0) \qquad \text{ for all } x \in \R^3.
\end{align}
Combining this information with \eqref{eq:measure zero section}, we deduce that the coincidence set of $u_0$ has measure zero, hence Remark~\ref{rem:global liouville} implies that $u_0$ coincides with a quadratic polynomial $q=q(x')$. 
On the other hand, \cite[Lemma B.2]{esw_arXiv} implies that the blow-down limit of $u_0$ is $p$ (being the blow-down limit of $u$), and therefore the only possibility is that $u_0=q=p.$
By the nondegeneracy of $p$ in $\R^{2}$ (see Definition \ref{def:solution}\eqref{PDE_asymptotics}), this implies that $\{u_0=0\}\cap \{x_3=0\}$ coincides with the origin,
a contradiction to the fact that this set has diameter at least 1 (see \eqref{eq:measure zero section}).
This contradiction proves the lemma.
\end{proof}

\begin{lem}[The measure of sections grows at most linearly] \label{lem:estimate_on_paraboloid_from_outside_in_measure_sense}
Let $N= 3$, and let $u$ be an $x_N$-monotone solution in the sense of Definition \ref{def:solution}.
Then there exists a constant $C$ such that for, all $t \geq 0$,
\begin{align}
\cH^{2}(\cC \cap \set{x_3 = t })\leq C(1+ t).
\end{align}
\end{lem}

\begin{proof}
We split the proof into 8 steps.
 \\
\textbf{Step 1.} \emph{Preliminary observations about sections of $\cC$}. \\
Recall the notation $\cC_t :=\{y' \in \R^2 : (y',t) \in \cC\}$ and $H(t) := \cH^{2}(\cC_t)$.

First, we claim  that $\sqrt{H}$ is a concave function, i.e. for all $\lambda \in [0,1]$, $t_1, t_2 \geq 0$
\begin{align}
\lambda \sqrt{H(t_1)} + (1-\lambda) \sqrt{H(t_2)} \leq \sqrt{H(\lambda t_1 + (1-\lambda) t_2)}.
\end{align}
Indeed, by the Brunn-Minkowski inequality (in $\R^2$)
\begin{align}
    \bra { \lambda \sqrt{\cH^{2}(\cC_{t_1})} + (1-\lambda) \sqrt{\cH^{2}(\cC_{t_2})}  }^2 \leq	\cH^{2}(\lambda \cC_{t_1} + (1-\lambda) \cC_{t_2} ).
\end{align}
Hence, since $\lambda \cC_{t_1} + (1-\lambda) \cC_{t_2} \subset \cC_{\lambda t_1 + (1-\lambda) t_2}$ (by the convexity of $\cC$), the claim follows. 

Now, 
the concavity of $\sqrt{H}$ together with the smoothness\footnote{Since $\cC$ is a convex set with non-empty interior, it follows from the regularity theory of the free boundary for the obstacle problem that $\partial\cC$ is smooth (see for instance \cite{Caffarelli-revisited}).
  However, if one does not want to rely on this result, it suffices to replace $(\sqrt{H})'$ with the right or left limit, respectively, of the derivative of $\sqrt{H}$, which always exists by the concavity of $\sqrt{H}$.}
of $\partial \cC$ implies that 
\begin{align} \label{eq:concavity_of_H_1}
(\sqrt{H})'(t) (t-s)   \leq \sqrt{H(t)} - \sqrt{H(s)} \leq (\sqrt{H})'(s) (t-s) \qquad \text{ for all }0 \leq s \leq t.
\end{align} 
In particular, since $H(0)\geq 0$,
\begin{align} \label{eq:concavity_of_H_2}
(\sqrt{H})'(t) \leq \frac{\sqrt{H(t)}}{t} \qquad  \text{ for all } t >0.
\end{align}
Furthermore, by the monotonicity of $u$ in the $x_3$-direction,
\begin{align} \label{eq:monotonicity_of_measure_of_sections}
0 \leq	H(s) \leq H(t) \qquad  \text{ for all } 0 \leq s \leq t.
\end{align}
Finally, from Lemma~\ref{lem:prop5.1} we infer that for every $\delta \in (0,1)$ there is $a(\delta)>0$ such that 
\begin{align} \label{eq:1+delta_estimate_for_H}
	H(t) \leq t^{1+\delta } \qquad  \text{ for all } t \geq a(\delta).
\end{align}
\textbf{Step 2.} \emph{The generalized Newtonian potential expansion.} \\
Let $V_\cC$ be the generalized Newtonian potential of the coincidence set $\cC$, cf. Definition \ref{NP}.
Thanks to Lemma~\ref{lem:sections_are_asymptotically_bounded_by_measure} and \eqref{eq:1+delta_estimate_for_H} it follows that, for $t > 0$  sufficiently large,
\begin{align}
	0 \leq \frac{1}{|t e^3 -y|} - \frac{1}{|y|} - \frac{t y_3}{|y|^3} \qquad \text{ for all } y \in \cR_t := \set { t - \sqrt{H(t)} < y_3 < t + \sqrt{H(t)} } \cap \cC.
\end{align}
Combining this with Proposition \ref{prop:Newton_potential_expansion}, Remark \ref{rem:coincidence_set_contains_ray} and Definition \ref{def:solution}(iv), we find that, for sufficiently large $t$,
\begin{align} \label{eq:estimate_of_tail_potential_outside_main_part}
	0 = u(t e^3) = V_\cC(t e^3) \geq \tilde{V}_{\cC}(t e^3) := \alpha_3 \int \limits_{\cC \setminus \cR_t} \bra { \frac{1}{|t e^3 -y|} - \frac{1}{|y|} - \frac{t y_3}{|y|^3} } \dx{y}.
\end{align}
\textbf{Step 3.} \emph{A one-dimensionalized version of $\tilde{V}_\cC$.} \\
The objective of this step is to replace the potential integral defining $\tilde{V}_\cC$ by a one-dimensional integral, up to a well-controlled error. To be more precise, we claim that there exists a constant $C$ such that, for all $t >0$ sufficiently large,
\begin{align} \label{eq:one-dimensionalitation_estimate}
\biggl| \int \limits_{(\cC \cap \{y_3 \geq a\}) \setminus \cR_t}   \bra { \frac{1}{|t e^3 -y|} - \frac{1}{|y|} - \frac{t y_3}{|y|^3} } \dx{y} - W(t) \biggr| \leq C( H(t) +{t}),
\end{align}
where 
\begin{align}
W(t):= \int \limits_{a}^{t- \sqrt{H(t)}}  \bra { \frac{1}{t-s} - \frac{1}{s} - \frac{t}{s^2}   } H(s) \dx{s} + \int \limits_{t + \sqrt{H(t)}}^\infty  \bra { \frac{1}{t-s} - \frac{1}{s} - \frac{t}{s^2}   } H(s) \dx{s},
\end{align}
$a := \max \{a_0, a(\delta) \}>0$ and $a_0$ is as in Lemma \ref{lem:sections_are_asymptotically_bounded_by_measure} and $a(\delta)$ is as in \eqref{eq:1+delta_estimate_for_H}. 

For the remainder of this step, fix a point $y=(y',s) \in \cC \setminus \cR_t$. Note that, by the definition of $\cR_t$, we have $|t-s| \geq \sqrt{H(t)}$.
 Combining Lemma \ref{lem:sections_are_asymptotically_bounded_by_measure}, \eqref{eq:concavity_of_H_1}, \eqref{eq:concavity_of_H_2}, and \eqref{eq:monotonicity_of_measure_of_sections},
 \begin{align}
|y'|^2 &\leq C H(s) \leq C H(t) \qquad \text{ for all }s \in (a, t-\sqrt{H(t)}) \label{eq:point_to_measure_1} \text{ and}\\
|y'|^2 &\leq C H(s) \leq C \biggl( \sqrt{H(t)} + \frac{\sqrt{H(t)}}{t}(s-t)  \biggr)^2 \\ & \leq C \biggl( H(t) + \frac{H(t)}{t^2} (s-t)^2 \biggr) \qquad \text{ for all } s \geq t + \sqrt{H(t)}. \label{eq:point_to_measure_2}
\end{align}
Let us now note that, by the mean value theorem, there is $\xi_{s,t} \in (0, |y'|^2)$ such that
\begin{multline} \label{eq:mean_vaiue_estimate}
\abs{ \frac{1}{\sqrt{(s-t)^2 + |y'|^2}} - \frac{1}{\sqrt{s^2 + |y'|^2}} - \frac{ts}{(s^2 + |y'|^2)^\frac{3}{2}} - \bra {  \frac{1}{|s-t|} - \frac{1}{s} - \frac{t}{s^2}    }       } \\
\leq \frac{1}{2} \abs{ \frac{1}{((s-t)^2 + \xi_{s,t})^\frac{3}{2}} - \frac{1}{(s^2 + \xi_{s,t})^\frac{3}{2} } -3 \frac{ts}{( s^2 + \xi_{s,t}  )^\frac{5}{2} }  } |y'|^2.
\end{multline}
For $s \in (a, t - \sqrt{H(t)})$, we can estimate the right-hand side above as
\begin{multline}
\frac{1}{2} \abs{ \frac{1}{((s-t)^2 + \xi_{s,t})^\frac{3}{2}} - \frac{1}{(s^2 + \xi_{s,t})^\frac{3}{2} } -3 \frac{ts}{( s^2 + \xi_{s,t}  )^\frac{5}{2} }  } |y'|^2 \\
\leq \frac{1}{2} \abs{ \frac{1}{({t-s})^3} + \frac{1}{s^3} + \frac{3 t s}{s^5}   } |y'|^2 
\leq 2 \bra { \frac{1}{({t-s})^3} + \frac{t}{s^4} } |y'|^2,
\end{multline}
so \eqref{eq:mean_vaiue_estimate} implies that
\begin{multline} \label{eq:mean_vaiue_estimate small}
\abs{ \frac{1}{\sqrt{(s-t)^2 + |y'|^2}} - \frac{1}{\sqrt{s^2 + |y'|^2}} - \frac{ts}{(s^2 + |y'|^2)^\frac{3}{2}} - \bra {  \frac{1}{|s-t|} - \frac{1}{s} - \frac{t}{s^2}    }       } \\
\leq  2 \bra { \frac{1}{({t-s})^3} + \frac{t}{s^4} } |y'|^2 \qquad \text{ for all } s \in (a, t - \sqrt{H(t)}).
\end{multline}
Now, combining first  \eqref{eq:mean_vaiue_estimate small},  \eqref{eq:point_to_measure_1}, \eqref{eq:1+delta_estimate_for_H} and \eqref{eq:monotonicity_of_measure_of_sections}, we see that
\begin{align}
&\Biggl| \int \limits_{\cC \cap \{a \leq y_3 \leq t-\sqrt{H(t)} \} } \bra{ \frac{1}{|t e^3-y|} - \frac{1}{|y|} - \frac{ty_3}{|y|^3}  } \dx{y} - \int \limits_{a}^{t- \sqrt{H(t)}}   \bra{ \frac{1}{t-s} - \frac{1}{2} - \frac{t}{s^2}  } H(s) \dx{s}    \Biggr| \\
&\leq C \int \limits_{a}^{t- \sqrt{H(t)}} \bra { \frac{1}{(t-s)^3} + \frac{t}{s^4} }  H(s)^2 \dx{s}\\
& \leq CH(t)^2  \int \limits_{a}^{t-\sqrt{H(t)}}  \frac{1}{(t-s)^3} \dx{s} +C t \int \limits_{a}^{t- \sqrt{H(t)}} \frac{s^{2+2\delta}}{s^4} \dx{s}    \leq C (H(t) +t). \label{eq:one-dimensionalization_estimate_1}
\end{align}
On the other hand, for $s \geq t + \sqrt{H(t)}$, 
we can apply the Taylor formula  $f(1) = f(0)+ f'(0) + \int_0^1 (1-\tau) f''(\tau) \dx{\tau}$ with $f(\tau) := ((s-t \tau)^2 + \xi_{s,t})^{-\frac{3}{2}}$ 
to get
\begin{align}
\abs{ \frac{1}{((s-t)^2 + \xi_{s,t})^\frac{3}{2}} - \frac{1}{(s^2 + \xi_{s,t})^\frac{3}{2} } -3 \frac{ts}{( s^2 + \xi_{s,t}  )^\frac{5}{2} }  } 
&\leq C \frac{t^2}{((s-t)^2 + \xi_{s,t})^\frac{5}{2}} \leq C\frac{t^2}{(s-t)^5},
\end{align}
and, by a direct estimate, we can estimate
\begin{align}
\abs{ \frac{1}{((s-t)^2 + \xi_{s,t})^\frac{3}{2}} - \frac{1}{(s^2 + \xi_{s,t})^\frac{3}{2} } -3 \frac{ts}{( s^2 + \xi_{s,t}  )^\frac{5}{2} }  }  &\leq \abs{ \frac{1}{(s-t)^3} + \frac{1}{s^3} + \frac{3 t}{s^4}   }  \leq C\frac{1}{(s-t)^3}.
\end{align}
Combining the last two inequalities with \eqref{eq:mean_vaiue_estimate} we get
\begin{multline} \label{eq:mean_value_estimate_2}
\abs{ \frac{1}{\sqrt{(s-t)^2 + |y'|^2}} - \frac{1}{\sqrt{s^2 + |y'|^2}} - \frac{ts}{(s^2 + |y'|^2)^\frac{3}{2}} - \bra {  \frac{1}{|s-t|} - \frac{1}{s} - \frac{t}{s^2}   }    }  \\
\leq C \min \set { \frac{1}{(s-t)^3} , \frac{t^2}{(s-t)^5} } |y'|^2\qquad \text{ for all }s \geq t + \sqrt{H(t)}.
\end{multline}
Moreover, using \eqref{eq:monotonicity_of_measure_of_sections}, \eqref{eq:point_to_measure_2}, \eqref{eq:mean_value_estimate_2} and \eqref{eq:1+delta_estimate_for_H} we estimate
\begin{align}
&\Biggl| \int \limits_{\cC \cap \{ y_3 \geq t+\sqrt{H(t)} \} } \bra{ \frac{1}{|t e^3-y|} - \frac{1}{|y|} - \frac{ty_3}{|y|^3}  } \dx{y} - \int \limits_{t+ \sqrt{H(t)}}^\infty   \bra{ \frac{1}{t-s} - \frac{1}{2} - \frac{t}{s^2}  } H(s) \dx{s}  \Biggr| \\ 
&\leq C \int \limits_{t + \sqrt{H(t)}}^\infty \min \set { \frac{1}{(s-t)^3} , \frac{t^2}{(s-t)^5} } H(s)^2 \dx{s} \\
&\leq C \int \limits_{t + \sqrt{H(t)}}^{2t} \frac{1}{(s-t)^3} H(s)^2 \dx{s} + C\int \limits_{2t}^\infty \frac{t^2}{(s-t)^5} H(s)^2 \dx{s}   \\
&\leq CH(2t)^2 \int \limits_{t+\sqrt{H(t)}}^{2t}  \frac{1}{(s-t)^3} \dx{s} + CH(t) \int \limits_{2t}^\infty \bra{\frac{t^2}{(s-t)^5} + \frac{1}{(s-t)^3}  } H(s) \dx{s}   \\
&\leq C \bra{H(t) + \frac{H(t)}{t^2}(2t-t)^2   }^2  \int \limits_{t+\sqrt{H(t)}}^{2t}  \frac{1}{(s-t)^3} \dx{s}  + CH(t) \int \limits_{2t}^\infty \bra{\frac{s^2}{\bra{\frac{s}{2}}^5} +\frac{1}{\bra{\frac{s}{2}}^3}   } s^{1+\delta} \dx{s}   \leq C H(t).
\end{align}
This bound, together with \eqref{eq:one-dimensionalization_estimate_1}, finishes the proof of \eqref{eq:one-dimensionalitation_estimate}.
\\
\textbf{Step 4.} \emph{Estimating $W$ from below.} \\
To simplify notation we set, for $t \geq a$,
\begin{align}
W_1(t) := \int \limits_{a}^{t- \sqrt{H(t)}} \bra{\frac{1}{t-s} - \frac{1}{s} - \frac{t}{s^2}  } H(s) \dx{s} , \qquad W_2(t) := \int \limits_{t + \sqrt{H(t)}}^\infty \bra{\frac{1}{s-t} - \frac{1}{s} - \frac{t}{s^2}  } H(s) \dx{s},
\end{align}
so that $W(t) = W_1(t)+W_2(t)$. 

To estimate $W_1$ from below, we split the integral so that the integrand in each part has a sign. More precisely, since
\begin{align}
\frac{1}{t-s} - \frac{1}{s} - \frac{t}{s^2} \leq 0 ~ \text{ for } s \in \pra{a, \frac{t}{\sqrt{2}} } \quad \text{ and } \quad \frac{1}{t-s} - \frac{1}{s} - \frac{t}{s^2}  \geq 0 ~ \text{ for } s \in \pra{\frac{t}{\sqrt{2}}, t- \sqrt{H(t)}  },
 \end{align}
we set $W_1 = W_{1,1} + W_{1,2}$ with
 \begin{align}
 W_{1,1}(t) :=  \int \limits_{a}^{t/\sqrt{2}} \bra{\frac{1}{t-s} - \frac{1}{s} - \frac{t}{s^2}  } H(s) \dx{s}, \qquad W_{1,2}(t) :=  \int \limits_{t/\sqrt{2}}^{t- \sqrt{H(t)}} \bra{\frac{1}{t-s} - \frac{1}{s} - \frac{t}{s^2}  } H(s) \dx{s}. 
 \end{align}
 We estimate $W_{1,1}$ by neglecting the first term, so to get
 \begin{align}
 W_{1,1}(t) \geq - \int \limits_{a}^{t/\sqrt{2}} \bra{\frac{1}{s}+ \frac{t}{s^2}  } H(s) \dx{s} \geq - \int \limits_{a}^t \frac{H(s)}{s} \dx{s} - t \int \limits_{a}^t \frac{H(s)}{s^2} \dx{s}.
 \end{align}
 To estimate $W_{1,2}$,
 using \eqref{eq:concavity_of_H_1} and \eqref{eq:monotonicity_of_measure_of_sections} as well as\eqref{eq:concavity_of_H_2} we obtain
 \begin{align}
 W_{1,2}(t) &= H(t)  \int \limits_{t/\sqrt{2}}^{t- \sqrt{H(t)}} \frac{1}{t-s} \dx{s} +  \int \limits_{t/\sqrt{2}}^{t- \sqrt{H(t)}} \frac{H(s) - H(t)}{t-s} \dx{s} -  \int \limits_{t/\sqrt{2}}^{t- \sqrt{H(t)}} \bra{\frac{1}{s} + \frac{t}{s^2}  } H(s) \dx{s} \\
 &\geq   H(t) \bra{\log(t) + \log\big(1-\tfrac{1}{\sqrt{2}}\big) } 
- \tfrac{1}{2} H(t) \log(H(t))  \\ &\quad-  \int \limits_{t/\sqrt{2}}^{t- \sqrt{H(t)}} \big(\sqrt{H(t)} + \sqrt{H(s)}\big) \frac{\sqrt{H(t)} - \sqrt{H(s)}}{t-s} \dx{s} -CH(t)\int \limits_{t/\sqrt{2}}^{t- \sqrt{H(t)}} \bra{\frac{1}{s} + \frac{t}{s^2}  }  \dx{s} \\
&\geq H(t) \log(t) - \tfrac{1}{2} H(t) \log(H(t))   -2 \sqrt{H(t)}  \int \limits_{t/\sqrt{2}}^{t- \sqrt{H(t)}} \frac{\sqrt{H(t)} - \sqrt{H(s)}}{t-s} \dx{s}     -CH(t) \\
&\geq H(t) \log(t) - \tfrac{1}{2} H(t) \log(H(t))   -2 \sqrt{H(t)}  \int \limits_{t/\sqrt{2}}^{t- \sqrt{H(t)}} (\sqrt{H})'(s) \dx{s}     -CH(t) \\
&\geq H(t) \log(t) - \tfrac{1}{2} H(t) \log(H(t))    -CH(t).
 \end{align}
 It remains to estimate $W_2(t)$. Since the integrand in $W_2$ is nonnegative,
 using \eqref{eq:monotonicity_of_measure_of_sections} we obtain that
 \begin{align}
 W_2(t) \geq H(t)   \int \limits_{t + \sqrt{H(t)}}^\infty \bra{\frac{1}{s-t} - \frac{1}{s} - \frac{t}{s^2}  }  \dx{s} \geq H(t) \log(t) - \tfrac{1}{2} H(t) \log(H(t)) - H(t).
 \end{align} 
Combining all estimates,  we find the lower bound
 \begin{align} \label{eq:estimate_of_W}
 W(t) \geq 2 H(t) \log(t) - H(t) \log(H(t)) - CH(t)  - \int \limits_{a}^t \frac{H(s)}{s} \dx{s} - t \int \limits_{a}^t \frac{H(s)}{s^2} \dx{s}.
 \end{align}
 \textbf{Step 5.} \emph{An integral inequality.} \\
 Combining \eqref{eq:estimate_of_tail_potential_outside_main_part}, \eqref{eq:one-dimensionalitation_estimate}, and \eqref{eq:estimate_of_W}, we deduce the existence of a constant $C$ such that, for sufficiently large $t$,
 \begin{align}
 C (H(t) + t) \geq 2 H(t) \log(t) - H(t) \log(H(t))  - \int \limits_{a}^t \frac{H(s)}{s} \dx{s} - t \int \limits_{a}^t \frac{H(s)}{s^2} \dx{s}
 \end{align}
which implies in particular that
 \begin{align}
 	C (H(t) + t) \geq & 2 (H(t) + t) \log(t) - (H(t)+t) \log(H(t) +t)\\
 	& - \int \limits_{a}^t \frac{H(s)+s}{s} \dx{s} -t \int \limits_{a}^t \frac{H(s)+s}{s^2} \dx{s}.
 \end{align}
 Hence, setting $\psi(t) := \tfrac{H(t)+t}{t}$, we find that 
 \begin{multline}
 C \psi(t) \geq \psi(t) \log(t) - \psi(t) \log(\psi(t)) - \frac{1}{t} \int \limits_{a}^t \psi(s) \dx{s} - \int \limits_{a}^t \frac{\psi(s)}{s} \dx{s} \\
 \geq -\psi(t) \log(\psi(t))- \frac{1}{t} \int \limits_{a}^t \psi(s) \dx{s} +\int \limits_{a}^t \frac{\psi(t)-\psi(s)}{s} \dx{s}\qquad \text{ for all $t$ sufficiently large}. \label{eq:integral_inequality_for_psi}
 \end{multline}
 Our goal in the following is to show that $\psi$ is bounded. To this end let us replace $\psi$ by the monotone function
 \begin{align}
 	\Psi(t) := \sup \limits_{s \in [a,t]} \psi(s).
 \end{align}
 In order to find an integral inequality for $\Psi$, given $t \geq a$ let $\tau=\tau(t) \in [a,t]$ be such that $\Psi(t) = \psi(\tau)$. Then, by the definition of $\Psi$,
 \begin{align} \label{eq:Psi_const_after_max}
 	\Psi(s) = \Psi(t)  \qquad \text{ for all }  s \in [\tau, t].
 \end{align}
Noticing that $\tfrac{1}{t} \int_{a}^t \Psi(s) \dx{s} \geq \tfrac{1}{\tau} \int_{a}^\tau \Psi(s) \dx{s}$ (since $\Psi$ is increasing) and that $\Psi(t) \geq \psi(t)$ by construction, we deduce that
 \begin{align}
 C \Psi(t) + \Psi(t) \log(\Psi(t)) + \frac{1}{t} \int \limits_{a}^t \Psi(s) \dx{s}  \geq C \psi(\tau) + \psi(\tau) \log(\psi(\tau)) + \frac{1}{\tau} \int \limits_{a}^\tau \psi(s) \dx{s}.
 \end{align}
Thus, since $\Psi(s) \geq \psi(s)$ for all $s \geq a$, it follows from \eqref{eq:integral_inequality_for_psi} and \eqref{eq:Psi_const_after_max} that 
 \begin{align}
 C \Psi(t) + \Psi(t) \log(\Psi(t)) + \frac{1}{t} \int \limits_{a}^t \Psi(s) \dx{s} &\geq \int \limits_{a}^\tau \frac{\psi(\tau) - \psi(s)}{s} \dx{s} =\int \limits_{a}^\tau \frac{\Psi(t) - \psi(s)}{s} \dx{s} \\
 &\geq \int \limits_{a}^\tau \frac{\Psi(t) - \Psi(s)}{s} \dx{s} = \int \limits_{a}^t \frac{\Psi(t) - \Psi(s)}{s} \dx{s}.
 \end{align}
 Since $\tfrac{1}{t} \int_{a}^t \Psi(s) \dx{s} \leq \Psi(t)$ (by the monotonicity of $\Psi$), we can simplify the relation above to conclude that
 \begin{align} \label{eq:integral_inequality_for Psi}
 C \Psi(t) + \Psi(t) \log(\Psi(t)) \geq \Psi(t) \log(t) - \int \limits_{a}^t \frac{\Psi(s)}{s} \dx{s}.
 \end{align}  
 \textbf{Step 6.} \emph{Switching to a differential inequality and comparison.} \\
Define $F(t) := \int_{a}^t \tfrac{\Psi(s)}{s} \dx{s}$. Then $F'(t) = \tfrac{\Psi(t)}{t} >0$ and \eqref{eq:integral_inequality_for Psi} becomes
 \begin{align} \label{eq:differential_inequality_for_F}
-C F'(t) - F'(t) \log(F'(t)) \leq \frac{F(t)}{t}.
 \end{align}
Since $H(t) \leq t^{1+\delta}$ (see \eqref{eq:1+delta_estimate_for_H}) it follows that $\psi(t) \leq t^\delta$. Therefore $\Psi(t) \leq t^\delta$, from which it follows that
 \begin{align} \label{eq:a-priori_estimates_for_F,F_prime}
F'(t) \leq t^{\delta-1} \quad \text{ and } \quad F(t) \leq \frac{1}{\delta} t^\delta.
 \end{align}
In particular, this yields 
$$
F'(t) \to 0\quad \text{ and }\quad 0> F'(t) \log(F'(t)) \to 0\qquad \text{ as }t \to \infty.
$$
 Note now that, for $\tau_0>0$ small enough, the function
 $$
 h:(0,\tau_0) \to (0,\infty),\qquad h(\tau) := -C \tau - \tau \log(\tau)=-C \tau+\tau|\log \tau|
 $$ 
is strictly increasing, invertible, and has a locally Lipschitz-continuous inverse. 
 This implies that, for sufficiently large $t_0>0$, the ordinary differential inequality \eqref{eq:differential_inequality_for_F} enjoys the comparison principle, i.e. if $G:(t_0,\infty) \to \infty$ satisfies
 \begin{align}
 	h(G'(t)) \geq \frac{G(t)}{t} \quad \text{ for all } t \geq t_0 \quad \text{ and } \quad F(t_0) \leq G(t_0)
 \end{align}
 then $F(t) \leq G(t)$ for all $t \geq t_0$. \\
 \textbf{Step 7.} \emph{Construction of a comparison solution}. \\
 Let $A>1, B>0$, and define $G_{A,B}: (0, \infty) \to \R$ as $G_{A,B}(t) := A \log(t) -B$. Then, for $B:= A(C + \log(A))$,
 \begin{align}
h(G'_{A,B}(t) ) & = -C \frac{A}{t} - \frac{A}{t} \log \bra{\frac{A}{t}  }  = - \frac{A}{t} \bra{C + \log(A) } + \frac{A \log(t)}{t} \\
& = \frac{A \log(t) -B}{t} = \frac{G_{A,B}(t)}{t}\qquad \text{ for all }t>0.
 \end{align}
Fix now $\delta \in (0,\tfrac{1}{2})$ and define $A:=t_0^{2\delta}$. Then, if $t_0$ is chosen sufficiently large, thanks to \eqref{eq:a-priori_estimates_for_F,F_prime} we get
 \begin{align}
G_{A,B}(t_0) = t_0^{2\delta} \log(t_0) - 2\delta t_0^{2\delta} \log(t_0) -C t_0^{2\delta} = \big((2-2\delta) \log(t_0) -C\big
) t_0^{2\delta} \geq F(t_0).
 \end{align}
 \textbf{Step 8.} \emph{Conclusion.} \\
By the comparison principle mentioned in Step 6, choosing $A$ and $B$ as in Step 7 we deduce that
 \begin{align} \label{eq:estimate_for_F}
 F(t) \leq G_{A,B}(t) \leq A \log(t)  \qquad \text{ for all }  t \geq t_0.
 \end{align} 
 Also, for all $0<x,y \ll 1$,
 \begin{align}
 y \geq h(x) = -C x + x |\log(x)| \quad \Longrightarrow \quad x \leq \frac{2y}{|\log(x)|} \leq \frac{2y}{|\log(y)|}=\frac{2y}{-\log(y)}.
 \end{align}
 Hence, recalling \eqref{eq:differential_inequality_for_F}, \eqref{eq:estimate_for_F} and \eqref{eq:a-priori_estimates_for_F,F_prime}, for $t \geq t_0$ we obtain
 \begin{align}
 F'(t) \leq 2 \frac{\frac{F(t)}{t}}{-\log \bra{\frac{F(t)}{t}  }} \leq 2 \frac{\frac{A \log(t)}{t}}{- \log \bra{\frac{A \log(t)}{t}  }} = \frac{2 A \log(t)}{ t \big( \log(t) - \log(\log(t)) - \log(A)\big) } \leq \frac{{4}A}{t},
 \end{align}
 provided that $t_0$ has been chosen sufficiently large.
 Recalling that $F'(t) = \tfrac{\Psi(t)}{t} \geq \tfrac{\psi(t)}{t}$ this implies that
 \begin{align}
 \frac{H(t)+t}{t} = \psi(t) \leq {4}A  \qquad \text{ for all }  t \geq t_0
 \end{align}
which concludes the proof.
\end{proof}

\section{Linear and almost-linear behavior of $V_{\cC}$} \label{sec:BMO-type_estimate}

In this section we prove that, 
 for $N\geq 4$, the generalized Newtonian potential $V_\cC$ can be written as the sum of a linear function and a correction with sublinear growth towards infinity.
 In contrast, for  $N=3$, 
the best one can show is the following {\em BMO-type property}:
on every large ball $B_R$ there exists an affine function $A_\cC^R$ whose slope grows like $\log R$ and whose average distance from $V_\cC$ is of order $R$.
 
We begin with the case $N\geq 4.$
\begin{lem}[Asymptotic growth of the Generalized Newtonian Potential
in dimension $N\geq 4$] 
\label{lem:improvement_of_growth_N4}
Let $N\geq 4$, let $u$ be an $x_N$-monotone solution in the sense of Definition \ref{def:solution}, and let $V_\cC$ be the generalized Newtonian potential (as defined in \eqref{eq:generalized Newtonian potential}) of the coincidence set $\cC$.
Then $V_{\cC}$ can be written as $V_{\cC}:= W_{\cC}-\ell_{\cC}$ with
$$
W_{\cC}(x):=\alpha_N\int_{\cC} \biggl(\frac{1}{\abs{x-y}^{N-2}} - \frac{1}{\abs{y}^{N-2}}\bigg)\dx{y},
\qquad \ell_\cC(x):=x\cdot \biggl(\alpha_N(N-2)\int_{\cC}\frac{y}{ \abs{y}^N}\dx{y}\biggr),
$$
where both integrals are well-defined.
Also, there exists a constant $C$ such that
\begin{align}
\label{eq:WA R}
\fint \limits_{B_R}  \abs { W_\cC(x)} \dx{x} \leq C R^{1/2}\qquad \text{ for all $R$ sufficiently large}.
\end{align}
\end{lem}
\begin{proof}
Since $V_{\cC}$ is well-defined, it suffices to show that $\ell_C$ is well-defined.
For this, we need to prove that $\frac{y}{ \abs{y}^N}\chi_{\cC} \in L^1(\R^N)$.

Thanks to Proposition~\ref{prop:C_is_contained_in_paraboloid}, and recalling that $N\geq 4$, as  $\frac1{|y|}\leq \frac1{y_N}$ we can estimate
\begin{align}
\int_{\cC} \frac{1}{|y|^{N-1}}\dx{y}&\leq \int_{\cC\cap \{y_N\leq a_0\}}\frac{1}{|y|^{N-1}}\dx{y}+\int_{\{y_N>a_0\}\cap \{|y'|^2 < \gamma_0 y_N\}}\frac{1}{|y|^{N-1}}\dx{y}\\
&\leq C+\int_{a_0}^\infty \cH^{N-1}(B_{\sqrt{\gamma_0 t}}')\frac1{t^{N-1}}\dx{t} \leq C\biggl(1+\int_{a_0}^\infty t^{(N-1)/2+(1-N)}\dx{t}\biggr) \\
&\leq C\biggl(1+\int_{a_0}^\infty t^{-3/2}\dx{t}\biggr)<\infty,
\end{align}
which proves that $\ell_\cC(x)$ is a well-defined linear function.

Now, to prove \eqref{eq:WA R}, we note that by the mean value theorem, 
\begin{equation}
\label{eq:TaylorW}
\biggl|\frac{1}{\abs{x-y}^{N-2}} - \frac{1}{\abs{y}^{N-2}}\biggr| \leq C\frac{|x|}{|y|^{N-1}}\qquad \text{for }|y|>2|x|.
\end{equation}
Hence, given $x \in B_R$ with sufficiently large $R$, we can write
\begin{align}
\abs { W_\cC(x)}&\leq  \alpha_N \int_{\mathcal C\cap \{0\leq y_N \leq a\}}\biggl|\frac{1}{\abs{x-y}^{N-2}} - \frac{1}{\abs{y}^{N-2}}\biggr|\dx{y}\\
&\qquad +\alpha_N \int_{\mathcal C\cap \{a\leq y_N \leq 2R\}}\biggl|\frac{1}{\abs{x-y}^{N-2}} - \frac{1}{\abs{y}^{N-2}}\biggr|\dx{y}\\
&\qquad+ \alpha_N \int_{\mathcal C\cap \{y_N \geq 2R\}}\biggl|\frac{1}{\abs{x-y}^{N-2}} - \frac{1}{\abs{y}^{N-2}}\biggr|\dx{y}=:J_1(x)+J_2(x)+J_3(x).
\end{align}
Thanks to \eqref{eq:TaylorW} and Proposition~\ref{prop:C_is_contained_in_paraboloid}(i), we can estimate
\begin{align}
J_3(x)\leq C|x|\int_{2R}^\infty \cH^{N-1}(\{|y'|^2<\gamma_0 t\}) \frac{1}{t^{N-1}}\dx{t}\leq CR\int_{2R}^\infty t^{-3/2}\dx{t}\leq CR^{1/2}.
\end{align}
Also, thanks to Proposition~\ref{prop:C_is_contained_in_paraboloid}(ii), it follows that $|J_1(x)|\leq C$.
Finally, for $J_2(x)$ we have 
\begin{align}
J_2(x)
&\leq \alpha_N \int_{\{|y'|^2 < \gamma_0 y_N\}\cap \{0\leq y_N\leq 2R\}}\frac{1}{\abs{x-y}^{N-2}}\dx{y}+C \int_{a_0}^{2R} \cH^{N-1}(B_{\sqrt{\gamma_0 t}}')\frac1{t^{N-2}}\dx{t}\\
& \leq \alpha_N \int_{\{|y'|^2 < 2\gamma_0 R\}\cap \{0\leq y_N\leq 2R\}}\frac{1}{\abs{x-y}^{N-2}}\dx{y}+CR^{1/2}.
\end{align}
Now, for the first term on the right hand side,
if we write $x=Rz$ with $z \in B_1$ and we perform the change of variables $y\mapsto Ry$, we see that
\begin{multline}\label{eq:x to z}
\int_{\{|y'|^2 < 2\gamma_0 R\}\cap \{0\leq y_N\leq 2R\}}\frac{1}{\abs{x-y}^{N-2}}\dx{y}=R^N\int_{\left\{|y'|^2<\frac{2\gamma_0}{R}\right\}\cap \{0\leq y_N\leq 2\}}
\frac{1}{\abs{x-Ry}^{N-2}} \dx{y}\\
=R^2\int_{\left\{|y'|^2<\frac{2\gamma_0}{R}\right\}\cap \{0\leq y_N\leq 2\}}\frac{1}{\abs{z-y}^{N-2}} \dx{y}.
\end{multline}
Hence, combining all these bounds, we obtain that 
\begin{align}
\fint \limits_{B_R} |W_\cC(x)| \dx{x}&\leq CR^{1/2}+  \fint \limits_{B_R} J_{{2}}(x)\dx{x}\\
&\leq CR^{1/2}+ \alpha_N R^2 \fint \limits_{B_1}\int_{\left\{|y'|^2\leq \frac{2\gamma_0}{R}\right\}\cap \{0\leq y_N\leq 2\}}\frac{1}{\abs{z-y}^{N-2}} \dx{y} \dx{z}.
\end{align}
Concerning the last integral we observe that
\begin{equation}
\label{eq:int G on B1}
\int_{B_1} \frac{1}{|z-y|^{N-2}}\dx{z}\leq C\qquad \text{ for all }y \in \R^N.
\end{equation}
Thus, since $\left|\left\{|y'|^2<\frac{2\gamma_0}{R}\right\}\cap \{0\leq y_N\leq 2\}  \right| \leq \frac{C}{R^{3/2}}$ for $N\geq 4$, it follows from  Fubini's Theorem that
\begin{align}
\fint \limits_{B_R}  \abs { W_\cC(x)    } \dx{x} 
&\leq CR^{1/2}+ \alpha_N R^2 \int_{\left\{|y'|^2<\frac{2\gamma_0}{R}\right\}\cap \{0\leq y_N\leq 2\} } \fint \limits_{B_1} \frac{1}{|z-y|^{N-2}}\dx{z} \dx{y}\\
&\leq CR^{1/2}+ CR^2 \left|\left\{|y'|^2<\frac{2\gamma_0}{R}\right\}\cap \{0\leq y_N\leq 2\}  \right| \leq CR^{1/2}.
\end{align}
\end{proof}

We now focus on the three-dimensional case.

\begin{lem}[BMO-type estimate in dimension $N=3$] \label{lem:improvement_of_growth_by_subtracting_affine_linear_function}
Let $N=3$, let $u$ be an $x_N$-monotone solution in the sense of Definition \ref{def:solution}, and let $V_\cC$ be the generalized Newtonian potential (as defined in Definition \ref{NP}) of the coincidence set $\cC$.
Then, for each $R>0$ the affine function $A^R_\cC$ given by
\begin{align} \label{eq:definition_of_A_R}
A^R_\cC(x) := \alpha_3 \int \limits_\cC \bra { - \frac{1}{|y|} + \frac{1}{|R e^3 +y|} - \frac{x \cdot y}{|y|^3} + \frac{(R e^3 +y) \cdot (x+ R e^3)}{|R e^3 +y|^3}              }	\dx{y}
\end{align}
is well-defined.
Also, there exists a constant $C$ such that
\begin{align}
\label{eq:VA R}
\fint \limits_{B_R}  \abs { \bV_\cC(x) -A^R_\cC(x)    } \dx{x} \leq C R\qquad \text{ for all $R$ sufficiently large}.
\end{align}
\end{lem}

\begin{proof}
We split the proof into two steps.\\
\textbf{Step 1.} \emph{$A^R_\cC$ is well-defined and affine.} 	\\
Set 
\begin{equation}
\label{eq:ell R}
a^R(x,y):=- \frac{1}{|y|} + \frac{1}{|R e^3 +y|} - \frac{x \cdot y}{|y|^3} + \frac{(R e^3 +y) \cdot (x+ R e^3)}{|R e^3 +y|^3}
\end{equation}
and write
$$
\int_{\cC}|a^R(x,y)|\dx{y}\leq \int_{\mathcal C\cap \{y_3 \leq 2R\}}|a^R(x,y)|\dx{y}+\int_{\mathcal C\cap \{y_3 \geq 2R\}}|a^R(x,y)|\dx{y}=:I_1+I_2.
$$
Since $a^R(x,\cdot) \in L^1_{\rm loc}(\R^3)$
and ${\mathcal C\cap \{y_3 \leq 2R\}}$ is bounded (thanks to Lemma~\ref{lem:prop5.1}), it follows that $I_1 \leq C_R$ for some constant depending on $R$. To estimate $I_2$, by a Taylor expansion as well as the mean value theorem we have that
\begin{align}
|a^R(x,y)|&\leq 
\abs{ \frac{1}{|R e^3 +y|} - \frac{1}{|y|} +\frac{(R e^3) \cdot y}{|y|^3} }
+ \frac{|R e^3 \cdot (x+ R e^3)|}{|R e^3 + y|^3}\\ \label{eq:Taylor ell}
&\quad + \abs { (R e^3 +x) \cdot y \bra {  \frac{1}{|y|^3} - \frac{1}{|R e^3 +y|^3}   }  } \\
&\leq C\frac{R^2}{ \abs{y}^{3}}+C\frac{R(R+|x|)}{ \abs{y}^{3}}+ C\frac{R(R+|x|)}{ \abs{y}^{3}}\qquad \text{for }|y| > 2R.
\end{align}
Hence, recalling Proposition~\ref{prop:C_is_contained_in_paraboloid}\eqref{item:growth_of_coincidence_set}, we can estimate
\begin{multline}
\label{eq:I2 ell}
I_2\leq CR(R+|x|)\int_{\mathcal C\cap \{y_3 \geq 2R\}} \frac{1}{|y|^3}\dx{y}
\leq CR(R+|x|)\int_{2R}^\infty \cH^2(\{|{y'}|^2<\gamma_0 t\}) \frac{1}{t^3}\dx{t}\\
\leq CR(R+|x|)\int_{2R}^\infty \frac{1}{t^2}\dx{t}\leq C(R+|x|),
\end{multline}
which proves that $A^R_\cC$ is well-defined.

Observe now that the integrand $a^R$ in the definition of $A^R_\cC$ is integrable (as we have just seen) and differentiable in $x$.
Also, for each $y \in \cC$ and $b:=\max\{a_0,R\}$ it holds that
\begin{align}
\abs { \nabla_x a^R(x,y) } &= \abs { \frac{y}{|y|^3} - \frac{R e^3 +y}{|R e^3 +y|^3}     } \\
&\leq  \frac{2}{|y|^{2}} \chi_{\cC \cap \set {y_3 \leq b }}(y) +        \biggl| \int \limits_0^1 \frac{\dx{}}{\dx{s}} \bra { \frac{s R e^3 +y}{|s R e^3 +y|^3}    } \dx{s}     \biggr|\chi_{\cC \cap \set {y_3 \geq b }}(y)  \label{eq:A3} \\
 &\leq C \bra { \frac1{|y|^{2}}  \chi_{\cC \cap \set {y_3 \leq b} }(y) + \frac{R}{|y|^3} \chi_{\cC \cap \set {y_3 \geq b}}(y)      }.
\end{align}
Since the right-hand side is  integrable in $\R^3$ (again, thanks to Proposition~\ref{prop:C_is_contained_in_paraboloid}), it follows from dominated convergence that
\begin{align}
\label{eq:DAR}
\nabla A^R_\cC(x) =\alpha_3\int \limits_\cC \nabla_x a^R(x,y) \dx{y}= \alpha_3 \int \limits_\cC \bra { - \frac{y}{|y|^3} + \frac{R e^3 +y}{|R e^3 +y|^3}      } \dx{y},
\end{align}
which is constant in $\R^3$. This proves that  $A^R_\cC$ is an affine function.
\\
\textbf{Step 2.} \emph{Proof of \eqref{eq:VA R}.} 
Recall the definition of $G(x,y)$ in Definition~\ref{NP},
which implies that
\begin{align}
\abs { \bV_\cC(x) -A^R_\cC(x)    } & 
\leq  \alpha_3 \int_{\mathcal C\cap \{0\leq y_3 \leq 2R\}}|G(x,y)-a^R(x,y)|\dx{y}\\
&+ \alpha_3 \int_{\mathcal C\cap \{y_3 \geq 2R\}}|G(x,y)-a^R(x,y)|\dx{y}=:J_1(x)+J_2(x).
\end{align}
Recalling \eqref{eq:Taylor G} and \eqref{eq:Taylor ell}, for $x \in B_R$ with $R\geq a_0$ we can estimate
\begin{align}
J_2(x)\leq CR^2\int_{2R}^\infty \cH^2(\{|y'|^2<\gamma_0 t\}) \frac{1}{t^3}\dx{t}\leq CR
\label{eq:bound J2 R}
\end{align}
(cp. \eqref{eq:I2 ell}). For $J_1(x)$, thanks to Proposition~\ref{prop:C_is_contained_in_paraboloid} it follows that, for $R$ sufficiently large,
$$
\mathcal C\cap \{y_3 \leq 2R\}\subset \left\{|y'|^2<{2\gamma_0 R}\right\}\times \{0\leq y_3\leq 2R\}.
$$
Thus, since
$$
G(Rx,Ry)-a^R(Rx,Ry)=\frac{1}{R}[G(x,y)-a^1(x,y)],
$$
if we write $x=Rz$ with $z \in B_1$ and perform the change of variables $y\mapsto Ry$, we see that
\begin{align}
\int_{\mathcal C\cap \{y_3 \leq 2R\}}|G(x,y)-a^R(x,y)|\dx{y}&\leq \int_{\left\{|y'|^2<{2\gamma_0 R}\right\}\cap \{0\leq y_3 \leq 2R\}}|G(x,y)-a^R(x,y)|\dx{y}\\
&=R^2\int_{\left\{|y'|^2<\frac{2\gamma_0}{R}\right\}\cap \{0\leq y_3 \leq 2\}}|G(z,y)-a^1(z,y)|\dx{y}
\end{align}
(cp. \eqref{eq:x to z}).
Altogether
we proved that 
\begin{multline}
\fint \limits_{B_R}  \abs { \bV_\cC(x) -A^R_\cC(x)    } \dx{x} \leq \fint \limits_{B_R} \big(J_1(x)+J_2(x)\big)\dx{x}\leq CR+  \fint \limits_{B_R} J_1(x)\dx{x}\\
\leq CR+ \alpha_3 R^2 \fint \limits_{B_1} \int_{\left\{|y'|^2<\frac{2\gamma_0}{R}\right\}\cap \{0\leq y_3 \leq 2\}}|G(z,y)-a^1(z,y)|\dx{y}\dx{z}.\label{eq:J12 z}
\end{multline}
Concerning the last integral we observe that, on the domain of integration, $|e^3+y| \geq 1$ and $|z+e^3|\leq 3$. Hence, we can estimate
\begin{equation}
\label{eq:G ell}
|G(z,y)-a^1(z,y)| =\left| \frac{1}{|z-y|}-\frac{1}{|e^3+y|}-\frac{(e^3+y)\cdot(z+e^3)}{|e^3+y|^3}\right|\leq \frac{1}{|z-y|}+4.
\end{equation}
Thus, since $\left|\left\{|y'|^2<\frac{2\gamma_0}{R}\right\}\cap \{0\leq y_3 \leq 2\}\right| \leq \frac{C}R$, it follows from  Fubini's Theorem and \eqref{eq:int G on B1} that
\begin{align}
\fint \limits_{B_R}  \abs { \bV_\cC(x) -A^R_\cC(x)    } \dx{x} 
& \leq CR+ \alpha_3  R^2 \int_{\left\{|y'|^2<\frac{2\gamma_0}{R}\right\}\cap \{0\leq y_3 \leq 2\}}\fint \limits_{B_1} \left(\frac{1}{|z-y|}+4\right)\dx{z} \dx{y}\\
& \leq CR+ CR^2 \left|\left\{|y'|^2<\frac{2\gamma_0}{R}\right\}\cap \{0\leq y_3 \leq 2\}\right|\leq CR, \label{eq:J12 final}
\end{align}
as desired.
\end{proof}

For later purposes, we will also need the next result.

\begin{lem}[Growth of $A^R_\cC$] \label{lem:DA}
Let $N=3$, let $u$ be an $x_N$-monotone solution in the sense of Definition \ref{def:solution}, and let $A^R_\cC$ be as in Lemma~\ref{lem:improvement_of_growth_by_subtracting_affine_linear_function}.
Then there exists a constant $C$ such that
\begin{align}
\label{eq:DA}
|\nabla' A^R_\cC |\leq C,\quad |\partial_3 A^R_\cC |\leq C\log R,\quad |A_\cC^R(0)|\leq CR \qquad \text{ for all $R$ sufficiently large}.
\end{align}
\end{lem}
\begin{proof}
For the first bound we note that, for $y \in \mathcal C$,
\begin{align}
\abs { \nabla_x' a^R(x,y) } = \abs { \frac{y'}{|y|^3} - \frac{y'}{|R e^3 +y|^3}     } \leq  2 \frac{|y'|}{|y|^3},\label{eq:D'ell}
\end{align}
where $a^R$ is the function defined in \eqref{eq:ell R}.
Thanks to Proposition~\ref{prop:C_is_contained_in_paraboloid},
\begin{align}
\int_\cC\abs { \nabla_{x'} a^R(x,y) } \dx{y}&\leq 2\int_{\cC \cap \set {y_3 \leq a_0} }  \frac1{|y|^{2}}  \dx{y}
+2\int_{\cC \cap \set {{y_3 \geq a_0}} }  \frac{|y'|}{|y|^3} \dx{y}\\
 &\leq C+2\int_{a_0}^\infty \cH^2(\{|y'|^2<\gamma_0 t\}) \frac{(\gamma_0 t)^{1/2}}{t^3}\dx{t}
 \leq C+C\int_{a_0}^\infty \frac{1}{t^{3/2}}\dx{t}\leq C.\label{eq:D'A}
\end{align}
Recalling \eqref{eq:DAR}, this proves that  $|\nabla' A^R_\cC |\leq C$.

For the second bound, we apply \eqref{eq:A3} and Proposition~\ref{prop:C_is_contained_in_paraboloid} to obtain that, for $R \geq a_0$,
\begin{align}
\int_\cC\abs { \partial_{x_3} a^R(x,y) } \dx{y}&\leq C\int_{\cC \cap \set {y_3 \leq a_0} }  \frac{1}{|y|^2}  \dx{y}
+C\int_{\cC \cap \set {{a_0} \leq y_3 \leq R} } \frac{1}{|y|^2}  \dx{y}
 + C\int_{\cC \cap \set {y_3 \geq R} } \frac{R}{|y|^3} \dx{y}\\
 &\leq C+C\int_{a_0}^R \cH^2(\{|y'|^2<\gamma_0 t\}) \frac{1}{t^2}\dx{t}+CR\int_R^\infty \cH^2(\{|y'|^2<\gamma_0 t\}) \frac{1}{t^3}\dx{t} \\
 &\leq C+C\int_{a_0}^R \frac{1}{t}\dx{t}+CR\int_R^\infty \frac{1}{t^2}\dx{t} \leq C\log R.
\end{align}
Thus, $|\partial_3 A^R_\cC |\leq C\log R$.

Finally, in order to estimate $A_\cC^R(0)$, recalling \eqref{eq:ell R} we write
$$
\int_{\cC}|a^R(0,y)|\dx{y}\leq \int_{\mathcal C\cap \{y_3 \leq 2R\}}|a^R(0,y)|\dx{y}+\int_{\mathcal C\cap \{y_3 \geq 2R\}}|a^R(0,y)|\dx{y}=:J_1+J_2.
$$
Using \eqref{eq:I2 ell}, we immediately get $J_2\leq CR$.
Concerning $J_1$, for $y \in \mathcal C\cap \{y_3 \leq 2R\}$ we can estimate 
\begin{align}
\label{eq:ell R J1}
|a^R(0,y)|\leq \frac{2}{|y|}+\frac{R}{R^2+|y|^2}\leq \frac{3}{|y|},
\end{align}
hence
\begin{align}
J_1 & \leq 3\int_{\cC \cap \set {y_3 \leq a_0} }  \frac{1}{|y|}  \dx{y}
+C\int_{\cC \cap \set {a_0 \leq y_3 \leq 2R} } \frac{1}{|y|}  \dx{y}\\
 & \leq C+3\int_{a_0}^R \cH^2(\{|y'|^2<\gamma_0 t\}) \frac{1}{t}\dx{t} \leq C+C\int_{a_0}^R \dx{t}\leq CR,
 \label{eq:ell R J1 2}
\end{align}
concluding the proof.

\end{proof}

\section{Constructing matching paraboloid solutions} \label{sec:constructing_the_correct_ellipsoids}

In this section we construct matching paraboloid solutions (i.e. solutions that have paraboloids as coincidence sets).

More precisely, given $N \geq 3$, we begin by constructing a one-parameter family of paraboloid solutions with the same asymptotics of second order at infinity as the solution $u$. For $N\geq 4$ we find one fixed paraboloid solution $u_P$ such that $u-u_P$ grows sublinearly at infinity.
The critical dimension $N=3$, however, requires a more subtle approach:
here we construct first 
for each $R$ sufficiently large a paraboloid solution $u_{P_R}$ such that $\sup_{B_R}|u-u_{P_R}|\le CR$ where the constant $C$ does not depend on $R$.

\begin{lem}[Existence of paraboloid solutions with prescribed asymptotic behavior at infinity] 
\label{lem:matching_the_quadratic_part}
Given $N\geq 3$, let 
$p = p(x')$ be a homogeneous quadratic polynomial as in Definition~\ref{def:solution}\eqref{PDE_asymptotics}.
Then there exists a (unique) ellipsoid 
$$
E' := \biggl\{ y' \in \R^{N-1} : \frac{y_1^2}{a_1^2} + \ldots+\frac{y_{N-1}^2}{a_{N-1}^2} \leq 1   \biggr\}\subset \R^{N-1}\quad\text{ with $a_i >0$ for $i=1,\ldots,N-1$,}
$$
such that the following holds.

Define the paraboloid $P=P_{E'}:=\{(y',y_N) \in \R^{N-1}\times [0,\infty): y' \in \sqrt{y_N} E'   \}$.
Then there exists a global solution $u_P $ with $P$ as coincidence set and $p$ as quadratic blow-down limit, i.e. 
\begin{align}
\Delta u_{P} = \chi_{\{u_{P} >0\}} ,\quad  u_{P} \geq 0 \qquad \text{ in } \R^N, \quad \{u_{P } = 0\} = P \text{ and}\\
 \frac{u_{P}(rx)}{r^2} \to p(x') \quad \text{ in } C_{\loc}^{1,\alpha}(\R^N) \text{ as } r \to \infty,\quad \alpha\in(0,1).
\end{align} 
Also, for each $\gamma>0$, the function $u_{\gamma P}(x):={\gamma^2 u_P\big(\frac{x}{\gamma}\big)}$ is a global solution with $\gamma P$ as coincidence set and $p$ as blow-down limit, and it satisfies the potential expansion
\begin{align} \label{eq:potential_expansion_of_quadratic_matching-rescaled}
u_{{\gamma P}}  = p + \bV_{{\gamma P}} \quad \text{ in } \R^N.
\end{align}
\end{lem}

\begin{proof}We split the proof into three steps.\\
\textbf{Step 1.} \emph{Construction of a suitable sequence of ellipsoids.} \\
Let $K_d$ denote the fundamental solution of the Laplace operator in $\R^{d}$, namely,
\begin{equation}
\label{eq:Green}
K_d(z):=\left\{
\begin{array}{ll}
{-}\frac{1}{2\pi}\log|z|&\text{if }d=2,\\
\frac{1}{d(d-2)|B_1|}\frac{1}{|z|^{d-2}}&\text{if }d\geq 3.
\end{array}
\right.
\end{equation}
Given $p$ as in the statement, it follows from  \cite[Equation (5.4)]{DiBenedettoFriedman} that there exists a unique ellipsoid  $E' := \left\{ y' \in \R^{N-1} : \frac{y_1^2}{a_1^2} + \ldots+\frac{y_{N-1}^2}{a_{N-1}^2} \leq 1   \right\}\subset \R^{N-1}$ , with $a_i>0$, such that
\begin{align}
	V^{NP}_{E'}(x') = V^{NP}_{E'}(0) -p(x') \qquad \text{ for all }  x' \in E',
\end{align} 
where $V^{NP}_{E'}$ denotes the classical $(N-1)$-dimensional Newtonian potential of $E'$, i.e.
\begin{align}
	V^{NP}_{E'}(x') := \int \limits_{E'} K_{N-1}(x'-y') \dx{y'} \qquad \text{ for all }  x' \in \R^{N-1}.
\end{align}
Set
\begin{align} \label{eq:def_of_u_E'}
	u'_{E'}(x') := p(x') - V^{NP}_{E'}(0) +V^{NP}_{E'}(x')   \qquad \text{ for all }  x' \in \R^{N-1}.
\end{align}
Then, it follows from \cite[Theorem II]{CaffarelliKarpShahgolian_Annals_2000} that $u'_{E'}$ is a nonnegative global solution to the obstacle problem in $\R^{N-1}$
satisfying $\{ u_{E'}' =0\} = E'$.

We now complete $E'$ to an $N$-dimensional ellipsoid approximating a paraboloid in the following way: for each $n \in \N$, set
\begin{align}
\tilde{E}^n := \biggl\{  x \in \R^N: x' \in \sqrt{\frac{n}2- \frac{x_N^2}{2n}} E'\biggr\} = \biggl\{ x \in \R^N : \sum \limits_{j=1}^{N-1} \frac{2x_j^2}{a_j^2 n} + \frac{x_N^2}{n^2} \leq 1    \biggr\}.
\end{align}
From \cite[Equation (5.3)]{DiBenedettoFriedman} we infer that, for each $n \in \N$, there is a homogeneous quadratic polynomial $q^n$ such that $\Delta q^n =1$ and 
\begin{align} \label{eq:NP_of_tilde_E_n}
	V^{NP}_{\tilde{E}^n}(x) = V_{\tilde{E}^n}^{NP}(0) - q^n(x) \qquad \text{ for all }  x \in \tilde{E}^n,
\end{align}
where 
\begin{align}
	V^{NP}_{\tilde{E}^n}(x) := \int \limits_{\tilde{E}^n} K_{N}(|x-y|) \dx{y} \qquad \text{ for all }  x \in \R^{N}.
\end{align}
Let us now translate the ellipsoids $\tilde{E}^n$ such that they all touch the origin:
\begin{align} \label{eq:def_of_E_n}
	E^n := \tilde{E}^n + n e^N &=\biggl\{ x \in \R^N : \sum \limits_{j=1}^{N-1} \frac{2x_j^2}{a_j^2 n} + \frac{(x_N-n)^2}{n^2} \leq 1    \biggr\}\\
	&=\biggl\{ x \in \R^N : \sum \limits_{j=1}^{N-1} \frac{2x_j^2}{a_j^2} + \frac{x_N^2}{n} \leq 2x_N   \biggr\}.
\end{align}
Then, recalling \eqref{eq:NP_of_tilde_E_n}, for all $x \in E^n$ we have
\begin{align}
V^{NP}_{E^n}(x)&= V^{NP}_{\tilde{E}^n + n e^N}(x) = V_{\tilde{E}^n}^{NP}(x-n e^N)\\
  &= V^{NP}_{\tilde{E}^n}({0}) - q^n(x-n e^N) = V^{NP}_{E^n}({ne^N}) - q^n(x-n e^N).
\end{align}
\textbf{Step 2.} \emph{Switching to the obstacle problem and passing to the limit.} \\
Let us now set, for all $n \in \N$,
\begin{align}
	u_{E^n}(x) := q^n(x-n e^N) - V^{NP}_{E^n}({ne^N}) + V^{NP}_{E^n}(x)  \qquad \text{ for all }  x\in \R^N.
\end{align}
As before,  \cite[Theorem II]{CaffarelliKarpShahgolian_Annals_2000} guaranteess that $u_{E^n}$ is a non-negative global solution to the obstacle problem satisfying $\{u_{E^n} =0\} = E^n$. \\
Since $0 \in E^n = \{u_{E^n} =0\}$ we deduce from Lemma~\ref{lem:compact C1}\eqref{item:compact} that, 
passing if necessary to a subsequence,
\begin{align}
u_{E^n} \to {u^\star} \quad \text{ in } C^{1,\alpha}_{\loc}(\R^N)\quad \text{and}\quad
\chi_{\{u_{E^n}=0\}} \to \chi_{\{{u^\star} =0\}}\quad\text{a.e.}\qquad \text{ as } n \to \infty.
\end{align}
On the other hand, by construction (cf. \eqref{eq:def_of_E_n}),
\begin{align}
\chi_{E^n} \to \chi_{P} \quad \text{ a.e. in } \R^N \text{ as } n \to \infty , \qquad \text{ where } P:= \biggl\{\sum \limits_{j=1}^{N-1} \frac{x_j^2}{a_j^2} \leq  x_N   \biggr\}=\left\{ x' \in \sqrt{x_N} E'   \right\},
\end{align}
and therefore $\{{u^\star} =0\} =P$. 
\\
\textbf{Step 3.} \emph{Identification of the blow-down limit of $u^\star$ and conclusion.}
\\
Let us define the following sequence of rescalings $(u^\star_k)_{k \in \N}$:
\begin{align}
	u^\star_k(x) := \frac{u^\star(r_k x + x^k)}{r_k^2}  \quad \text{ for }  x \in \R^N,\qquad \text{where }x^k  := {(0,k)}\text{ and }r_k := \sqrt{k}.
\end{align}
Since $0 \in \{ u^\star_k =0\}$, using Lemma~\ref{lem:compact C1}\eqref{item:compact} once more we deduce that, passing if necessary to a subsequence,
\begin{align}
u^\star_k \to u^\star_0 \quad \text{ in } C_{\loc}^{1,\alpha}(\R^N)  \quad  \text{ as } k \to \infty,
\end{align}
where $u^\star_0$ is a non-negative global solution to the obstacle problem.
Also, arguing as in Step 2, we see that the coincidence sets of $u^\star_k$ converge to $E' \times \R$, hence $\{ u^\star_0 =0\} = E' \times \R$. 
This implies that $\{u^\star_0=0\}$ contains the ray $\{t e^N : t \in \R\}$,
so it follows from Lemma~\ref{lem:infinite line} that $u^\star_0$ is independent of $x_N$, i.e.
\begin{align} \label{eq:u_0_is_clyindric}
	u^\star_0(x) = u^\star_0(x',0) =: (u^\star_0)'(x')   \qquad \text{ for all }  x \in \R^N.
\end{align}
Since $u_{E'}'$ constructed in Step 1 (cf. \eqref{eq:def_of_u_E'}) is the unique global solution to the obstacle problem in $\R^{N-1}$ with $E'$ as coincidence set, we deduce that $(u^\star_0)' \equiv u_{E'}'$. Also, since the classical Newton-potential $V_{E'}^{NP}(x')$ has subquadratic growth as $|x'| \to \infty$, we have that $p(x')$ is the blow-down limit of $u'_{E'}$. Thus, recalling \eqref{eq:u_0_is_clyindric}, 
\begin{align}
\frac{u^\star_0(\rho x)}{\rho^2}=\frac{u'_{E'}(\rho x')}{\rho^2} \to p(x') \quad \text{ in } C^{1,\alpha}_{\rm loc}(\R^{N-1}) \text{ as } \rho \to \infty
\end{align} 
On the other hand, since in the limit as $\rho \to \infty$ the contact set of $\frac{{u^\star} (\rho x)}{\rho^2}$ has measure zero (as a consequence of  Lemma~\ref{lem:compact C1}\eqref{item:compact}), the blow-down limit of $u^\star$ is a homogeneous quadratic polynomial $\tilde{p}$ as well, i.e.
\begin{align}
	\frac{{u^\star} (\rho x)}{\rho^2} \to \tilde{p}(x) \quad \text{ in } C^{1,\alpha}_{\rm loc}(\R^N) \text{ as } \rho \to \infty.
\end{align}
Hence, we are in the position to apply \cite[Lemma B.2]{esw_arXiv} to $u^\star$ and the sequence of rescalings $(u^\star_k)_{k\in \N}$ to deduce that $\tilde p=p$, which proves that $p$ is the blow-down limit of ${u^\star} $. 
In conclusion, ${u^\star}$ is the desired paraboloid solution.

Finally, the remaining statements follow from Proposition \ref{prop:Newton_potential_expansion} and Lemma~\ref{lem:scaling_of_generalized_potential}.
\end{proof}

In the next result we show that, if $N\geq 4$, we can actually find paraboloid solutions with prescrived behavior up to linear order. Given $a \in \R^N$, use the notation $u_{\gamma P -a}$ to denote the solution that has $\gamma P-a$ as coincidence set. Note that this solution is obtained simply by translating the solution having $\gamma P$ as coincidence set, namely $u_{\gamma P -a}(x)=u_{\gamma P }(x+a)$.

\begin{lem}[Existence of paraboloid solutions with prescribed linear behavior at infinity in dimension $N\ge 4$]  
\label{lem:matching_the_quadratic_and_linear_part}
Let $N \geq 4$, let
$p = p(x')$ be a homogeneous quadratic polynomial as in Definition~\ref{def:solution}\eqref{PDE_asymptotics},
and let $P$ be as in Lemma~\ref{lem:matching_the_quadratic_part}.
For any ${b}=(b',b_N)\in \R^{N-1}\times (-\infty,0)$
there exist $\tau' \in \R^{N-1}$ and $\gamma>0$ such that the following holds:
for each $\sigma \in \R$,
\begin{align} \label{eq:potential_expansion_of_quadratic+linear_matching-rescaled}
\frac{1}{R}\fint \limits_{B_R} \left| u_{\gamma P -(\tau',\sigma)}(x)- p(x') - b\cdot x\right|\dx{x} \to 0\qquad \text{as }R\to \infty.
\end{align}
\end{lem}	 
\begin{proof}
As noted before, $u_{\gamma P -(\tau',\sigma)}(x)=u_{\gamma P}(x'+\tau',x_N+\sigma)$. Hence, since $p$ is a quadratic polynomial,
recalling Lemma~\ref{lem:improvement_of_growth_N4}  we have
\begin{align}
u_{\gamma P -(\tau',\sigma)}(x)&=p(x'+\tau')+V_{\gamma P}(x'+\tau',x_N+\sigma)\\
&=p(x')+\nabla p(x')\cdot \tau'+p(\tau') \label{eq:translate dilate u}\\
&\qquad -(x'+\tau',x_N+\sigma)\cdot \biggl(\alpha_N(N-2)\int_{\gamma P}\frac{y}{ \abs{y}^N}\dx{y}\biggr)+W_{\gamma P}(x'+\tau',x_N+\sigma).
\end{align}
Note now that, by symmetry,
$$
\alpha_N(N-2)\int_{\gamma P}\frac{y}{ \abs{y}^N}\dx{y}=\lambda_\gamma e^N,\qquad
\text{where }\lambda_\gamma:=\alpha_N(N-2)\int_{\gamma P}\frac{y_N}{ \abs{y}^N}\dx{y}>0.
$$
Thus
\begin{align}
u_{\gamma P -(\tau',\sigma)}(x)=p(x')+\nabla p(x')\cdot \tau'+p(\tau')-\lambda_\gamma(x_N+\sigma)+W_{\gamma P}(x'+\tau',x_N+\sigma).
\end{align}
Recalling that $\nabla p(x')=Qx'$ with $Q$ symmetric and invertible, we choose $\tau':= Q^{-1}b'$ to ensure that $\nabla p(x')\cdot \tau'=b'\cdot x'$.

On the other hand, it follows by monotone convergence that $\gamma\mapsto \lambda_\gamma$ is continuous and that
\begin{align}
& \lambda_\gamma \to \alpha_N(N-2)\int_{\R^{N-1}\times (0,\infty)}\frac{y_N}{ \abs{y}^N}\dx{y}=+\infty\qquad \text{as }\gamma \to \infty,\\
& \lambda_\gamma \to 0\qquad \text{as }\gamma \to 0.
\end{align}
Thus, by continuity, there exists $\gamma>0$ such that $\lambda_{\gamma}=-b_N$.
Hence, with these choices of $\tau'$ and $\gamma$, we get
$$
u_{\gamma P -(\tau',\sigma)}(x)- p(x') - b\cdot x=p(\tau')-\lambda_\gamma \sigma+W_{\gamma P}(x'+\tau',x_N+\sigma).
$$
Recalling Lemma~\ref{lem:improvement_of_growth_N4} this implies that
for sufficiently large $R$ (depending on $\tau'$, $\gamma$, and $\sigma$),
\begin{multline}
\fint \limits_{B_R} \left| u_{\gamma P -(\tau',\sigma)}(x)- p(x') - b\cdot x\right|\dx{x}
\leq |p(\tau')|+\lambda_\gamma |\sigma| +\fint \limits_{B_R} \left| W_{\gamma P}(x'+\tau',x_N+\sigma)\right|\dx{x}\\
\leq |p(\tau')|+\lambda_\gamma |\sigma|+\frac{1}{|B_R|}\int_{B_{2R}} |W_{\gamma P}(x)|\dx{x} \leq CR^{1/2},
\end{multline}
and the result follows.
\end{proof}

\begin{cor}
\label{cor:matching4}
Let $N\geq 4$, and let $u$ be an $x_N$-monotone solution in the sense of Definition \ref{def:solution}.
Then there exist a paraboloid 
$P$ as in Lemma~\ref{lem:matching_the_quadratic_part}, $\gamma>0$, and $\tau' \in \R^{N-1}$ such that, for each $\sigma \in \R$,
\begin{align} \label{eq:u uP 4}
\frac{1}{R}\fint \limits_{B_R} \left| u- u_{\gamma P -(\tau',\sigma)}  \right|\dx{x} \to 0 \qquad \text{as }R\to \infty.
\end{align}
\end{cor}
\begin{proof}
Let $p$ be as in Definition~\ref{def:solution}, let $\cC$ be the coincidence set of $u$, and let $\ell_\cC$ be as in Lemma~\ref{lem:improvement_of_growth_N4},
so that $u=p-\ell_\cC+W_\cC$.
Define
$$
b=-\nabla \ell_\cC=-\alpha_N(N-2)\int_{\cC}\frac{y}{ \abs{y}^N}\dx{y}.
$$
Since $b_N=-\alpha_N(N-2)\int_{\cC}\frac{y_N}{ \abs{y}^N}\dx{y}<0$, we can
apply Lemmas~\ref{lem:matching_the_quadratic_and_linear_part} and~\ref{lem:improvement_of_growth_N4} to deduce the existence of $P,$ $\gamma$, and $\tau'$ such that, for each $\sigma \in \R$,
$$
\frac{1}{R}\fint \limits_{B_R} \left| u-u_{\gamma P -(\tau',\sigma)}  \right|\dx{x} \leq  \frac{1}{R}\fint \limits_{B_R} \left| W_\cC (x) \right|\dx{x}+\frac{1}{R}\fint \limits_{B_R} \left| u_{\gamma P -(\tau',\sigma)} -p(x')- b\cdot x \right|\dx{x} 
 \to 0
$$
as $R\to \infty.$
\end{proof}

In the critical dimension
$N=3$, the statement of Lemma \ref{lem:matching_the_quadratic_and_linear_part}
does not hold.
The best we can do in dimension $N=3$ is
to match on each $B_R$ the slope of the affine approximation of $V_{P_R}$ in the $e^3$-direction. Note that $P_R$ depends on $R$ and
the slope diverges as $R\to\infty$.

\begin{lem}[Matching on each ball $B_R$ in dimension $N=3$] 
\label{lem:matching_D3A}
Let $N =3$, 
$p = p(x')$ be the blow-down polynomial defined in Definition~\ref{def:solution}\eqref{PDE_asymptotics},
and $P$ as in Lemma~\ref{lem:matching_the_quadratic_part}.
Also, for $\gamma \geq 0$, let $A_{\gamma P}^R$ be defined as in Lemma~\ref{lem:improvement_of_growth_by_subtracting_affine_linear_function}.

Then, given $B>0$, there exist $\gamma_B>0$ and $R_B\geq 1$ such that
the following holds: For each $\beta \in [0,B]$ and $R\geq R_B$ there exists $\gamma=\gamma(\beta,R) \in [0,\gamma_B]$ such that
\begin{align} \label{eq:linear AP}
\partial_3 A_{\gamma P}^R=-\beta\log R.
\end{align}
\end{lem}	 
\begin{proof}
Note that, since $u_{\gamma P}$ is trivially an $x_N$-monotone solution, $A_{\gamma P}^R$ is well-defined thanks to Lemma~\ref{lem:improvement_of_growth_by_subtracting_affine_linear_function}.
Also, as shown in Step 1 in the proof of Lemma~\ref{lem:improvement_of_growth_by_subtracting_affine_linear_function},
\begin{equation}\label{eq:p3 A1}
\partial_3 A_{\gamma P}^R = \alpha_3 \int \limits_{\gamma P} \bra { - \frac{y_3}{|y|^3} + \frac{R +y_3}{|R e^3 +y|^3}      } \dx{y}.
\end{equation}
We now observe that
\begin{align}
\int \limits_{\gamma P} \bra { - \frac{y_3}{|y|^3} + \frac{R +y_3}{|R e^3 +y|^3}      } \dx{y}
&=\lim_{\epsilon \to 0}
\int \limits_{\gamma P} \bra { - \frac{y_3}{|y|^{3+\epsilon}} + \frac{R +y_3}{|R e^3 +y|^{3+\epsilon}}      } \dx{y}\\
&=\lim_{\epsilon \to 0}-\int \limits_{\gamma P}  \frac{y_3}{|y|^{3+\epsilon}} \dx{y}+
\int_{\gamma P} \frac{R +y_3}{|R e^3 +y|^{3+\epsilon}}     \dx{y}\\
&=\lim_{\epsilon \to 0}-\int \limits_{\gamma P\setminus (\gamma P+Re^3)}  \frac{y_3}{|y|^{3+\epsilon}} \dx{y}=- \int \limits_{\gamma P\setminus (\gamma P+Re^3)} \frac{y_3}{|y|^3} \dx{y},
\end{align}
where the first equality follows from dominated convergence since the integrand is uniformly convergent at infinity (see \eqref{eq:A3}), the second equality from the fact that each term in the integrand is integrable for $\epsilon>0$, the third equality from a change of variables, and the last equality from monotone convergence in $|y|>1$
and dominated convergence in $|y|\leq 1$.
This proves that
\begin{equation}\label{eq:p3 A2}
\partial_3 A_{\gamma P}^R =-\alpha_3 \int \limits_{\gamma P\setminus (\gamma P+Re^3)} \frac{y_3}{|y|^3} \dx{y}<0,
\end{equation}
In particular, since $\gamma P+Re^3\subset \{y_3 \geq R\}$, for $R>1$ it follows  that
\begin{align}
\label{eq:D3A R}
\partial_3 A_{\gamma P}^R&\leq -\alpha_3 \int_{\gamma P\cap \{R^{1/2}<y_3 < R\}}\frac{y_3}{|y|^3} \dx{y}.
\end{align}
Note now that, for $y' \in \sqrt{\gamma y_3}E'$ with $y_3\geq R^{1/2}$ and ${R^{1/2}}\geq \gamma$, we have
$$
|y|\leq |y'|+y_3 \leq C_{E'}\gamma^{1/2} y_3^{1/2}+y_3 \leq C_{E'}{R^{1/4}} y_3^{1/2}+y_3\leq (C_{E'}+1)y_3,
$$
for some constant $C_{E'}$ depending only on $E'$.
Hence, thanks to \eqref{eq:D3A R},
\begin{align}
\partial_3 A_{\gamma P}^R&\leq -\frac{\alpha_3}{(C_{E'}+1)^3} \int_{\gamma P\cap \{R^{1/2}<y_3 < R\}}\frac{1}{y_3^2} \dx{y}\leq -\frac{\alpha_3}{(C_{E'}+1)^3} \int_{R^{1/2}}^R \cH^2(\sqrt{\gamma t}E')\frac{1}{t^2} \dx{t} \\
&\leq -c_{E'}\gamma \int_{R^{1/2}}^R \frac{1}{t} \dx{t}=-\frac{c_{E'}}{2}\gamma \log R, \label{eq:p3A log}
\end{align}
where $c_{E'}>0$ is a constant depending only on $E'$.

Hence, given $B>0$, set $\gamma_B:=\frac{2}{c_{E'}}B$ and $R_B:=\max\{{\gamma_B^2},1\}$.
Then, with these choices,
$$
\partial_3 A_{\gamma_B P}^R\leq -B\log R\qquad \text{ for all }R\geq R_B.
$$
On the other hand, recalling \eqref{eq:p3 A1} and \eqref{eq:A3}, we can apply dominated convergence to deduce that $\gamma\mapsto \partial_3 A_{\gamma P}^R$ is continuous and
$$
\partial_3 A_{\gamma P}^R \to 0\qquad \text{as }\gamma \to 0.
$$ 
Hence, by continuity, given any $\beta \in [0,B]$ and $R \ge R_B$ there exists $\gamma=\gamma(\beta,R) \in [0,\gamma_B]$ such that \eqref{eq:linear AP} holds.
\end{proof}

\begin{cor} \label{cor:matching3}
Let $N=3$, and let $u$ be an $x_N$-monotone solution in the sense of Definition \ref{def:solution}.
Then there exist a paraboloid 
$P$ as in Lemma~\ref{lem:matching_the_quadratic_part} as well as constants $\bar\gamma>0$, $\bar R>0$, and $\bar C$ such that the following holds:
for any $R \geq \bar R$ there exists $\gamma_R \in [0, \bar\gamma]$ such that 
\begin{align}
\label{eq:average R N3}
	\frac{1}{R}\fint \limits_{B_R} \left| u-u_{P_R}  \right|  \dx{x}   \leq \bar C,\qquad \text{where }P_R:=\gamma_R P.
\end{align}
\end{cor}

\begin{proof}
We begin by noticing that, since $\cC$ is contained in some paraboloid (see Proposition~\ref{prop:C_is_contained_in_paraboloid}), we can  repeat the proof of \eqref{eq:p3 A2} with $\cC$ in place of $P$ to show that
$$
\partial_3 A_{\cC}^R =-\alpha_3 \int \limits_{\cC\setminus (\cC+Re^3)} \frac{y_3}{|y|^3} \dx{y}.
$$
We now observe that, as a consequence of  the monotonicity of the contact set in the $e^3$-direction 
(a direct consequence of the monotonicity of $u$), the right hand side above is strictly negative.
Thus, thanks to Lemma~\ref{lem:DA}, there exist $\hat R > 1$ and $B>0$ such that
$$
0>\partial_3A_\cC^R \geq - B\log R \qquad \text{ for all }R \geq \hat R.
$$
This allows us to apply Lemma \ref{lem:matching_D3A}  to deduce that, if we set $\bar\gamma:=\gamma_B$ and $\bar R:=\max\{\hat R,R_B\}$, then for any $R\geq \bar R$ there exists $\gamma_R \in (0, \bar\gamma]$
such that
\begin{align}
\label{eq:D3A C P}
	\partial_3 A^R_{P_R} ={\partial_3 A^R_\cC } < 0,\qquad \text{where }P_R:=\gamma_R P.
\end{align}
Using the potential expansion of both $u$ and $u_{P_R}$ (cf. 
Proposition \ref{prop:Newton_potential_expansion} and \eqref{eq:potential_expansion_of_quadratic_matching-rescaled}), thanks to \eqref{eq:D3A C P} we find that for all $R\geq \bar R$,
\begin{align}
\fint \limits_{B_R}  \abs{u-u_{P_R} } \dx{x} 
&=  \fint \limits_{B_R} \abs{   \bV_\cC - \bV_{P_R}               } \dx{x}  \leq  \fint \limits_{B_R} \abs{ \bV_\cC - A_\cC^R  } \dx{x}  + \fint \limits_{B_R} \abs{ \bV_{P_R} - A^R_{P_R}  } \dx{x}  \\
&\quad + \fint \limits_{B_R} \abs { A_\cC^R(0) - A_{P_R}^R(0)  } \dx{x}  + \fint \limits_{B_R} \abs{ \nabla' A_\cC^R - \nabla' A_{P_R}^R  } |x| \dx{x} .
\end{align}
Applying Lemmas~\ref{lem:improvement_of_growth_by_subtracting_affine_linear_function} and~\ref{lem:DA} to $u$,
we can estimate
$$
 \fint \limits_{B_R} \abs{ \bV_\cC - A_\cC^R  } \dx{x} + \fint \limits_{B_R} \abs { A_\cC^R(0)  } \dx{x}  + \fint \limits_{B_R} \abs{ \nabla' A_\cC^R  } |x| \dx{x}  \leq CR.
$$
Also, since $\gamma_R \in (0,\bar \gamma]$, the very same arguments used for proving Lemmas~\ref{lem:improvement_of_growth_by_subtracting_affine_linear_function} and~\ref{lem:DA}  show that
$$
 \fint \limits_{B_R} \abs{ \bV_{P_R} - A_{P_R} ^R  } \dx{x} + \fint \limits_{B_R} \abs { A_{P_R} ^R(0)  } \dx{x}  + \fint \limits_{B_R} \abs{ \nabla' A_{P_R}^R  } |x| \dx{x}  \leq C_{\bar \gamma}R,
$$
where $C_{\bar \gamma}$ depends only on $P$ and $\bar \gamma$ and is thus independent of $R$.
Combining all these estimates, we conclude the validity of \eqref{eq:average R N3}.
\end{proof}

\section{Proof of Theorem~\ref{thm:main}: the case $N \geq 4$}
\label{sect:proof 4}
Given $u$ an $x_N$-monotone solution as in Definition~\ref{def:solution},  using the ACF monotonicity formula from Lemma~\ref{lem:ACF} we will show that $u$ and the comparison solutions
$u_{\gamma P -(\tau',\sigma)}$ provided by
Corollary~\ref{cor:matching4}
 \emph{are ordered}. Thanks to this important fact, the result will follow easily.

In order to simplify notation, we set
\begin{align}
\label{eq:notation P u sigma}
	P_\sigma:= \gamma P -(\tau',\sigma) \quad \text{ and } \quad u_\sigma := u_{\gamma P -(\tau',\sigma)}.
\end{align}

\begin{prop}[Ordering in dimension $N\ge 4$] \label{prop:translation_in_eN-direction_does_not_affect_ordering}
	Let $N\geq 4$. Then, for all $\sigma \in \R$,
	\begin{align}
	\text{ either } \quad  u \leq u_\sigma ~ \text{ in } \R^N \quad \text{ or } \quad u \geq u_\sigma~ \text{ in } \R^N.
	\end{align}
\end{prop}

\begin{proof}
Thanks to
 Lemma~\ref{lem:subharm}\eqref{subharm}, we can apply Lemma~\ref{lem:ACF}\eqref{item:ACF monotone}-\eqref{item:ACF L2 bound} with $v=u-u_\sigma$ to deduce that, for every $r>0,$
\begin{align}
\Phi(u-u_\sigma,r)&\leq \limsup_{R\to \infty} \Phi(u-u_\sigma,R)\leq C_N\limsup_{R\to \infty}\biggl(\frac1R \fint \limits_{B_{4R}} \left| u-u_\sigma\right|\biggr)^4=0,
\end{align}
where the last equality follows from Corollary~\ref{cor:matching4}.
Hence, thanks to Lemma~\ref{lem:ACF}\eqref{item:ACF order} we conclude that either $u-u_\sigma \geq 0$ or $u-u_\sigma \leq 0$, as desired.
\end{proof}
We can now easily prove our main theorem.

\begin{proof}[Proof of Theorem~\ref{thm:main}: the case $N\geq 4$]
Since $u$ is an $x_N$-monotone solution,
\begin{equation} \label{eq:u 0 pos}
u(0)=0,\qquad u(-e^N)>0.
\end{equation}
On the other hand, recalling \eqref{eq:notation P u sigma}, since $P$ is a paraboloid contained in $\{x_N \geq 0\}$ with tip at the origin, recalling the definition of $P_\sigma$ (see \eqref{eq:notation P u sigma}) it follows that 
$$
0 \not\in P_\sigma \qquad \text{for $\sigma <0$},\qquad
-e^N \in P_\sigma \qquad \text{for $\sigma \gg 1$},
$$
and therefore
\begin{equation} \label{eq:usigma 0 pos}
u_\sigma(0)>0 \qquad\text{for $\sigma <0$},\qquad u_\sigma(-e^N)=0\quad \text{for $\sigma \gg 1$}.
\end{equation}
Combining \eqref{eq:u 0 pos}, \eqref{eq:usigma 0 pos}, and Proposition~\ref{prop:translation_in_eN-direction_does_not_affect_ordering},
we conclude that
\begin{equation} 
\label{eq:ordered sigma large}
u_\sigma\geq u \qquad\text{for $\sigma <0$},\qquad u_\sigma \leq u\quad \text{for $\sigma \gg 1$}.
\end{equation}
Now, let us define
$$
\bar\sigma:=\inf\set{\sigma \in \R:u_{\sigma} \leq u}.
$$
Thanks to \eqref{eq:ordered sigma large}, $\bar \sigma \in \R$ is well-defined. We now claim that $u\equiv u_{\bar \sigma}$.

Indeed, by definition of $\bar\sigma$ there exists a sequence $\sigma_k\to \bar\sigma$ such that $u_{\sigma_k}\leq u$, therefore $u_{\bar\sigma}\leq u$.
Assume now towards a contradiction that $u\not\equiv u_{\bar \sigma}$. Then there exists $\bar x \in \R^N$ such that $u_{\bar \sigma}(\bar x)<u(\bar x)$, and by continuity we can find $\epsilon>0$ such that $u_{\bar \sigma-\epsilon}(\bar x)<u(\bar x)$. Since $u$ and $u_{\bar\sigma -\epsilon}$ must be ordered (because of Proposition~\ref{prop:translation_in_eN-direction_does_not_affect_ordering}), we conclude that $u_{\bar \sigma-\epsilon}\leq u$, contradicting the definition of $\bar \sigma$.

Since $u\equiv u_{\bar \sigma}$ we conclude that $\{u=0\}$ is a paraboloid\footnote{A posteriori, by the fact that $\{u=0\}$ is a convex set contained in $\{x_N\geq 0\}$ with tip at the origin, the only possibility is that $\bar\sigma=0$ and $\tau'=0$. However this information is not relevant for our proof.}, as desired.
\end{proof}

\begin{rem}\label{rem:compact case}
It is worth noticing that our argument gives a new proof of the characterization of global solutions with compact coincidence set for any dimension $N\geq 2$.
Indeed, when $\cC=\{u=0\}$ is compact, we can write the expansion (cp. Lemma~\ref{lem:improvement_of_growth_N4} and \cite{ellipsoid})
$$
u(x)=p(x)+V_\cC(x)=p(x)-x\cdot\int_\cC \nabla K_N(y)\dx{y}+\int_\cC\Big(K_N(x-y)-K_N(y)\Big)\dx{y},
$$ 
where $K_N$ is the fundamental solution of the Laplacian (see \eqref{eq:Green}). Since $\cC$ is compact, all integrals converge and the remainder term (the lastintegral) is sublinear.
Also, in this compact case, $p(x)=\frac12 x^TQx$ where $Q \in \R^{N\times N}$ is symmetric and positive definite\footnote{This follows, for instance, from the proof of Theorem \ref{thm:MainTheorem_Intro_I} in Section~\ref{sect:proof thm}.
}. \\
Now, arguing as in Lemma~\ref{lem:matching_the_quadratic_part}, we find an ellipsoid $E \subset \R^N$ such that $u_{E}$ has $p$ as quadratic blow-down limit. In addition, since $Q$ is invertible on $\R^N$, choosing $\tau:=Q^{-1}b$ with $b:=\int_\cC \nabla K_N(y)\dx{y}$,  for each $\gamma > 0$ the function
{$u- u_{\gamma E-\tau}$} has sublinear growth at infinity  (cp. Lemma~\ref{lem:matching_the_quadratic_and_linear_part}). Then the ACF monotonicity formula implies that either $u \leq {u_{\gamma E -\tau}}$ or $u \geq u_{\gamma E -\tau}$ (cp.  Proposition~\ref{prop:translation_in_eN-direction_does_not_affect_ordering}), and finally a continuity argument  implies the existence of a value $\bar \gamma>0$ such that $u\equiv u_{\bar\gamma E -\tau}$ (cp. Proof of Theorem~\ref{thm:main}: the case $N\geq 4$), as desired.
\end{rem}

\section{Proof of Theorem~\ref{thm:main}: the case $N=3$}\label{sect:proof 3}
Let $u$ be an $x_N$-monotone solution  as in Definition~\ref{def:solution}, for each $R \geq \bar R$ let $\gamma_R \in [0, \bar\gamma]$ and let $P_R=\gamma_R P$ be the paraboloid provided by Corollary~\ref{cor:matching3},
so that \eqref{eq:average R N3} holds.

To simplify the notation, for each $R\geq \bar R$ and $x \in \R^3$ we define
\begin{align} \label{eq:definition_of_v_R_and_v_R_R}
v_R(x) := \frac{(u -u_{P_R})(Rx)}{R},
\end{align}
so that \eqref{eq:average R N3} becomes equivalent to $\norm{v_R}_{L^1(B_{1})} \leq C$ for all $R \geq \bar R$.
Hence, thanks to uniform $L^1$-bound and
 Lemma~\ref{lem:subharm}\eqref{subharm}, we can apply Lemma~\ref{lem:ACF}\eqref{item:ACF monotone}-\eqref{item:ACF L2 bound} 
 to deduce that, for all $r \in (2\bar R,R)$,
\begin{align} \label{eq:ACF_for_difference}
\Phi \bra {u - u_{P_R}, \tfrac{r}{2}} \leq \Phi \bra {u - u_{P_R}, \tfrac{R}{2}} = \Phi\bra{ v_R,\tfrac{1}{2}} \leq C.
\end{align}
Since $0 \in  \{u_{P_R} =0\}$ for all $R \geq \bar R$, it follows from Lemma~\ref{lem:compact C1}\eqref{item:compact} that, passing if necessary to a subsequence,
\begin{align}\label{eq:definition_of_tilde_u}
	u_{P_R} \to u_\infty \quad \text{ in } C^{1,\alpha}_{\rm loc}(\R^3)\quad \text{ as } R \to \infty,
\end{align}
where $u_\infty$ is a global solution to the obstacle problem.
Also, since ${P_R}=\gamma_R P$ with $\gamma_R\in [0,\bar \gamma],$ it follows that
\begin{align}\label{eq:definition_of_tilde_u2}
\{u_\infty=0\}={\gamma_\infty P}\quad\text{ for some $\gamma_\infty \in [0,\bar \gamma]$},\qquad P=\{ y' \in \sqrt{y_N} E'   \}
\end{align}
(if $\gamma_\infty=0$ then $u_\infty=p$).

Thanks to \eqref{eq:definition_of_tilde_u} and Fatou's Lemma, it follows from \eqref{eq:ACF_for_difference} that $\Phi\bra{ u-u_\infty, \tfrac{r}{2}} \leq C $ for all  $r>2\bar R$. Hence, since $\Phi$ is non-decreasing in $r$ (see Lemma~\ref{lem:ACF}\eqref{item:ACF monotone}), we obtain that
\begin{align} \label{eq:ACF_energy_of_difference_is_bounded}
\Phi\bra{ u-u_\infty, r} \leq C \qquad \text{ for all }r>0.
\end{align}

\subsection{Linear rescaling and ACF dichotomy}
Let us now introduce the linear rescaling
\begin{align}
\label{eq:wr}
	w_r(x) := \frac{(u-u_\infty)(rx)}{r}   \qquad \text{ for }  x \in \R^3
\text{ and }r>0.
\end{align}
We prove the following important dichotomy. 

\begin{prop}[ACF alternative for $u-u_\infty$] \label{prop:ACF-alternative_for_linear_rescaling}
{Let $u_\infty$, $\gamma_\infty$, and $P$ be as in \eqref{eq:definition_of_tilde_u}-\eqref{eq:definition_of_tilde_u2}},
and $w_r$ as in \eqref{eq:wr}. Then there exists a sequence $r_k \to \infty$ such that $w_{r_k} \to w$ strongly in $W^{1,2}(B_1)$ as $k \to \infty$. Also, 
\begin{enumerate}[(i)]
	\item either $w$ has constant sign inside $B_1$ (i.e. either $w\geq 0$ {a.e. in $B_1$} or $w \leq 0$ a.e. in $B_1$);\label{case:w ordered}
	\item or $w$ is a linear function, i.e. there exists $b \in \R^3$ such that $w(x) = b \cdot x$ a.e. in $B_1$.\label{case:w affine}
\end{enumerate}
\end{prop}

\begin{proof} We will first prove that the family $(w_r)_{r>0}$ is bounded in $L^1(B_2)$, and use the boundedness of the ACF functional to deduce the desired dichotomy.\\
\textbf{Step 1.} \emph{There exists a constant $C$ such that $\|w_r\|_{L^1(B_{{4}})}\leq C$ for all $r$ sufficiently large.}\\
We begin by noticing that, thanks to Lemma \ref{lem:improvement_of_growth_by_subtracting_affine_linear_function}, there exist affine linear functions $A_\cC^r, A_{\gamma_\infty P}^r$ such that, for all $r$ sufficiently large,
\begin{align}
\fint \limits_{B_{4r}} \big| u-p-A_\cC^r  \big| \dx{x} \leq C r \quad \text{ and } \quad \fint \limits_{B_{4r} } \big| u_\infty - p - A_{\gamma_\infty P} ^r \big| \dx{x} \leq C r
\end{align}
(note that, in the case $u_\infty =p$, the second estimate is trivially satisfied with $A^r_{\gamma_\infty P} \equiv 0$). Setting $\cA^r(x) := \tfrac{1}{r} (A_\cC^r(rx)  - A^r_{}(rx))$ it follows that, for all $r$ sufficiently large,
\begin{align} \label{eq:w_r-A_on_sphere 1}
	\int_{B_4} \abs{ w_r - \cA^r } \dx{x} \leq C.
\end{align}
Also, applying Lemma \ref{lem:DA} to both $u$ and $u_\infty$ we deduce that $\abs{\cA^r(0)}+\abs{\nabla' \cA^r(0)  } \leq C$ , and therefore \eqref{eq:w_r-A_on_sphere 1} implies that
\begin{align} \label{eq:w_r-A_on_sphere}
	\int_{B_4} \abs{ w_r(x) - \alpha_r x_3} \dx{x} \leq C
\text{ for all $r$ sufficiently large}, \text{ where }
\alpha_r:=\partial_3 \cA^r.
\end{align}
{In particular, it follows from Chebyshev's inequality that
\begin{align} \label{eq:w_r-A_on_sphere2}
|B_4\cap \{ |w_r-\alpha_r x_3|>|\alpha_r|/2\}| \le \frac{2C}{\alpha_r} \qquad \text{for all $r$ sufficiently large}.
\end{align}
}
Suppose now towards a contradiction that the claim of this step is false.
Then there exists a sequence $\rho_k \to \infty$ as $k\to\infty$ such that $\|w_{\rho_k}\|_{L^1(B_4)}\to +\infty$. It follows from \eqref{eq:w_r-A_on_sphere} that $|\alpha_{\rho_k}| \to +\infty$ as $k\to\infty$.

We may assume that $\alpha_{\rho_k}\to +\infty$. Then \eqref{eq:w_r-A_on_sphere} and \eqref{eq:w_r-A_on_sphere2} imply that, for sufficiently large $k$,
\begin{align}
&|\{(w_{\rho_k})_-=0\} \cap {B_4}| \geq |\{w_{\rho_k}<0\} \cap {B_4}| \geq \frac12 |{B_4}\cap \{x_3<-1/2\}|,\qquad \|(w_{\rho_k})_-\|_{L^1(B_2)} \to +\infty,\\
&|\{(w_{\rho_k})_+=0\} \cap {B_4}| \geq |\{w_{\rho_k}>0\} \cap {B_4}|\geq \frac12 |{B_4}\cap \{x_3>1/2\}|,\qquad \|(w_{\rho_k})_+\|_{L^1(B_4)} \to +\infty.
\end{align}
This allows us to apply Poincar\'e's inequality to $(w_{\rho_k})_\pm$. We obtain that
$$
+\infty \leftarrow \|(w_{\rho_k})_\pm\|_{L^1(B_4)}\leq C\|\nabla (w_{\rho_k})_\pm\|_{L^1(B_4)}.
$$
Since $\frac{1}{|x|}\geq \frac12$ inside $B_4$, it follows  by H\"older's inequality that
$$
\|\nabla (w_{\rho_k})_\pm\|_{L^1(B_4)}^2 \leq C\int_{B_4}|\nabla (w_{\rho_k})_\pm |^2\dx{x}\leq C \int_{B_4} \frac{|\nabla (w_{\rho_k})_\pm |^2}{|x|}\dx{x}.
$$
Thus, recalling that $N=3$ and the definition of $\Phi$, we conclude that
$$
+\infty \leftarrow \Phi(w_{\rho_k},4)=\Phi(u-u_\infty, 4\rho_k),
$$
a contradiction to \eqref{eq:ACF_energy_of_difference_is_bounded}.\\
\textbf{Step 2.} \emph{Proof of the dichotomy.}\\
Thanks to Step 1 and Lemma \ref{lem:ACF}\eqref{eq:bound ACF individual}, there exists $\bar r>0$ such that
$$
\int \limits_{B_1} \frac{|\nabla (w_r)_\pm|^2}{|x|}\dx{x}\leq C\qquad \text{ for all }r>\bar r.
$$
This implies the following non-concentration estimate for the ACF integrands: for each
$\delta \in (0,1)$ and every $r >\frac{\bar r}{\delta}$
\begin{align}
\label{eq:non conc}
0 \leq \int \limits_{B_\delta} \frac{\abs{\nabla (w_r)_\pm}^2}{|x|}\dx{x} = \delta^2 \int \limits_{B_1} \frac{ \abs{\nabla (w_{r \delta})_\pm  }^2 }{|x|}\dx{x} \leq C \delta^2 .
\end{align}
In particular, since $\frac{1}{|x|}\geq 1$ inside $B_1$, we have that
$\|\nabla w_r\|_{L^2(B_1)} \leq C$ for all $r \geq \bar r.$ 
{Also, since $|w_r|$ is subharmonic (as a consequence of Lemma~\ref{lem:subharm}\eqref{subharm}), it follows from Step 1 that $\|w_r\|_{L^2(B_2)}\leq C$ for all $r\geq \bar r$.
Thus, there exists a sequence $r_k\to \infty$}
such that
\begin{align}
w_{r_k} \rightharpoonup w \quad \text{ weakly in  } W^{1,2}(B_1) \text{ as } k \to \infty.
\end{align}
We now observe that, since the coincidence sets of $u$ and $u_\infty$ satisfy the properties in Proposition~\ref{prop:C_is_contained_in_paraboloid},
it follows that $\frac{1}{r}(\{u=0\}\cup\{u_\infty=0\})\to \{ se^3 : s\geq 0\}$ as $r \to \infty$. Hence, since $\Delta (u-u_\infty)=0$ outside $\{u=0\}\cup\{u_\infty=0\}$, we deduce that
 $\Delta w= 0$ outside $\{ se^3 : s\geq 0\}$. However, because this set has $2$-capacity zero in $\R^3$, we conclude that $\Delta w \equiv 0$ and therefore, by Lemma \ref{lem:strong_W_1_2_convergence_of_v_R_R},
\begin{align} \label{eq:strong_convergence_in_W12_ACF_alternative_case_1}
w_{r_k} \to w \quad \text{ strongly in  } W^{1,2}_{\loc}(B_1) \text{ as } k \to \infty.
\end{align}
Combining the strong convergence \eqref{eq:strong_convergence_in_W12_ACF_alternative_case_1}  with the non-concentration estimate \eqref{eq:non conc} we conclude that,
for each $\rho \in (0,1)$,
$$
\limsup_{k\to \infty} \biggl|\int \limits_{B_\rho} \frac{\abs{\nabla (w_{r_k})_\pm}^2}{|x|}\dx{x}-\int \limits_{B_\rho} \frac{\abs{\nabla w_\pm}^2}{|x|}\dx{x}\biggr| \leq
\limsup_{k\to \infty} \biggl|\int \limits_{B_\delta} \frac{\abs{\nabla (w_{r_k})_\pm}^2}{|x|}\dx{x}-\int \limits_{B_\delta} \frac{\abs{\nabla w_\pm}^2}{|x|}\dx{x}\biggr| \leq C\delta^2,
$$
so
$$
\int \limits_{B_\rho} \frac{\abs{\nabla (w_{r_k})_\pm}^2}{|x|}\dx{x}\to \int \limits_{B_\rho} \frac{\abs{\nabla w_\pm}^2}{|x|}\dx{x}\qquad \text{as }k \to \infty.
$$
In particular $\Phi(w_{r_k}, \rho) \to \Phi(w,\rho)$  as $k \to \infty$ and therefore, by the monotonicity of the ACF functional as well as \eqref{eq:ACF_energy_of_difference_is_bounded},
\begin{align}
\infty> \Phi(u-u_\infty, \infty)=\lim \limits_{k \to \infty} \Phi(u-u_\infty, r_k \rho) = \lim \limits_{k \to \infty} \Phi(w_{r_k}, \rho) = \Phi(w,\rho)\qquad \text{ for each }\rho \in (0,1).
\end{align}
Thus $\rho \mapsto \Phi(w,\rho)$ is constant on $(0,1)$, so the result follows from \cite[Theorem 2.9]{PetrosyanShahgholianUraltseva_book}, bearing in mind that $w$ is harmonic in $B_1$.
\end{proof}

As we shall see later, if we are in Case~\eqref{case:w ordered} of Proposition~\ref{prop:ACF-alternative_for_linear_rescaling}, then it is easy to conclude. On the other hand, Case~\eqref{case:w affine} requires a delicate argument that is performed in the next section.

\subsection{Fine adjustment of $u_\infty$ at large scales} 
\label{sec:fine_adujustment_of_the_comparison_solution_on_large_spheres}

The goal of this section is to show that, if Case~\eqref{case:w affine} of Proposition~\ref{prop:ACF-alternative_for_linear_rescaling} occurs, then we can find some fine adjustments of $u_\infty$ at large scales to cancel the linear function $b\cdot x$ appearing in the blow-down limit.

\begin{prop}[Fine adjustment of the matching] \label{prop:calibration_in_the_case_that_tilde_u_is_paraboloid_solution}
{Let $u_\infty$, $\gamma_\infty$, and $P$ be as in \eqref{eq:definition_of_tilde_u}-\eqref{eq:definition_of_tilde_u2}},
and assume that
\begin{align} \label{eq:assumption_that_u-tilde_u_converges_in_lin_scaling_to_lin_function}
	\frac{(u-u_{\gamma_\infty P})(r_k x)}{r_k} \to b \cdot x \quad \text{ strongly in } W^{1,2}(B_1), r_k \to \infty \text{ as } k \to \infty.
\end{align}
Then there exist $\tau' \in \R^2$ and a sequence $\gamma_{k}\to \gamma_\infty$, $\gamma_k\in [0,\bar\gamma+1]$ such that, for each $\sigma \in \R$,
\begin{align}\label{eq:u ugammak 0}
\frac{(u-u_{\gamma_{k} P - (\tau',\sigma)})(r_k x)}{r_k} \to 0 \quad \text{ strongly in } L^1(B_{1}) \text{ as } k \to \infty.
\end{align}
\end{prop}

To prove this result, we will need a series of preliminary estimates on the behaviour of paraboloid solutions under translation and scaling.
We collect these in the lemmas below.

{
\begin{lem}[Translations of paraboloid solutions] \label{lem:translate V}
Let $P$ be as in \eqref{eq:definition_of_tilde_u2}.
Then there exists a constant $C=C(P,\bar\gamma)$ such that,
for every $\gamma \in [0,\bar\gamma+1]$ and every $\tau \in \R^3$,
\begin{align}
\fint \limits_{B_R} \abs{ V_{\gamma P}(x+\tau) - V_{\gamma P}(x)} \dx{x} \leq C|\tau|
R^{3/4} \qquad \text{ for all } R\geq \max\{|\tau|,1\}.
\end{align}
\end{lem}
\begin{proof}
Using the fundamental theorem of calculus, for any $R\geq |\tau|$  we can estimate
$$
\fint \limits_{B_R} \abs{ V_{\gamma P}(x+\tau) - V_{\gamma P}(x)} \dx{x} \leq C|\tau| \fint \limits_{B_{2R}} \abs{ \nabla V_{\gamma P}(x) }\dx{x}.
$$
We now apply H\"older's inequality and Lemma \ref{lem:subharm}\eqref{cacc} (note that $V_{\gamma P} = u_{\gamma P}- p$, where both $u_{\gamma P}$ and $p$ are solutions to the obstacle problem)
to the right-hand side and obtain
\begin{align}
\fint \limits_{B_R} \abs{ V_{\gamma P}(x+\tau) - V_{\gamma P}(x)} \dx{x} \leq C|\tau|
\sqrt{\fint \limits_{B_{2R}} \abs{ \nabla V_{\gamma P}(x) }^2\dx{x}}
\leq C\frac{|\tau|}{R}
\sqrt{\fint \limits_{B_{4R}} \abs{ V_{\gamma P}(x) }^2\dx{x}}.
\end{align}
On the other hand, we know from \eqref{eq:sublinear} that there exists a constant $C_{\bar\gamma}$ such that $|V_{\gamma P}(x)|\leq C_{\bar\gamma}(1+|x|)^{7/4}$ for all $\gamma \in [0,\bar\gamma+1]$.
Combining these facts we obtain that
\begin{align}
\fint \limits_{B_R} \abs{ V_{\gamma P}(x+\tau) - V_{\gamma P}(x)} \dx{x} \leq C|\tau|
R^{3/4}  \qquad \text{ for all } R\geq \max\{|\tau|,1\},
\end{align}
and the result follows.
\end{proof}
}

\begin{lem}[The generalized Newtonian potential of scaled paraboloids] \label{lem:perturb A}
{Let $u_\infty$, $\gamma_\infty$, and $P$ be as in \eqref{eq:definition_of_tilde_u}-\eqref{eq:definition_of_tilde_u2}}, and for $\gamma \in [0,\bar\gamma+1]$ and $R>0$ define the 
affine function 
\begin{multline} \label{eq:definition_of_cal_A_diff_of _scaled_paraboloids}
\cA^R_{(\gamma P,\gamma_\infty P)}(x) := A^R_{\gamma P}(x) - A^R_{\gamma_\infty P}(x) \\
= \alpha_3 \int \limits_{\R^3}  \bra { - \frac{1}{|y|} + \frac{1}{|-R e^3 -y|} - \frac{x \cdot y}{|y|^3} - \frac{(-R e^3 -y) \cdot (x+ R e^3)}{|-R e^3 -y|^3}            }\bra {\chi_{\gamma P}-\chi_{\gamma_\infty P} }(y)  	\dx{y}.
\end{multline}
Then there exists {a modulus of continuity $\omega=\omega_{P,\bar\gamma}:
[0,+\infty)\to [0,+\infty)$ such that $\omega(0)=0$ and, for all $R\geq 1$,
\begin{align}
	\fint \limits_{B_R} \abs{ V_{\gamma P}- V_{ \gamma_\infty P}  -\cA^R_{(\gamma P,\gamma_\infty P)}} \dx{x} & \leq \omega(|\gamma-\gamma_\infty|)R,
\label{eq:potential vary lambda}
\end{align}
\begin{align}
|\nabla' \cA^R_{(\gamma P,\gamma_\infty P)}| \leq \omega(|\gamma-\gamma_\infty|) \qquad \text{and}\qquad|\cA^R_{(\gamma P,\gamma_\infty P)}(0)| \leq \omega(|\gamma-\gamma_\infty|)R.
\label{eq:affine vary lambda}
\end{align}
}
\end{lem}

\begin{proof} 
As we shall see, the proof of is slight modification of the ones of Lemmas~\ref{lem:improvement_of_growth_by_subtracting_affine_linear_function} and~\ref{lem:DA}.\\
\textbf{Step 1.} \emph{Proof of \eqref{eq:potential vary lambda}.} \\
We follow the notation used in the proof of Lemma~\ref{lem:improvement_of_growth_by_subtracting_affine_linear_function}.
Recalling the definition of $G(x,y)$ in Definition~\ref{NP} and of $a^R(x,y)$ in \eqref{eq:ell R}, and recalling that $A\triangle B$ denotes the symmetric difference of two sets $A$ and $B$, we have
\begin{multline}
\abs { V_{ {\gamma P}}(x) - V_{ \gamma_\infty P} (x) -\cA^R_{(\gamma P,\gamma_\infty P)}(x)    } 
\leq \int_{({{\gamma P} \triangle \gamma_\infty P})\cap \{0\leq y_3 \leq 2R\}}|G(x,y)-a^R(x,y)|\dx{y}\\
+\int_{({{\gamma P} \triangle \gamma_\infty P})\cap \{y_3 \geq 2R\}}|G(x,y)-a^R(x,y)|\dx{y}=:J_1(x)+J_2(x).
\end{multline}
Using \eqref{eq:Taylor G} and \eqref{eq:Taylor ell}, for $x \in B_R$ we estimate
\begin{multline}
J_2(x)\leq CR^2\int_{2R}^\infty \cH^2\left(\Big\{y' \in \sqrt{t}\big(\sqrt{\gamma} E' \triangle \sqrt{\gamma_\infty} E'\big)\Big\}\right) \frac{1}{t^3}\dx{t}\\
=CR^2\int_{2R}^\infty |\gamma-\gamma_\infty|\cH^2(E') \frac{1}{t^2}\dx{t} \leq C|\gamma-\gamma_\infty|R
\end{multline}
(cp. \eqref{eq:bound J2 R}).
For $J_1(x)$, we write $x=Rz$ with $z \in B_1$ and we perform the change of variables $y\mapsto Ry$, so that
\begin{align}
J_1(x)&\leq \int_{({{\gamma P} \triangle \gamma_\infty P})\cap \{y_3 \leq 2R\}}|G(x,y)-a^R(x,y)|\dx{y}\\
&\leq CR^2\int_{\left\{y' \in \sqrt{\frac{y_3}{R}}\left(\sqrt{\gamma} E' \triangle \sqrt{\gamma_\infty} E'\right)\right\}\cap \{0\leq y_3 \leq 2\}}|G(z,y)-a^1(z,y)|\dx{y}
\end{align}
(cp. \eqref{eq:x to z}).
Combining these two bounds, we get
\begin{multline}
\fint \limits_{B_R}  \abs { V_{ {\gamma P}}(x) - V_{ \gamma_\infty P} (x) -\cA^R_{(\gamma P,\gamma_\infty P)}(x)  } \dx{x}\leq \fint \limits_{B_R} \big(J_1(x)+J_2(x)) \dx{x}
\\
 \leq C|\gamma-\gamma_\infty|R
+ CR^2 \fint \limits_{B_1} \int_{\left\{y' \in \sqrt{\frac{y_3}{R}}\left(\sqrt{\gamma} E' \triangle \sqrt{\gamma_\infty} E'\right)\right\}\cap \{0\leq y_3 \leq 2\}}|G(z,y)-a^1(z,y)|\dx{y}\dx{z}
\end{multline}
(cp. \eqref{eq:J12 z}).
Concerning the last integral we note that \eqref{eq:G ell} holds on the domain of integration. Hence, since $\left|\left\{y' \in \sqrt{\frac{y_3}{R}}\left(\sqrt{\gamma} E' \triangle \sqrt{\gamma_\infty} E'\right)\right\}\cap \{0\leq y_3 \leq 2\}\right| \leq \frac{C|\gamma-\gamma_\infty|}R$,
using \eqref{eq:int G on B1} and  Fubini's Theorem we obtain \eqref{eq:potential vary lambda} 
(cp. \eqref{eq:J12 final}).\\
\textbf{Step 2.} \emph{Proof of \eqref{eq:affine vary lambda}.} \\
For the first bound we note that, for $y \in {\gamma P}\cup \gamma_\infty P$, \eqref{eq:D'ell} holds.
Hence
\begin{align}
\int_{({{\gamma P} \triangle \gamma_\infty P})}\abs { \nabla_{x'} a^R(x,y) } \dx{y}&\leq 2\int_{({{\gamma P} \triangle \gamma_\infty P}) \cap \set {y_3 \leq 1} }  \frac{1}{|y|^2}  \dx{y}
+2\int_{({{\gamma P} \triangle \gamma_\infty P}) \cap \set {y_3\geq 1} }  \frac{|y'|}{|y|^3} \dx{y}\\
 &\leq 2\int_{({{\gamma P} \triangle \gamma_\infty P}) \cap \set {y_3 \leq 1} }  \frac{1}{|y|^2}  \dx{y}\\
 &\qquad +C\int_{1}^\infty \cH^2\left(\Big\{y' \in \sqrt{t}\big(\sqrt{\gamma} E' \triangle \sqrt{\gamma_\infty} E'\big)\Big\}\right) \frac{t^{1/2}}{t^3}\dx{t}\\
& \leq 2\int_{({{\gamma P} \triangle \gamma_\infty P}) \cap \set {y_3 \leq 1} }  \frac{1}{|y|^2}  \dx{y}+C|\gamma-\gamma_\infty|\int_{1}^\infty \frac{1}{t^{3/2}}\dx{t}.
\end{align}
By dominated convergence, this proves the existence of a modulus of continuity $\omega$ such that
 $|\nabla' \cA^R_{(\gamma P,\gamma_\infty P)}| \leq \omega(|\gamma-\gamma_\infty|).$

For the second bound we write
\begin{multline}
\int_{({{\gamma P} \triangle \gamma_\infty P})}|a^R(0,y)|\dx{y}\leq \int_{({{\gamma P} \triangle \gamma_\infty P})\cap \{y_3 \leq 2R\}}|a^R(0,y)|\dx{y}\\
+\int_{({{\gamma P} \triangle \gamma_\infty P})\cap \{y_3 \geq 2R\}}|a^R(0,y)|\dx{y}=:J_1+J_2.
\end{multline}
Using \eqref{eq:Taylor ell}, we immediately get 
\begin{multline}
J_2
\leq CR^2\int_{2R}^\infty \cH^2\left(\Big\{y' \in \sqrt{t}\big(\sqrt{\gamma} E' \triangle \sqrt{\gamma_\infty} E'\big)\Big\}\right) \frac{1}{t^3}\dx{t}\\
\leq CR^2|\gamma-\gamma_\infty|\int_{2R}^\infty \frac{1}{t^2}\dx{t}\leq C|\gamma-\gamma_\infty| R,
\end{multline}
(cp. \eqref{eq:I2 ell}).
Concerning $J_1$, using \eqref{eq:ell R J1} we have
\begin{align}
J_1&\leq 3\int_{0}^R \cH^2\left(\Big\{y' \in \sqrt{t}\big(\sqrt{\gamma} E' \triangle \sqrt{\gamma_\infty} E'\big)\Big\}\right) \frac{1}{t}\dx{t} \leq C|\gamma-\gamma_\infty| \int_{0}^R \dx{t}\leq C|\gamma-\gamma_\infty| R
\end{align}
(cp.  \eqref{eq:ell R J1 2}).
This implies that $|\cA^R_{(\gamma P,\gamma_\infty P)}(0)| \leq C|\gamma-\gamma_\infty|R$,
concluding the proof.
\end{proof}

\begin{lem} \label{lem:perturb A3}
Let $P$ be as in \eqref{eq:definition_of_tilde_u2}, and fix $\gamma \in [0,\bar\gamma+1]$.
Then there  exists a constant $C=C(P,\bar\gamma)$ such that for all $R\geq 1$,
$$
\abs{ \frac{\partial}{\partial R}\partial_3A_{\gamma P}^R } \leq \frac{C}{R}.
$$
\end{lem}
\begin{proof}
Recalling \eqref{eq:DAR}, it follows that
\begin{align}
\left|\frac{\partial}{\partial R} \partial_3A_{\gamma P}^R\right|&\leq \alpha_3 \int \limits_{\gamma P} \left|\frac{\partial}{\partial R}\bra { - \frac{y_3}{|y|^3} + \frac{R +y_3}{|R e^3 +y|^3}      }\right| \dx{y}= \alpha_3 \int \limits_{\gamma P} \left|\frac{\partial}{\partial R}\biggl(\frac{R +y_3}{|R e^3 +y|^3}  \biggr)    \right| \dx{y}\\
&\leq C\int \limits_{\gamma P} \frac{1}{|R e^3 +y|^3}     \dx{y} \leq C\int_{0}^\infty \cH^2\left(\big\{y' \in \sqrt{\gamma t} E' \big\}\right) \frac{1}{(R+t)^3}\dx{t}\\
&\leq C\int_{0}^\infty  \frac{t}{(R+t)^3}\dx{t}\leq \frac{C}{R^3}\int_0^R t\dx{t}+C\int_R^\infty \frac{1}{t^2}\dx{t}\leq \frac{C}{R}.
\end{align}
\end{proof}

\begin{proof}[Proof of Proposition~\ref{prop:calibration_in_the_case_that_tilde_u_is_paraboloid_solution}]
We note that, given $\gamma \in [0,\bar\gamma+1],$ we have (cp. \eqref{eq:translate dilate u})
\begin{align}
u_{\gamma P-(\tau',\sigma)}(x)&=u_{\gamma P}(x'+\tau',x_3+\sigma)=p(x'+\tau')+V_{\gamma P}(x'+\tau',x_3+\sigma)\\
&=p(x')+\nabla p(x')\cdot \tau'+p(\tau')+\Big(V_{\gamma P}(x'+\tau',x_3+\sigma)-V_{\gamma P}(x)\Big)\\
&\qquad\qquad\qquad\qquad\qquad\qquad+\Big(V_{\gamma P}(x)-V_{\gamma_\infty P}(x)\Big)+V_{\gamma_\infty P}(x)\\
&=u_{\gamma_\infty P}(x)+\nabla p(x')\cdot \tau'+\Big(V_{\gamma P}(x)-V_{\gamma_\infty P}(x)\Big) \\
&\qquad \qquad\qquad\qquad\qquad\qquad +\Big[p(\tau')+\Big(V_{\gamma P}(x'+\tau',x_3+\sigma)-V_{\gamma P}(x)\Big)\Big].
\end{align}
Let $b=(b',b_3)\in \R^2\times \R$.
Recalling that $\nabla p(x')=Qx'$ with $Q$ invertible, we choose $\tau':= Q^{-1}b'$ so that $\nabla p(x')\cdot \tau'=b'\cdot x'$.
Also, for each $R>1$,
\begin{multline}
V_{\gamma P}(x)-V_{\gamma_\infty P}(x)=\Big(V_{\gamma P}(x)-V_{\gamma_\infty P}(x)-\cA^R_{(\gamma P,\gamma_\infty P)}(x)\Big)\\
+\cA^R_{(\gamma P,\gamma_\infty P)}(0)+\nabla'\cA^R_{(\gamma P,\gamma_\infty P)}\cdot x'+\partial_3 \cA^R_{(\gamma P,\gamma_\infty P)}x_3.
\end{multline}
Combining all that and applying Lemmas~\ref{lem:translate V} and~\ref{lem:perturb A}, we obtain that for $x\in B_R,$
\begin{align}
\label{eq:u gamma gammainfty}
\fint \limits_{B_R}\Big| u_{\gamma P-(\tau',\sigma)}(x)-u_{\gamma_\infty P}(x)-b'\cdot x' -\partial_3 \cA^R_{(\gamma P,\gamma_\infty P)}x_3 \Big|\dx{x} \leq 
C\left(|\gamma-\gamma_\infty|R+R^{7/8}\right).
\end{align}
We now distinguish two cases, depending on whether $\gamma_\infty>0$ or not.\\
\textbf{Case 1.} \emph{$\gamma_\infty>0$}. \\
We note that, by a change of variables, for each $\lambda>0$ 
\begin{align}
\partial_3A^R_{\lambda\gamma_\infty P}&=\int \limits_{\lambda \gamma_\infty P} \bra { - \frac{y_3}{|y|^3} + \frac{R +y_3}{|R e^3 +y|^3}      }\dx{y}=\lambda^3\int \limits_{\gamma_\infty P} \bra { - \frac{\lambda z_3}{\lambda^3|z|^3} + \frac{R +\lambda z_3}{|R e^3 +\lambda z|^3}      }\dx{z}\\
&=\lambda\int \limits_{\gamma_\infty P} \bra { - \frac{z_3}{|z|^3} + \frac{R/\lambda +z_3}{|R/\lambda e^3 + z|^3}      }\dx{z}=\lambda \partial_3A^{R/\lambda}_{\gamma_\infty P}.
\label{eq:scaling D3A}
\end{align}
Thus, if we set $\lambda:=\frac{\gamma}{\gamma_\infty}$,
\begin{equation}
\label{eq:p3 AR scaling}
\partial_3 \cA^R_{(\gamma P,\gamma_\infty P)}=\partial_3A^R_{\gamma P} - \partial_3A^R_{\gamma_\infty P}=\lambda\big[\partial_3A^{R/\lambda}_{\gamma_\infty P}-\partial_3A^{R}_{\gamma_\infty P}\big]+ (\lambda-1)\partial_3A^R_{\gamma_\infty P}.
\end{equation}
Assuming now
that $\gamma$ is sufficiently close to $\gamma_\infty$ so that $\lambda \in [1/2,2]$,
it follows from Lemma~\ref{lem:perturb A3} that
$$
\big|\partial_3A^{R/\lambda}_{\gamma_\infty P}-\partial_3A^{R}_{\gamma_\infty P}\big|\leq \int_{R/\lambda}^R\abs{ \frac{\partial}{\partial r}\partial_3A_{\gamma P}^r } \,dr \leq C|\lambda-1| ={C \frac{|\gamma-\gamma_\infty|}{|\gamma_\infty|}}.
$$
Thus, combining \eqref{eq:u gamma gammainfty} and \eqref{eq:p3 AR scaling}, we obtain
\begin{align}
\fint \limits_{B_R}\Big| u_{\gamma P-(\tau',\sigma)}(x)-u_{\gamma_\infty P}(x)-b'\cdot x' -\frac{\gamma-\gamma_\infty}{\gamma_\infty} \partial_3A^R_{\gamma_\infty P} x_3 \Big|\dx{x} \leq 
C\left(|\gamma-\gamma_\infty|R+R^{7/8}\right).
\end{align}
We now observe that, as a consequence of Lemma~\ref{lem:DA} and \eqref{eq:p3A log}, there exist constants $0<c_\infty <C_\infty$ such that
$$
-c_\infty \log R\geq  \partial_3A^R_{\gamma_\infty P}\geq -C_\infty \log R \qquad \text{ for all $R$ sufficiently large}.
$$
In particular, for each $R$ sufficiently large we can choose $\gamma=\gamma_R \in \left[\gamma_\infty - \frac{|b_3|\gamma_\infty}{c_\infty\log R},\gamma_\infty+\frac{|b_3|\gamma_\infty}{c_\infty\log R}\right]$ such that $\frac{\gamma-\gamma_\infty}{\gamma_\infty} \partial_3A^R_{\gamma_\infty P}=b_3$, and with such a choice we have
\begin{align}
\label{eq:u gamma gammainfty 2}
\fint \limits_{B_R}\Big| u_{\gamma_R P-(\tau',\sigma)}(x)-u_{\gamma_\infty P}(x)-b'\cdot x' -b_3 x_3 \Big|\dx{x} \leq 
C\left(\frac{R}{\log R}+R^{7/8}\right)
\end{align}
for sufficiently large $R$.
Choosing $R=r_k$ and $\gamma_k=\gamma_{r_k}$, and combining \eqref{eq:assumption_that_u-tilde_u_converges_in_lin_scaling_to_lin_function} and \eqref{eq:u gamma gammainfty 2},
we obtain
\eqref{eq:u ugammak 0}.\\
\textbf{Case 2.} \emph{$\gamma_\infty=0$}. \\
Note that in this case $V_{\gamma_\infty P}\equiv 0$.

We first claim that $b_3 \leq 0$.
Indeed, from Definition \ref{def:solution}\eqref{eq:conincidence_set_of_u_is_on_the_right} we know that $u-p$ is harmonic in $\{y_3 \leq 0\}$, which combined with \eqref{eq:assumption_that_u-tilde_u_converges_in_lin_scaling_to_lin_function} implies that
\begin{align}
\frac{(u-p)(r_kx)}{r_k} \to b \cdot x \quad \text{ uniformly in } B_{1/2}(-e^3) \text{ as } k \to \infty.
\end{align}
On the other hand, Definition \ref{def:solution}\eqref{eq:conincidence_set_of_u_is_on_the_right}-\eqref{PDE_asymptotics} implies that $(u-p)(-t e^3) \geq 0$
for all $t \geq 0$. Thus 
\begin{align}
0 \leq \frac{(u-p)(- r_k e^3)}{r_k} \to -b_3 \quad \text{ as } k \to \infty,
\end{align} 
proving the claim

Now, if $b_3=0$ then the result follows by choosing $\gamma=0$. Otherwise we note that,
 thanks to \eqref{eq:scaling D3A},
$$
\partial_3 \cA^R_{(\gamma P,\gamma_\infty P)}=\partial_3 A^R_{\gamma P}=\gamma\partial_3A^{R/\gamma}_P.
$$
Moreover, for each $\gamma \in (0,1]$, we know from  Lemma~\ref{lem:DA} and \eqref{eq:p3A log} that
{
$$
-c_\infty \log (R/\gamma) \geq \partial_3A^{R/\gamma}_P \geq -C_\infty\log (R/\gamma) \qquad \text{ for all $R$ sufficiently large}.
$$
}
Thus, recalling that $b_3 < 0$, for each $R$ sufficiently large we can find $\gamma=\gamma_R \in \left(0,\frac{2|b_3|}{c_\infty\log R}\right]$ such that $\gamma_R\partial_3A^{R/\gamma_R}_P=b_3$.
Choosing again $R=r_k$ and $\gamma_k=\gamma_{r_k}$, we conclude as before.
\end{proof}

\subsection{Ordering of solutions and conclusion}

\begin{prop}[Ordering in dimension $N=3$]\label{prop:ordered 3}
{Let $u_\infty$, $\gamma_\infty$, and $P$ be as in \eqref{eq:definition_of_tilde_u}-\eqref{eq:definition_of_tilde_u2}}. 
Then there exists $\tau' \in \R^2$ such that for each $\sigma \in \R$,
	\begin{align}
	\text{ either } \quad  u \leq u_{\gamma_\infty P-(\tau',\sigma)} ~ \text{ in } \R^N \quad \text{ or } \quad u \geq u_{\gamma_\infty P-(\tau',\sigma)}~ \text{ in } \R^N.
	\end{align}
\end{prop}
\begin{proof}
Let $w_{r_k}$ and $w$ be as in Proposition~\ref{prop:ACF-alternative_for_linear_rescaling}. We distinguish between the two cases in the dichotomy.\\
{\bf Case 1}. {\it Proposition~\ref{prop:ACF-alternative_for_linear_rescaling}\eqref{case:w ordered} holds.}\\
In this case we have that either $\norm{(w_{r_k})_+}_{L^1(B_{1})}\to 0$ or $\norm{(w_{r_k})_-}_{L^1(B_{1})}\to 0$ as $k \to \infty.$ 
Also, 
since 
$$
u_{\gamma_\infty P} (x)-u_{\gamma_\infty P-(0,\sigma)} (x)=u_{\gamma_\infty P} (x)-u_{\gamma_\infty P} (x+\sigma e^3)=V_{\gamma_\infty P}(x)-V_{\gamma_\infty P}(x+\sigma e^3),
$$
defining $w_{k,\sigma}(x):=\frac{(u-u_{\gamma_\infty P-(0,\sigma)})(r_k x)}{r_k} $ Lemma~\ref{lem:translate V} yields
$$
\norm{w_k-w_{k,\sigma}}_{L^1(B_1)}=\frac{1}{r_k} \fint \limits_{B_{r_k}}|V_{\gamma_\infty P}(x)-V_{\gamma_\infty P}(x+\sigma e^3)|\dx{x} \leq C\frac{|\sigma|}{r_k^{1/8}} \to 0\qquad \text{as }k \to \infty.
$$
Hence, for each $\sigma \in \R$,  either $\norm{(w_{k,\sigma})_+}_{L^1(B_{1})}\to 0$ or $\norm{(w_{k,\sigma})_-}_{L^1(B_{1})}\to 0$ as $k \to \infty.$ 
Therefore, thanks to Lemma~\ref{lem:ACF}\eqref{item:ACF monotone}-\eqref{item:ACF L2 bound}, for each $\rho \in (0,+\infty)$  it holds
\begin{multline}
0 \leq \Phi\big(u-u_\infty(\cdot+\sigma e^3),\rho\big) \leq \lim_{k\to \infty} \Phi\big(u-u_\infty(\cdot+\sigma e^3), \tfrac{r_k}{4}\big) \\
=  \lim_{k\to \infty} \Phi\big(w_{k,\sigma}, \tfrac{1}{4}\big) 
\leq C \limsup_{k\to \infty}  \norm{(w_{k,\sigma})_+}^2_{L^1(B_{1})} \norm{(w_{k,\sigma})_-}^2_{L^1(B_{1})} =0.
\end{multline}
Applying Lemma~\ref{lem:ACF}\eqref{item:ACF order} proves the result with $\tau'=0$.\\
{\bf Case 2}. {\it Proposition~\ref{prop:ACF-alternative_for_linear_rescaling}\eqref{case:w affine} holds.}\\
Let $\gamma_k$ and $\tau'$ be as in Proposition~\ref{prop:calibration_in_the_case_that_tilde_u_is_paraboloid_solution}, and define
$w_{k,\sigma}'(x):=\frac{(u-u_{\gamma_{k} P - (\tau',\sigma)})(r_k x)}{r_k}$.
Note that, since $\gamma_k \to \gamma_\infty$ and solutions to the obstacle problems are locally bounded in $C^{1,1}$ (cf. Lemma~\ref{lem:compact C1}\eqref{item:C11}), 
\begin{align}
	u_{\gamma_{k} P - (\tau',\sigma)} \to u_{\gamma_\infty P - (\tau',\sigma)}=u_\infty(\cdot + (\tau',\sigma)) \quad \text{ in } C^{1,\alpha}_{\rm loc}(\R^3)\quad \text{ as } k \to \infty,
\end{align}
which implies in particular that, given $\rho \in (0,+\infty)$,
$$
\Phi\big(u-u_{\gamma_{k} P - (\tau',\sigma)} ,\rho\big) \to \Phi\big(u-u_\infty(\cdot+(\tau',\sigma)),\rho\big)\quad \text{ as } k \to \infty.
$$
Since  $\|w_{k,\sigma}'\|_{L^1(B_{1})}\to 0$, it follows from Lemma~\ref{lem:ACF}\eqref{item:ACF monotone}-\eqref{item:ACF L2 bound} that, for each $\rho \in (0,+\infty)$,\begin{align}
0 &\leq \Phi\big(u-u_\infty(\cdot+(\tau',\sigma)),\rho\big) =  \lim_{k\to \infty} \Phi\big(u-u_{\gamma_{k} P - (\tau',\sigma)} ,\rho\big) \leq \limsup_{k\to \infty} \Phi\big(u-u_{\gamma_{k} P - (\tau',\sigma)} , \tfrac{r_k}{4}\big) \\
&=  \limsup_{k\to \infty} \Phi\big(w_{k,\sigma}', \tfrac{1}{4}\big) 
\leq C \limsup_{k\to \infty}  \norm{(w_{k,\sigma}')_+}^2_{L^1(B_{1})} \norm{(w_{k,\sigma}')_-}^2_{L^1(B_{1})} = 0.
\end{align}
Hence, the result follows again from Lemma~\ref{lem:ACF}\eqref{item:ACF order}.
\end{proof}

We can now prove our main result.
\begin{proof}[Proof of Theorem~\ref{thm:main}: the case $N=3$]
The proof is almost identical to the one of Theorem~\ref{thm:main} for $N\geq 4$ given in Section~\ref{sect:proof 4}, the only difference being the application of
Proposition~\ref{prop:ordered 3} instead of Proposition~\ref{prop:translation_in_eN-direction_does_not_affect_ordering}.
\end{proof}

\section{Proof of Theorem \ref{thm:MainTheorem_Intro_I}}\label{sect:proof thm}
As explained after the statement of \cite[Main Theorem**]{esw_arXiv}, every non-cylindrical solution with unbounded coincidence set is $x_N$-monotone. Hence, to prove Theorem~\ref{thm:MainTheorem_Intro_I} for $N\geq 3$, it suffices to characterize $x_N$-monotone solutions for $N\geq 3$, which is exactly the result of Theorem~\ref{thm:main}. For completeness and convenience of the interested reader, we present here an original alternative argument to explain how Theorem \ref{thm:MainTheorem_Intro_I} follows from Theorem~\ref{thm:main}.

\begin{proof}[Proof of Theorem \ref{thm:MainTheorem_Intro_I}]
Let $u$ be a global solution with non-empty coincidence. We can assume that $u$ is non-cylindrical (see Definition~\ref{def:cylindrical}), as otherwise a restriction of $u$ coincides with a non-cylindrical global solution $v$ in some lower dimension and it suffices to prove Theorem~\ref{thm:MainTheorem_Intro_I} for $v$.
Also, as mentioned in the introduction, Theorem \ref{thm:MainTheorem_Intro_I} has already been proved for $N=2$ in \cite{Sakai_Null_quadrature_domains}.

Hence, we assume that $u$ is a non-cylindrical global solution in dimension $N \geq 3$, whose coincidence set 
$\cC:=\{u=0\}$ has non-empty interior (recall that $\cC$ is convex, see Remark~\ref{rem:convexity_of_coincidence_set_of_global_solution}). 

Set $u_r(x):=\frac{u(rx)}{r^2}$, and define $g(x):=\lim_{r\to \infty}u_r(x)$ (cf. Lemma~\ref{lem:C11}). We distinguish several cases.\\
$\bullet$  {\bf Case 1.} {\it $g$ is a half-space solution.}\\
By a translation and a rotation we can assume that $\{u=0\}\subset \{x_N \leq 0\}$ and that $0 \in \partial \{u=0\}$. Then, since $\{u=0\}$ is convex we deduce that $\{u_r=0\}=\frac{1}r\{u=0\} \subset \{u=0\}$ for every $r \geq 1$, and letting $r\to \infty$ we conclude that $\{g=0\}\subset \{u=0\}$.

On the other hand, since $g$ is a half-space solution, $\{g=0\}$ is a half-space passing through the origin.

So the only option is that $\{g=0\}=\{u=0\}=\{x_N \leq 0\}$, from which it follows that 
$\Delta (u-g)\equiv 0$ and $(u-g)|_{\{x_N \leq 0\}}=0$. By unique continuation this implies that $u \equiv g$, and since $g$ is constant in the directions orthogonal to $e^N$ we deduce that $u$ is cylindrical, a contradiction.\\
$\bullet$  {\bf Case 2.} {\it $g(x)=\frac12 x^TQx$ is a quadratic polynomial solution with $Q$ positive definite.}\\
Since $Q$ is positive definite, there exists a constant $c_0>0$ such that
$g\geq c_0$ on  $\partial B_1$. Hence, it follows from the local uniform convergence of $u_r$ to $g$ that $u(x)\geq \frac{c_0}{2}|x|^2$ for sufficiently large $|x|$.
This implies that the coincidence set $\{u=0\}$ is compact, so the result follows from \cite{FriedmanSakai}  (see also \cite{ellipsoid} or Remark~\ref{rem:compact case}).\\
$\bullet$ {\bf Case 3.} {\it $g(x)=\frac12 x^TQx$ is a quadratic polynomial solution with ${\rm ker}(Q) \neq {\{0\}}$.}\\
{\bf Step 1.} {\it $u$ is monotone in the directions of ${\rm ker}(Q)$.}\\
Let $e \in \partial B_1\cap {\rm ker}(Q)$. Then  $\partial_eu=\partial_e(u-g).$ 
Also, if we define $v_r(x):=\frac{(u-g)(rx)}{r^2}=u_r(x)-g(x)$, 
{
H\"older's inequality and Lemma \ref{lem:subharm}\eqref{cacc} imply that
$$ 
\biggl(\fint \limits_{B_1} |\nabla v_r|  \dx{x}\biggr)^2
\leq
\fint \limits_{B_1} |\nabla v_r|^2 \dx{x}
\leq C\fint \limits_{B_{2}} v_r^2 \dx{x}  \to 0 \text{ as } r \to \infty.$$
}
It follows that
\begin{equation}
\label{eq:ACF Deu}
\frac{1}r  \fint \limits_{B_r}|\partial_e u|\dx{x}=\frac{1}r  \fint \limits_{B_r}|\partial_e (u-g)|\dx{x}=\|\partial_e v_r\|_{L^1(B_1)} \to 0\qquad \text{as }r\to \infty.
\end{equation}
Thus, thanks to Remark~\ref{rem:Deu}, \eqref{eq:ACF Deu}, and Lemma~\ref{lem:ACF}\eqref{item:ACF L2 bound}-\eqref{item:ACF order}, we deduce that either $\partial_eu \geq 0$ or $\partial_eu \leq 0.$ Since $\partial_e u \not\equiv 0$ (as $u$ is non-cylindrical) and $\partial_e u$ is harmonic outside the coincidence set of $u$, it follows by the strong maximum principle that
\begin{equation}
\label{eq:strong monotone}
\text{either $\partial_eu > 0$ or $\partial_eu < 0$}\quad \text{ inside }\{u>0\}.
\end{equation}
{\bf Step 2.} {\it ${\rm ker}(Q)$ is one-dimensional.}\\
Indeed, assume by contradiction that there exists a two-dimensional plane $\Pi\subset {\rm ker}(Q)$. Then, by the argument above we deduce that 
$$
\text{for any $e \in \partial B_1\cap \Pi$,}\quad \eqref{eq:strong monotone} \text{ holds}.
$$
Fix a point $\bar x \in \{u>0\}$ and consider a curve $[0,1]\ni s\mapsto e(s)\subset \partial B_1\cap \Pi$ such that $e(0)=-e(1)$. Then, since $\partial_{e(0)}u(\bar x) =-\partial_{e(1)}u(\bar x)$, it follows by continuity that there exists $s \in (0,1)$ such that $\partial_{e(s)}u(\bar x)=0$. This contradicts \eqref{eq:strong monotone} and proves that ${\rm ker}(Q)$ is one-dimensional.\\
{\bf Step 3.} {\it $u$ is $x_N$-monotone.}\\
Since ${\rm ker}(Q)$ is one-dimensional, we may assume that ${\rm ker}(Q)=\R e^N$ and that $\partial_Nu \geq 0$.
Thus, since by assumption $\cC=\{u=0\}$ has non-empty interior,
$u$ satisfies Definition~\ref{def:solution}\eqref{item:cC_has_nonempty_interior}-\eqref{eq:monotone}-\eqref{PDE_asymptotics}.
Also, up to a translation, we can assume that $\cC\subset \{x_N\geq 0\}$ and that $0 \in \partial \cC$, so $u$ is $x_N$-monotone, as desired.
%
%We claim that also Definition~\ref{def:solution}\eqref{eq:conincidence_set_of_u_is_on_the_right} holds. Indeed, if not, by convexity it would follow that $\partial \cC$ contains a segment. But then, since $\partial \cC$ is analytic \cite{}, $\cC$ would contain an infinite line $\R e$,
%and so it follows from \cite[proof of Theorem II, Case 3]{CaffarelliKarpShahgolian_Annals_2000} that $u_0$ is cylindrical, a contradiction.
%
%In conclusion, $u$ is $x_N$-monotone, as desired.
\end{proof}

\bibliographystyle{abbrv}
\bibliography{global-obstacle-0807}
\end{document}